\RequirePackage{fix-cm}
\documentclass[smallextended]{svjour3}       
\smartqed  
\usepackage[letterpaper, margin=1.2in]{geometry}
\spnewtheorem*{conj}{Conjecture}{\bf}{\it}
\usepackage[utf8]{inputenc} 
\usepackage{graphicx}
\usepackage{amsmath,amssymb,bbm,url,amsfonts}
\usepackage{mathtools}
\usepackage{pgf, tikz}
\usetikzlibrary{arrows, automata}

\usepackage{xcolor}
\usepackage[hidelinks]{hyperref}       
\usepackage{url}            
\usepackage{booktabs}       
\usepackage{amsfonts}       
\usepackage{nicefrac}       
\usepackage{microtype}      
\usepackage{algorithm}
\usepackage{subfiles}
\usepackage{bm}
\usepackage[T1]{fontenc}    
\usepackage{url}            
\usepackage{algorithm}
\usepackage{algpseudocode}
\usepackage{caption}
\usepackage{subcaption}
\usepackage{tikz}
\usetikzlibrary{arrows.meta,positioning,fit}
\newcounter{subroutine}

\makeatletter
\newenvironment{subroutine}[1][]{
  \refstepcounter{subroutine}
  \floatname{algorithm}{Subroutine}
  
  \begin{algorithm}[#1]
}{
  \end{algorithm}
}
\makeatother

\DeclareMathOperator*{\argmin}{argmin}

\global\long\def\prox{\mathbf{prox}}%




\newcommand{\R}{\mathbb{R}}

\newcommand{\ALiA}{\hyperref[alg:A1]{ALiA}}

\let\emph\textit

\begin{document}

\title{ALiA: \textbf{A}daptive \textbf{Li}nearized \textbf{A}DMM
}

\titlerunning{ALiA: Adaptive Linearized ADMM}        

\author{
\mbox{Uijeong Jang \and Kaizhao Sun \and Wotao Yin \and Ernest K.\ Ryu }}

\authorrunning{Jang, Sun, Yin, Ryu} 

\institute{Uijeong Jang \at
              Department of Mathematics, University of California, Los Angeles\\
              \email{uijeongjang@math.ucla.edu} 
           \and
Kaizhao~Sun \at
             DAMO Academy, Alibaba Group (U.S.) Inc.\\
              \email{kaizhao.s@alibaba-inc.com} 
                \and
Wotao~Yin\at
            DAMO Academy, Alibaba Group (U.S.) Inc. \\
              \email{wotao.yin@alibaba-inc.com} 
                \and
              Ernest K. Ryu \at
              Department of Mathematics, University of California, Los Angeles\\
              \email{eryu@math.ucla.edu} 
}        

\date{Received: date / Accepted: date}

\makeatletter

\renewcommand{\paragraph}{%
  \@startsection{paragraph}{4} {\z@}
  {1ex \@plus 1ex \@minus .2ex}
  {-1ex}%
  {\bfseries}%
}                                  
 
\makeatother

\maketitle

\begin{abstract}
We propose ALiA, a novel adaptive variant of the alternating direction method of multipliers (ADMM). Specifically, ALiA is a variant of function-linearized proximal ADMM (FLiP ADMM), which generalizes the classical ADMM by leveraging the differentiable structure of the objective function, making it highly versatile. Notably, ALiA features an adaptive stepsize selection scheme that eliminates the need for backtracking linesearch. Motivated by recent advances in adaptive gradient and proximal methods, we establish point convergence of ALiA for convex and differentiable objectives. Furthermore, by introducing negligible computational overhead, we develop an alternative stepsize selection scheme for ALiA that improves the convergence speed both theoretically and empirically. Extensive numerical experiments on practical datasets confirm the accelerated performance of ALiA compared to standard FLiP ADMM. Additionally, we demonstrate that ALiA either outperforms or matches the practical performance of existing adaptive methods across problem classes where it is applicable. 
\end{abstract}

\section{Introduction}
Primal-dual methods explicitly maintain and update both primal and dual variables. Among these, the alternating direction method of multipliers (ADMM) is particularly prominent, with a long history of theoretical development and a wide range of applications. In recent years, a growing body of work has studied a collection of variants loosely referred to as \emph{linearized ADMM} (discussed further in Section~\ref{ss:admm-prior}), which offer additional flexibility that can be leveraged to exploit a wider range of problem structures.

Simultaneously, adaptivity has recently emerged as a key property of interest in first-order optimization. Adaptive algorithms dynamically adjust algorithm parameters without requiring prior knowledge of problem-specific constants, such as the smoothness parameter or the initial distance to the solution. Specifically, there has been active research on adaptive gradient and proximal methods~\cite{nesterov2013gradient,malitsky2020adaptive,yang2018modified,malitsky2020golden,li2023simple,malitsky2024adaptive}, but the topic of adaptive primal-dual methods has received comparatively less attention, despite their ability to solve a broader class of constrained optimization problems.

In this paper, we propose \textbf{A}daptive \textbf{Li}nearized \textbf{A}DMM (ALiA) to solve the following general-form optimization problem:
\begin{equation}\label{eq:linadmm}
\begin{array}{ll}
\underset{x\in \mathbb{R}^p, y\in \mathbb{R}^q}{\mbox{minimize}}
&f_1(x)+f_2(x)+g_1(y)+g_2(y)\\
{\mbox{subject to}}
&Ax+By=c,
\end{array}
\end{equation}
where $A\in\mathbb{R}^{r\times p}$, $B\in\mathbb{R}^{r\times q}$, $c\in \mathbb{R}^r$, $f_2\colon \mathbb{R}^p\rightarrow\mathbb{R}$ and $g_2\colon \mathbb{R}^q\rightarrow\mathbb{R}$ are differentiable and convex, and
$f_1\colon \mathbb{R}^p\rightarrow\mathbb{R}\cup\{+\infty\}$ and
$g_1\colon \mathbb{R}^q\rightarrow\mathbb{R}\cup\{+\infty\}$ are closed, convex, and proper.
\begin{algorithm}[H]
				\caption{ALiA}
				\label{alg:A1}
				\begin{algorithmic}[1]
				\itemsep=3pt%
				\renewcommand{\algorithmicindent}{0.3cm}%
				\Require
					\begin{tabular}[t]{@{}l@{}}
				  Initial stepsize $\gamma_0>0$ \\
                       Primal-dual stepsize ratio parameter $\sigma>0$\\
                        Initial primal-dual points $(x^{0},y^{0},u^{0})=(x^{-1},y^{-1},u^{-1})$
					\end{tabular}
                
				\item[{for} \(k=0,1,\ldots\)]
                                        
            \State Update stepsize $\gamma_{k+1}>0$ and direction $\Delta u^{k+1}$ using Subroutine~\ref{alg:S1} or Subroutine~\ref{alg:S2} 

                \State $u^{k+1} = u^{k} + \sigma\gamma_{k+1}\Delta u^{k+1}$ {\color{gray} {\# Dual update }}
                
                   \State     $x^{k+1}= \argmin_{x\in \mathbb{R}^p}\left\{f_1(x)+\langle  A^\top u^{k+1} + \nabla f_2(x^k), x\rangle +\frac{1}{2\gamma_{k+1}}\|x-x^{k}\|^2 \right\}$ 
 {\color{gray}{\# Primal $x$-update}}
                \State  $y^{k+1}=\argmin_{y\in \mathbb{R}^q}\left\{g_1(y)+\langle  B^\top u^{k+1} + \nabla g_2(y^k), y\rangle + \frac{1}{2\gamma_{k+1}}\|y-y^{k}\|^2 \right\}$  {\color{gray}{\#  Primal $y$-update}}
                 \end{algorithmic}
\end{algorithm}
In the description of ALiA, $\langle \cdot,\cdot\rangle$ denotes the Euclidean inner product and $\|\cdot\|$ the associated Euclidean norm. The stepsize $\gamma_{k+1}$ and the dual update direction $\Delta u^{k+1}$ are computed by adaptive subroutines, which adaptively set $\gamma_{k+1}$ and $\Delta u^{k+1}$ without relying on unknown problem-instance parameters and without incurring significant computational cost. We fully describe the subroutines in Section~\ref{sec:2}.

The formulation \eqref{eq:linadmm} provides the user with the flexibility to choose which components of the objective are accessed via gradients (i.e., linearized) and which are accessed through the minimization. In most applications, only a subset of the functions $f_1$, $f_2$, $g_1$, and $g_2$ is used, with the remaining functions set to zero. However, ALiA does allow all four functions to be used as non-zero functions.

\subsection{Preliminaries and notations}
We say that a function $f\colon \mathbb{R}^d \to \mathbb{R}\cup\{+\infty\}$ is \emph{convex} if
\[
f(\theta x + (1-\theta)y) \le \theta f(x) + (1-\theta)f(y),
\qquad \forall\, x,y \in \mathbb{R}^d,\, \theta \in (0,1).
\]
A function $f$ is called \emph{closed, convex, and proper (CCP)} if it is convex, satisfies
$f(x) < +\infty$ for some $x \in \mathbb{R}^d$ and $f(x) > -\infty$ for all $x \in \mathbb{R}^d$,
and has a closed epigraph
$\{(x,\alpha) \in \mathbb{R}^d \times \mathbb{R} : f(x) \le \alpha\}
\subset \mathbb{R}^{d+1}$. The \emph{subdifferential} of $f$ at $x$ is defined as
\[
\partial f(x)
:=\big\{ g\in\mathbb{R}^d :
f(y)\ge f(x)+\langle g,y-x\rangle,\ \forall y\in\mathbb{R}^d\big\}.
\]
Any vector $g\in\partial f(x)$ is called a \emph{subgradient} of $f$ at $x$. If $f$ is differentiable at $x$, then $\partial f(x)=\{\nabla f(x)\}$.
For a CCP function $f$, $\partial f(x)$ is nonempty for all
$x$ in the interior of the domain of $f$~\cite[Theorem~23.4]{rockafellar1970convex}, and a point $x$ is a global minimizer of $f$
if and only if $0 \in \partial f(x)$~\cite[Chapter~27]{rockafellar1970convex}.

A \emph{proximal operator} for $f\colon \mathbb{R}^d \to \mathbb{R}\cup\{+\infty\}$ is defined as
\[
\prox_{\alpha f}(x)=\argmin_{z\in \mathbb{R}^d}\big\{ f(z)+\tfrac{1}{2\alpha}\|z-x\|^2\big\}
\]
for any $\alpha>0$. Proximal operators are well-defined for CCP functions, i.e., the argmin uniquely exists when $f$ is CCP. Using the notation of proximal operators, \ALiA \  can be described more compactly as:
\begin{align}    
u^{k+1} &= u^{k} + \sigma\gamma_{k+1}\Delta u^{k+1}\nonumber\\
x^{k+1} &= \prox_{\gamma_{k+1}f_1}\big(x^k-\gamma_{k+1}\nabla f_2(x^k)-\gamma_{k+1}A^\top u^{k+1}\big)
\label{eq:ALiA-prox}\\
y^{k+1} &= \prox_{\gamma_{k+1}g_1}\big(y^k-\gamma_{k+1}\nabla g_2(y^k)-\gamma_{k+1}B^\top u^{k+1}\big).\nonumber
\end{align}

We say a differentiable $f\colon \mathbb{R}^d\rightarrow\mathbb{R}$ is \emph{(globally) smooth} if its gradient is Lipschitz continuous. In other words, there exists $L>0$ such that 
\[
\|\nabla f(x) -\nabla f(y) \|\le L\|x-y\|,
\qquad\forall\, x,y\in \mathbb{R}^d.
\]
We say a differentiable $f\colon \mathbb{R}^d\rightarrow\mathbb{R}$ is \emph{locally smooth} if for every compact $C\subset\mathbb{R}^d$, there exists $L_C>0$ such that
\[
\|\nabla f(x) -\nabla f(y) \|\le L_C\|x-y\|,\qquad\forall\,x,y\in C.
\]
A convex function such as $f(x)=x^4$ is locally but not globally smooth.

Classical analyses of non-adaptive first-order methods,
including gradient descent~\cite{nesterov2018lectures} and function-linearized proximal (FLiP)
ADMM~\cite{ryu2022large}, typically assume global smoothness. In contrast,
recent work on adaptive first-order methods requires only local smoothness
\cite{latafat2024adaptive,malitsky2020adaptive,malitsky2024adaptive}.

\subsection{Classical and linearized ADMMs}
\label{ss:admm-prior}
The alternating direction method of multipliers (ADMM) and its variants are widely used optimization algorithms with a broad range of applications. In this subsection, we briefly review the classical ADMM and its variants most relevant to ALiA.

\paragraph{Classical ADMM.}
The classical ADMM addresses the optimization problem
\[
\begin{array}{ll}
\underset{x\in \mathbb{R}^p, \,y\in \mathbb{R}^q}{\mbox{minimize}}
&f(x)+g(y)\\
{\mbox{subject to}}
&Ax+By=c,
\end{array}
\] 
where $A\in\mathbb{R}^{r\times p}$, $B\in\mathbb{R}^{r\times q}$, $c\in\mathbb{R}^r$, and $f\colon \mathbb{R}^p\rightarrow\mathbb{R}\cup\{+\infty\}$ is CCP, and
$g\colon \mathbb{R}^q\rightarrow\mathbb{R}\cup\{+\infty\}$ is CCP. This problem is a special case of \eqref{eq:linadmm}, the problem addressed by ALiA, with $f_2 = 0$ and $g_2 = 0$.
The classical ADMM \cite{GlowinskiMarroco1975_lapproximation,GabayMercier1976_dual,fortin1983chapter,boyd2011distributed} has the form:
 \begin{align*}
      x^{k+1}&\in \argmin_{x\in\mathbb{R}^p}\mathbf{L}_{\sigma}(x, y^k, u^k)\\
        y^{k+1}&\in \argmin_{y\in\mathbb{R}^q}\mathbf{L}_{\sigma}(x^{k+1}, y, u^k)\\
u^{k+1}&=u^k+\sigma(Ax^{k+1}+By^{k+1}-c),
\end{align*}
where
\[
\mathbf{L}_{\sigma}(x,y,u)=f(x)+g(y)+\langle u , Ax+By-c\rangle+\frac{\sigma}{2}\|Ax+By-c\|^{2}
\]
is the augmented Lagrangian, and $\sigma>0$ is a penalty parameter. 

ADMM is most naturally described and understood from a primal-dual perspective. For this, we consider (unaugmented) Lagrangian
\[
\mathbf{L}(x,y,u) := f(x) + g(y) + \langle u, Ax + By - c \rangle,
\]
where $u \in \mathbb{R}^r$ denotes the dual variable, and the dual optimization problem
\[
\begin{array}{ll}
\underset{u \in \mathbb{R}^r}{\mbox{maximize}}
&
- f^{\ast}(-A^{\top}u) - g^{\ast}(-B^{\top}u) - \langle c, u \rangle,
\end{array}
\]
where $f^{\ast}$ and $g^{\ast}$ denote the convex conjugates of $f$ and $g$, defined by $f^{\ast}(y) = \sup_{x}\{\langle x,y\rangle - f(x)\}$, and analogously for $g^{\ast}$. Analyses of ADMM typically assume the existence (but not the uniqueness) of a saddle point of $\mathbf{L}$, defined as a triplet $(x^\star, y^\star, u^\star)$ satisfying
\[
\mathbf{L}(x^\star, y^\star, u)
\le \mathbf{L}(x^\star, y^\star, u^\star)
\le \mathbf{L}(x, y, u^\star),
\qquad
\forall\,  x\in \mathbb{R}^p,\,y\in \mathbb{R}^q,\,u\in \mathbb{R}^r.
\]
The existence of a saddle point is equivalent to the simultaneous satisfaction of the following three conditions: the existence of a primal solution, the existence of a dual solution, and strong duality \cite[Section~36]{rockafellar1970convex}. Also, the Lagrangian $\mathbf{L}$ and the augmented Lagrangian $\mathbf{L}_{\sigma}$ share the same set of saddle points. A saddle point $(x^\star, y^\star, u^\star)$ corresponds to both a primal solution $(x^\star, y^\star)$
and a dual solution $u^\star$. ADMM can be interpreted as a primal-dual algorithm for computing saddle points of $\mathbf{L}$: the first two steps perform alternating
primal updates of the $x$- and $y$-variables, followed by a dual ascent step that updates the dual variable $u$.

\paragraph{ADMM variants.}
Beyond the classical algorithm, a substantial body of work has focused on modifying ADMM to improve the tractability of its subproblems.
Among these, \emph{proximal ADMM} \cite{deng2016global,yang2022proximal} and \emph{(function) linearized ADMM} \cite{lin2017extragradient,gao2018information,liu2019linearized} are particularly relevant to our setting, as they replace difficult subproblems with proximal or first-order approximations that can be solved efficiently.

Other ADMM variants have also been proposed to address different challenges, including acceleration, stochastic updates, and symmetrization \cite{goldstein2014fast,ouyang2015accelerated,he2016convergence,ouyang2013stochastic}.
While these methods are important in their own right, we focus here on the linearized and proximal framework, as it provides the necessary flexibility to handle the general problem class~\eqref{eq:linadmm}.

\paragraph{Linearized proximal ADMM.}
Proximal ADMM was introduced as a modification of classical ADMM to improve the tractability of the primal subproblems by adding proximal regularization terms as follows:
 \begin{equation*}
 x^{k+1}\in \argmin_{x}\Big\{ f(x)+ \langle  A^\top u^{k} , x\rangle + \frac{\sigma}{2}\|Ax +By^k-c\|^2 +\frac{1}{2}\|x-x^k\|_{P}^2 \Big\},
 \end{equation*}
where $\sigma>0$, $P$ is positive-semidefinite, and $\|x\|_P := \sqrt{x^\top P x}$. The semidefinite matrix $P$ is often chosen as $P=I-\sigma A^\top A$. Then, we obtain
 \begin{align*}
\frac12\|x-x^k\|_P^2+\frac{\sigma}{2}\|Ax+By^k-c\|^2
&=\frac12\|x-x^k\|^2
-\frac{\sigma}{2}(x-x^k)^\top A^\top A(x-x^k)\\
&\quad+\frac{\sigma}{2}x^\top A^\top A x
+\sigma\langle A^\top(By^k-c),x\rangle
+\mathrm{const} \\
&=\frac12\|x-x^k\|^2
+\sigma\big\langle A^\top(Ax^k+By^k-c),x\big\rangle
+\mathrm{const}.
\end{align*}
Therefore, the $x$-update can be equivalently written as
\begin{align*}
x^{k+1}
&\in \argmin_x \Big\{f(x)+\langle  A^\top u^{k} , x\rangle+\frac12\|x-x^k\|^2+\sigma\big\langle A^\top(Ax^k+By^k-c),x\big\rangle
\Big\} \\
&=\argmin_x \Big\{f(x)+\langle  A^\top u^{k} , x\rangle+\frac12\Big\|x-\big(x^k-\sigma A^\top(Ax^k+By^k-c)-A^\top u^k\big)\Big\|^2\Big\}
\\
&=\prox_{f}\Big(x^k-\sigma A^\top(Ax^k+By^k-c)-A^\top u^k\Big).
\end{align*}
Consequently, the dependence on $Ax$ from the augmented Lagrangian term 
$\frac{\sigma}{2}\|Ax+By^k-c\|^2$ reduces to its first-order Taylor expansion 
$\sigma\langle A^\top(Ax^k+By^k-c),x\rangle$, making it easily tractable for 
``proximable'' functions and reducing the per-iteration cost.

 \paragraph{Function linearized ADMM.}
Subsequent developments further explored
linearization strategies based on first-order approximations of the objective function. Specifically, one of the subroutine updates is typically equivalent to the following: 
 \begin{equation*}
 x^{k+1}\in \argmin\left\{\langle  A^\top u^{k} + \nabla f(x^k), x\rangle + \frac{\sigma}{2}\|Ax +By^k-c\|^2 +\frac{1}{2}\|x-x^k\|_{P}^2 \right\}
 \end{equation*}
 where $f$ is accessed through its first-order approximation at $k$-th iterate $x^k$. Thus, both proximal ADMM and (function) linearized ADMM are commonly referred to as the umbrella term linearized ADMM \cite{melo2017iteration,yashtini2022convergence}. Convergence of linearized ADMM can be established under suitable conditions (e.g., convexity and Lipschitz differentiability of the objective) \cite{he20121,lu2021linearized}, and in some nonconvex cases, it converges to stationary points under additional assumptions \cite{liu2019linearized}. 
 
 The convergence analysis of linearized ADMM with acceleration is also an active research area, and a large body of prior work investigates the accelerated convergence rates with a suitable performance measure \cite{ouyang2015accelerated,lu2016fast,xu2017accelerated,li2019accelerated,zeng2024accelerated,liu2025accelerated}. A limitation, however, is that traditional linearized ADMM assumes the objective function $f$ or $g$ is smooth (differentiable) as a whole, or it requires one of $f$ and $g$ to be nonsmooth so that the other can be linearized. 

 \paragraph{Function-linearized proximal ADMM.}
To address this problem, a highly generalized version of ADMM that incorporates these individual technical components has been introduced in \cite{gao2019randomized} as RPDBU (\textbf{R}andomized \textbf{P}rimal-\textbf{D}ual \textbf{B}ock Coordinate \textbf{U}pdate Method) and in \cite{ryu2022large} as FLiP ADMM (\textbf{F}unction \textbf{Li}nearized \textbf{P}roximable ADMM) without the randomized block-selection component present in \cite{gao2019randomized}. The specific updates of the latter are as follows:
\begin{align*}
    x^{k+1} &\in \arg\min_{x\in\mathbb{R}^p} \left\{ f_1(x) + \left\langle \nabla f_2(x^k) + A^\top u^k, x \right\rangle + \frac{\sigma}{2}\|Ax + By^k - c\|^2 + \frac{1}{2}\|x - x^k\|_P^2 \right\}, \\
    y^{k+1} &\in \arg\min_{y\in\mathbb{R}^q} \left\{ g_1(y) + \left\langle \nabla g_2(y^k) + B^\top u^k, y \right\rangle + \frac{\sigma}{2}\|Ax^{k+1} + By - c\|^2 + \frac{1}{2}\|y - y^k\|_Q^2 \right\}, \\
    u^{k+1} &= u^k + \varphi \sigma (Ax^{k+1} + By^{k+1} - c).
\end{align*}
Here, $P, Q \succeq 0$ are positive semidefinite matrices, and the stepsize $\sigma > 0$ and the dual extrapolation parameter $\varphi > 0$ are predetermined penalty terms. This provides significantly greater flexibility in solving \eqref{eq:linadmm}, allowing the user to more effectively utilize the four individual structure present in $f$, $g$, $A$, and $B$.
We also note that \ALiA\ differs from the classical ADMM in its update order: the dual variable $u^{k+1}$ is updated first, after which the primal variables $x^{k+1}$ and $y^{k+1}$ can be computed in parallel.

\paragraph{Related primal-dual methods.}
The optimization problem \eqref{eq:linadmm} is sufficiently general to model a wide range of optimization problems of interest. In particular, it can be seen as a generalization of the following three-term splitting problem:
\begin{equation}\label{eq:condatvu}
   \begin{array}{ll}
\underset{x\in \mathbb{R}^p}{\mbox{minimize}}
&f(x)+g(x)+h(Ax),
\end{array} 
\end{equation}
where $f$ is convex differentiable, and $g$ and $h$ are CCP. Indeed, \eqref{eq:condatvu} can be equivalently written as
\[
\underset{x\in\mathbb{R}^p, y\in\mathbb{R}^q}{\mbox{minimize}} \,\, f(x)+g(x)+h(y)
\quad\text{subject to}\quad Ax-y=0,
\]
which is a special case of \eqref{eq:linadmm} with $B=-I$, $c=0$, $f_1=f$, $f_2=g$, $g_1=h$, and $g_2\equiv 0$.

Within this specialization, FLiP-ADMM with appropriate parameters reduces to classical primal-dual splitting schemes \cite[Section~8.2]{ryu2022large}, including the Condat--Vu method \cite{condat2013primal,vu2013splitting}:
\begin{align*}
x^{k+1}
&= \prox_{\alpha g}
\bigl(x^{k}- \alpha A^{\top} u^{k}- \alpha \nabla f(x^{k})\bigr),\\
u^{k+1}
&= \prox_{\beta h^{\ast}}
\bigl(u^{k}+ \beta A(2x^{k+1} - x^{k})\bigr),
\end{align*}
with suitable stepsizes $\alpha,\beta>0$.

Moreover, since Condat--Vu and \ALiA \ share the same computational complexity; both requires gradient evaluations of the smooth term, proximal mappings of the nonsmooth terms, and applications of the linear operator and its adjoint. Therefore in this work, we will compare \ALiA\ with the Condat--Vu method as a primary baseline in our numerical experiments.

\subsection{Related works}

\paragraph{Adaptive first-order methods.}
Although the linearized and accelerated methods mentioned in previous paragraphs enjoy rigorous worst-case iteration bounds, those guarantees are intrinsically conservative and can significantly underestimate real-world problem parameters, such as a Lipschitz constant or an upper bound on the distance to the solution. In practice, however, the local geometry around the current iterate is often significantly different from the global assumptions used in worst-case analysis, meaning that the \emph{safe} stepsize can be much larger than what conservative bounds would suggest. This theory-practice gap has led to research on adaptive schemes that estimate problem parameters from past iterates. These schemes frequently achieve faster empirical convergence with provable theoretical guarantees under mild additional assumptions.

This line of research started with linesearch backtracking technique \cite{armijo1966minimization} in minimizing a single function $f$. Linesearch backtracking adapts to an unknown smoothness constant $L$ by repeatedly testing trial stepsizes $\alpha_k = 1/L_k$:
e.g.,
\[
x^{k+1} = x^k -\alpha_k \nabla f(x^k)
\]
which typically requires additional function/gradient evaluations per iteration.

Recently, groundbreaking work \cite{malitsky2020golden} started the study of adaptive optimization methods that rely solely on gradient oracle of a convex function, avoiding the need for backtracking linesearch. This work yields a simple two-step scheme with guaranteed convergence and serves as a foundation for later developments in linesearch-free adaptive methods:
\[
\bar z^k = \frac{(\varphi-1)z^k+\bar z^{k-1}}{\varphi},
\qquad
z^{k+1} = \prox_{\lambda_k g}\!\left(\bar z^k-\lambda_k F(z^k)\right),
\]
where $\varphi=\tfrac{1+\sqrt{5}}{2}$ and $\lambda_k$ is chosen adaptively from previous iterates. See also \cite{malitsky2020adaptive,li2023simple,suh2025adaptive} for more recent advances in adaptive methods for unconstrained single-function convex optimization.

Building on this idea, researchers have extended adaptivity to proximal and primal–dual methods \cite{malitsky2024adaptive,vladarean2021first,chang2022golden}. More relevant to our setting, adaptive primal–dual methods were also proposed in \cite{malitsky2018first,latafat2024adaptive} to solve \eqref{eq:condatvu}, as an adaptive variant of the Condat--Vu splitting method which extends upon PDHG. The update takes the form
\begin{align*}
y^{k+1}
&=\prox_{\gamma_{k+1} h^{\ast}}\big(y^k+ \gamma_{k+1}\big(
\big(1+\frac{\gamma_{k+1}}{\gamma_k}\big) A x^k
- \frac{\gamma_{k+1}}{\gamma_k} A x^{k-1}\big)\big),\\
x^{k+1}
&=\prox_{\gamma_{k+1} g}\big(x^k-\gamma_{k+1} \nabla f(x^k)- \gamma_{k+1} A^{\top} y^{k+1}\big),
\end{align*}
where the stepsize $\gamma_{k+1}$ is adaptively updated.

However, although these prior works on primal-dual adaptive methods are linesearch-free, they still depend on another global problem parameter--the norm of the matrix $A$. This reveals an additional layer of adaptivity that remains unaddressed. Moreover, a resolution to this issue in \cite{latafat2024adaptive} reintroduces a form of backtracking linesearch in the process, thereby highlighting a challenge in achieving full adaptivity without linesearch.

\paragraph{Adaptive ADMM-type methods.}
Apart from these, several attempts have been made to discover adaptive ADMM-type methods as well. In the \emph{golden ratio ADMM} \cite{fortin1983chapter}, dual extrapolation parameter $\varphi \in(0,\tfrac{1+\sqrt{5}}{2})$ was to balance the primal and dual updates:
 \begin{align*}
      x^{k+1}&\in \argmin_{x\in\mathbb{R}^p}\mathbf{L}_{\sigma}(x, y^k, u^k)\\
        y^{k+1}&\in \argmin_{x\in\mathbb{R}^q}\mathbf{L}_{\sigma}(x^{k+1}, y, u^k)\\
u^{k+1}&=u^k+\sigma\varphi(Ax^{k+1}+By^{k+1}-c).
\end{align*}
Unlike the classical gradient descent method, it is worth noting that classical ADMM is guaranteed to converge for any fixed penalty parameter $\sigma>0$. But to improve practical performance, residual balancing --- a simple tactic for classical ADMM --- was originally suggested by \cite{he2000alternating}. 
 This technique monitors the primal and dual residuals at each iteration and adjusts the penalty parameter $\sigma$. This work inspired later adaptive schemes like ACADMM \cite{xu2017adaptive2} and ARADMM \cite{xu2017adaptive3}, where the latter was also inspired by Barzilai--Borwein stepsizes \cite{barzilai1988two}. However, residual balancing does not come with a formal guarantee of convergence rate in general. 

To date, a fully adaptive selection rule for $\sigma$ with a global convergence guarantee in the general ADMM setting is still unavailable. Accordingly, in this work, we also treat $\sigma$ as a user-specified hyperparameter. However, we emphasize that
all remaining algorithmic parameters of \ALiA\ are selected adaptively from iterate-dependent quantities, while preserving global convergence guarantees for the broad problem class~\eqref{eq:linadmm}.

More recently, researchers have explored the adaptivity of ADMM in broader contexts. In a method introduced in \cite{wang2024adaptive}, the linearization matrix $P$ in proximal ADMM was chosen dynamically based on the current iterate, with a theoretical guarantee of global convergence for convex objectives. Additionally, an adaptive proximal ADMM for weakly convex problems was also proposed in \cite{maia2024adaptive}. However, these methods also have a backtracking procedure.

It is important to distinguish the notion of ``adaptivity'' from the linearized ADMM in the previous literature. The term \emph{adaptive} often
refers to the use of extrapolation or anchoring strategies based on a
predefined schedule or problem-dependent constants \cite{ouyang2015accelerated,lu2016fast,sun2025accelerating,liu2025accelerated}, but these methods do not choose stepsizes based on past iterates. Likewise, \cite{he2023accelerated} considers the exact setting of \eqref{eq:linadmm} and attains an accelerated convergence rate under strongly convex assumptions; however, the stepsizes do not depend on current or past iterates. By contrast, our method derives its adaptive updates directly from
iterate-dependent quantities that is orthogonal to classical acceleration techniques.

Very recently, a linesearch-free adaptive ADMM that selects stepsizes based on past and current iterates was proposed \cite{lan2024auto} with an accelerated convergence rate, provided that the hyperparameters are carefully chosen. However, its experimental validation is insufficient, and it only linearizes a single function and matrix component, thereby restricting itself to a strict subclass of \eqref{eq:linadmm}. 

\subsection{Contributions and organization}
In this paper, we propose ALiA, the first adaptive primal-dual method capable of solving the broad FLiP-ADMM problem class \eqref{eq:linadmm}, and establish its global convergence guarantee. Notably, ALiA achieves this adaptivity without relying on any backtracking linesearch. Finally, we present extensive numerical experiments on real-world problem instances, demonstrating that ALiA consistently outperforms existing methods in practice.

The remainder of the paper is organized as follows. In Section~\ref{sec:2}, we introduce the subroutines used in ALiA, state the main convergence results, and provide a high-level overview of the proof techniques. In Section~\ref{sec:3}, we present the full convergence proofs. In Section~\ref{sec:4}, we present extensive experimental results. Finally, Section~\ref{sec:5} concludes the paper.

\section{Adaptive FLiP ADMM}\label{sec:2}
Although the high-level structure of Adaptive Linearized ADMM (ALiA), as shown in Algorithm~\ref{alg:A1} and \eqref{eq:ALiA-prox}, is simple, the subroutines used to
compute the sequences $\{\gamma_k\}_{k=1,2,\dots}$ and
$\{\Delta u_k\}_{k=1,2,\dots}$ do carry some complexity. In this
section, we describe Subroutines~\ref{alg:S1} and~\ref{alg:S2} and state their corresponding convergence results.

\subsection{Simpler adaptive stepsize and dual update selection subroutine}

We start by describing Subroutine~\ref{alg:S1}, our first adaptive stepsize and dual update selection scheme for ALiA:
\begin{subroutine}[H]
\normalsize
				\caption{
                Adaptive stepsize and dual update selection scheme for \ALiA\ }
				\label{alg:S1}
				\begin{algorithmic}[1]
				\itemsep=3pt%
				\renewcommand{\algorithmicindent}{0.3cm}%
				\Require
					\begin{tabular}[t]{@{}l@{}}
				  Previous stepsize $\gamma_k>0$. \\
                    Primal-dual stepsize ratio parameter $\sigma>0$\\
                    Current and previous primal-dual pairs $(x^{k},y^{k},u^{k})$ and $(x^{k-1},y^{k-1},u^{k-1})$\\
                    Strictly positive $\varepsilon$ close to zero and $0<\varepsilon < \min\{ \frac{1}{2}, \frac{1}{4\sigma}\}$.\\
                      
					\end{tabular}
\begin{align*}
&\Delta u^{k+1}
= Ax^k + By^k - c + 2A(x^k-x^{k-1})+2B(y^k-y^{k-1}),\\[0.1in]
&a_{k+1}
= \frac{\|A^\top\Delta u^{k+1}\|}{\| \Delta u^{k+1}\|},
\qquad
b_{k+1}
= \frac{\|B^\top\Delta u^{k+1}\|}{\| \Delta u^{k+1}\|},\\[0.1in]
&\lambda^{A}_{k+1}
= \frac{\langle A^\top \Delta u^{k+1},\, x^{k}- x^{k-1} \rangle}
{\frac{\| A^\top \Delta u^{k+1}\|^2}{16a_{k+1}^2}+4a_{k+1}^2\| x^{k}- x^{k-1}\|^2},\quad
\lambda^{B}_{k+1}
= \frac{\langle B^\top \Delta u^{k+1},\, y^{k}- y^{k-1} \rangle}
{\frac{\| B^\top \Delta u^{k+1}\|^2}{16b_{k+1}^2}+4b_{k+1}^2\| y^{k}- y^{k-1}\|^2},\\[0.1in]
&\ell_{x,k}
= \frac{\langle \nabla f_2(x^{k-1})-\nabla f_2(x^{k}),\, x^{k-1}-x^k \rangle}{\|x^{k-1}-x^k\|^2},
\qquad
L_{x,k}
= \frac{\|\nabla f_2(x^{k-1})-\nabla f_2(x^k)\|}{\|x^{k-1}-x^k\|},\\
&\ell_{y,k}
= \frac{\langle \nabla g_2(y^{k-1})-\nabla g_2(y^{k}),\, y^{k-1}-y^k \rangle}{\|y^{k-1}-y^k\|^2},
\qquad
L_{y,k}
= \frac{\|\nabla g_2(y^{k-1})-\nabla g_2(y^k)\|}{\|y^{k-1}-y^k\|},\\[0.1in]
&\text{(use convention $0/0=0$ for $a_{k+1},b_{k+1},\lambda^{A}_{k+1},\lambda^{B}_{k+1},\ell_{x,k},L_{x,k},\ell_{y,k},L_{y,k}$)},\\[0.1in]
&\delta_{x,k}
= \gamma_k^2 L_{x,k}^2 - 2\gamma_k \ell_{x,k},
\qquad
\delta_{y,k}
= \gamma_k^2 L_{y,k}^2 - 2\gamma_k \ell_{y,k},\\[0.1in]
&\Gamma_x=\begin{cases}
\frac{1-2\varepsilon}{2}\cdot
\frac{\gamma_k}{
\gamma_k\ell_{x,k}+
\sqrt{(\gamma_k\ell_{x,k})^2 +\frac{2-4\varepsilon}{3}\Bigl(\delta_{x,k} + 6\sigma a_{k+1}^2\gamma_k^2\lambda_{k+1}^{A}\Bigr)}
}
& \text{if the square root is real-valued},\\[0.1in]
+\infty & \text{otherwise},
\end{cases}\\[0.1in]
&\Gamma_y=\begin{cases}
\frac{1-2\varepsilon}{2}\cdot
\frac{\gamma_k}{
\gamma_k\ell_{y,k}+
\sqrt{(\gamma_k\ell_{y,k})^2 +\frac{2-4\varepsilon}{3}\Bigl(\delta_{y,k} + 6\sigma b_{k+1}^2\gamma_k^2\lambda_{k+1}^{B}\Bigr)}
}
& \text{if the square root is real-valued},\\[3mm]
+\infty & \text{otherwise},
\end{cases}\\[2mm]
&\gamma_{k+1}
= \min\biggl\{
\frac{3}{2}\gamma_{k},\ 
\sqrt{\frac{4-\lambda^{A}_{k+1}-\lambda^{B}_{k+1}-8\sigma\varepsilon}{32\sigma(a_{k+1}^2+b^2_{k+1})}},\ 
\Gamma_x,\ 
\Gamma_y
\biggr\}.
\end{align*}
                 \!\!\!\!\!\!\!\!\!\!\!\!\!\! \Return next stepsize $\gamma_{k+1}$ and next update direction $\Delta u^{k+1}$.
                 \end{algorithmic}
\end{subroutine}

Adaptive methods typically estimate the local curvature of the functions using recent iterates. Inspired by \cite{latafat2024adaptive}, we estimate the curvature of $f_2$ and $g_2$ with $\ell_{x,k}$, $\ell_{y,k}$, $L_{x,k}$, and $L_{y,k}$. We note that convexity of $f_2$ and $g_2$ implies $\ell_{x,k}$ and $\ell_{y,k}$ are nonnegative~\cite[Section~17]{bauschke2017convex}. In addition, ALiA requires estimates of the largest singular values of $A \in \mathbb{R}^{r \times p}$ and $B \in \mathbb{R}^{r \times q}$, which we approximate by \(a_k\) and \(b_k\), respectively.

Recalling the definition 
\[
\mathbf{L}(x,y,u)=f_1(x)+f_2(x)+g_1(y)+g_2(y)+\langle u , Ax+By-c\rangle,
\]
we present the following convergence guarantee for \ALiA\ with Subroutine \ref{alg:S1}.
\begin{theorem}\label{thm:1}
Assume $f_1$ and $g_1$ are convex, closed, and proper. Assume $f_2$ and $g_2$ are convex and locally smooth.
Assume $\mathbf{L}$ has a saddle point (not necessarily unique). 
Then the sequence $\{(x^k, y^k, u^k)\}_{k=0,1,2,\dots}$ generated by \ALiA\ with Subroutine~\ref{alg:S1} converges to a saddle point of $\mathbf{L}$.
\end{theorem}
We provide a proof sketch in Section~\ref{ss:thm1-sketch} and present the full proof in Section~\ref{sec:3}.

\paragraph{Discussion: $\varepsilon>0$ for point convergence.}
A small constant $\varepsilon>0$ (as in \cite{latafat2024adaptive}) is used to establish the point convergence result of Theorem~\ref{thm:1}, but we observe in the experiments of Section~\ref{sec:4} that $\varepsilon=0$ works just as well empirically.

\paragraph{Discussion: No linesearch.}
A notable strength of ALiA is that it does not require the use of a linesearch. Recall $u^{k+1} = u^{k} + \sigma\gamma_{k+1}\Delta u^{k+1}$. Note that 
\[
a_{k+1} = 
 \frac{\|A^\top\Delta u^{k+1}\|}{\| \Delta u^{k+1}\|} =
\frac{\|A^\top(u^{k+1} - u^k)\|}{\|u^{k+1} - u^k\|},\qquad  b_{k+1}  = \frac{\|B^\top\Delta u^{k+1}\|}{\| \Delta u^{k+1}\|}=
\frac{\|B^\top(u^{k+1} - u^k)\|}{\|u^{k+1} - u^k\|}
\]
do not depend on $\gamma_{k+1}$.
This crucial observation means we can compute $a_{k+1}$ and $b_{k+1}$ without knowing $u^{k+1}$ (which in turn depends on $a_{k+1}$ and $b_{k+1}$), and thus eliminating the need for backtracking linesearch.
In particular, adaPDM+ \cite{latafat2024adaptive} has an explicit backtracking loop to resolve the nested dependence between the update $y^{k+1}$ and the local operator-norm estimate $\eta_{k+1}$:
\[
\text{ while }\ \eta_{k+1}<\frac{\|A^\top (y^{k+1}-y^k)\|}{\|y^{k+1}-y^k\|}\ \text{ do }\ 
\eta_{k+1}\leftarrow 2\eta_{k+1} \,\text{ and recompute }y^{k+1}.
\]
Likewise, the Malitsky--Pock method \cite{malitsky2018first} computes the tentative update $y^{k+1}$ and accepts it only if the stopping inequality
\[
\tau_k\sigma_k\|A^\top(y^{k+1}-y^k)\|^2 +2\sigma_k\Bigl(h(y^{k+1})-h(y^k)-\langle\nabla h(y^k),\,y^{k+1}-y^k\rangle\Bigr)\le0.99\|y^{k+1}-y^k\|^2
\]
holds. Here, $\tau_k,\sigma_k>0$ are primal/dual stepsizes updated by backtracking and $h$ is the smooth term given by the problem. In \ALiA\, by contrast, the quantities can be computed directly without any backtracking.

\paragraph{Discussion: $\lambda^A$ and $\lambda^B$ strengthen Young's inequality.}
By Young's inequality, 
\[
\langle A^\top \Delta u^{k+1},x^{k}- x^{k-1} \rangle \le \frac{1}{16a_{k+1}^2}\| A^\top \Delta u^{k+1}\|^2+4a_{k+1}^2\| x^{k}- x^{k-1}\|^2.
\]
In contrast, the quantity $\lambda^{A}_{k+1}$ is defined so that
\[
\langle A^\top \Delta u^{k+1},x^{k}- x^{k-1} \rangle = \frac{\lambda^{A}_{k+1}}{16a_{k+1}^2}\| A^\top \Delta u^{k+1}\|^2+4a_{k+1}^2\lambda^{A}_{k+1}\| x^{k}- x^{k-1}\|^2
\]
holds as an equality.  The quantity $\lambda^{B}_{k+1}$ is defined analogously. Bounding inner products by sums of squared norms via Young's inequality is a standard step in adaptive methods (see, e.g., \cite{malitsky2020adaptive,malitsky2024adaptive,latafat2024adaptive}). The introduction of $\lambda^A$ and $\lambda^B$ can therefore be viewed as a tightened, equality-based analogue of this classical approach.

\paragraph{Discussion: Growth factor $\tfrac{3}{2}$ for $\gamma_k$.}
One may consider the more general update rule
\[
\gamma_{k+1} = \min\{ \kappa\gamma_k, \dots, \Gamma_x,\Gamma_y \},
\]
with a general growth factor $\kappa\in (1, \tfrac{1 + \sqrt{5}}{2})$, rather than fixing $\kappa = \tfrac{3}{2}$. There is, however, a trade-off in the choice of $\kappa$: larger values of $\kappa$ increase the growth rate of $\kappa\gamma_k$, but at the same time worsen other terms in the analysis involving $\lambda^A_{k+1}$, $\lambda^B_{k+1}$, $\Gamma_x$, and $\Gamma_y$. For algebraic simplicity, discussed further in Section~\ref{sec:3}, we fix $\kappa = \tfrac{3}{2}$ throughout our presentation and analysis. We remark that in Subroutine~\ref{alg:S2}, we take $\kappa=\tfrac{1+\sqrt{5}}{2}$ (the golden ratio), which requires a mildly modified analysis.

\paragraph{Discussion: Primal-dual ratio parameter $\sigma>0$.}
The parameter $\sigma>0$ must be manually tuned, and this is a limitation shared by adaptive primal-dual method \cite{malitsky2018first,latafat2024adaptive,lan2024auto}. Nonetheless, our experiments show that the method is generally robust to the choice of $\sigma$, and in practice, it converges faster than plain FLiP ADMM.

\subsection{Proof sketch of Theorem~\ref{thm:1}}
\label{ss:thm1-sketch}
Define
\[
P_k:= \mathbf{L}(x^{k}, y^{k}, u^\star)-\mathbf{L}(x^{\star}, y^\star, u^\star)
\]
and note that $P_k\ge 0$ for $k=0,1,\dots$.
The key step in the proof of Theorem~\ref{thm:1} is to establish the following descent lemma:
\begin{lemma} \label{lem:A1}
Let $\{(x^k, y^k, u^k)\}_{k=0,1,2,\dots}$ be the sequence generated by \ALiA\ with Subroutine~\ref{alg:S1}. Let
\begin{align*}
\mathcal{U}_k &= \frac{1}{2}\|x^{k}-x^\star\|^2 + \frac{1}{2}\|x^{k}-x^{k-1}\|^2 + \frac{1}{2}\|y^{k}-y^\star\|^2 +  \frac{1}{2}\|y^{k}-y^{k-1}\|^2 \nonumber \\
&+\frac{1}{2\sigma}\|u^{k}-u^\star\|^2 + 3\gamma_{k}P_{k-1}
\end{align*}
for $k=1,2,\dots$,
where $(x^\star,y^\star,u^\star)$ is a saddle point of the Lagrangian. Then, \begin{align}\label{eq:lya}
    \mathcal{U}_{k+1} \le \mathcal{U}_{k} -\varepsilon\|x^k-x^{k-1}\|^2 -\varepsilon\|y^k-y^{k-1}\|^2 -\varepsilon\|u^k-u^{k-1}\|^2-\underbrace{(3\gamma_k - 2\gamma_{k+1})}_{\ge 0} P_{k-1}
\end{align}
holds for $k=1,2,\dots$.

\end{lemma}

Next, we show that the stepsize sequence $\{\gamma_k\}_{k \ge 0}$ is bounded away from $0$.

\begin{lemma}\label{lem:gamma1}
    Let $\{\gamma_k\}_{k \ge 0}$ be the stepsize sequence generated by \ALiA\ with Subroutine~\ref{alg:S1}. Then, $\{\gamma_k\}_{k \ge 0}$ is bounded away from $0$, i.e., there exists $\gamma>0$ such that $ \gamma_k \ge \gamma >0$ for all $k=0,1,2,\dots$.
\end{lemma}

With these lemmas in hand, the remainder of the proof follows with a few additional steps. By Lemma~\ref{lem:A1}, the algorithm's sequence $\{(x^k, y^k, u^k)\}_{k=0,1,2,\dots}$ is bounded. Then, by a summability argument, we can argue
\[
x^k-x^{k-1}\rightarrow0,\qquad
y^k-y^{k-1}\rightarrow0,\qquad
u^k-u^{k-1}\rightarrow0.
\]
For $\{P_k\}_{k\ge 0}$ sequence in particular, we have
\[
\min_{k=0,\dots,K}P_k\le \frac{\mathcal{U}_1}{2\gamma_1+\sum_{k=1}^{K}\gamma_k+\gamma_{K+1}}\le\frac{\mathcal{U}_1}{(K+3)\gamma},
\]
where $\gamma>0$ is from Lemma~\ref{lem:gamma1}. The full proof is given in Section~\ref{sec:3}.

Note that the bound improves as $\gamma_k$ becomes larger. This observation provides intuition that a subroutine permitting larger choices of $\gamma_k$, while still ensuring a descent lemma in the style of Lemma~\ref{lem:A1}, can lead to faster convergence.

\subsection{Improved adaptive stepsize and dual update selection subroutine} 
Following the discussion in Section~\ref{ss:thm1-sketch}, we present Subroutine~\ref{alg:S2}, which admits a larger stepsize sequence $\{\gamma_k\}_{k=0,1,\dots}$ through a more careful parameter selection mechanism and a tighter analysis. 
Define 
\[
F_{k}(\cdot) = \mathrm{Id}(\cdot) - \gamma_{k} \nabla f_2(\cdot) ,\qquad G_{k}(\cdot) = \mathrm{Id}(\cdot) - \gamma_{k} \nabla g_2(\cdot).
\]
for $k=0,1,\dots$.

\begin{subroutine}[H]
\normalsize
\caption{Alternative stepsize selection scheme for \ALiA}
\label{alg:S2}
\begin{algorithmic}[1]
\itemsep=3pt%
\renewcommand{\algorithmicindent}{0.3cm}%
\Require
\begin{tabular}[t]{@{}l@{}}
Previous stepsize $\gamma_k>0$.\\
Primal--dual stepsize ratio parameter $\sigma>0$.\\
Current and previous primal--dual pairs $(x^{k},y^{k},u^{k})$ and $(x^{k-1},y^{k-1},u^{k-1})$.\\
Strictly positive $\varepsilon$ with $0<\varepsilon<\min\{\tfrac12,\tfrac{1}{4\sigma}\}$.\\
\end{tabular}

\begin{align*}
&\varphi=\frac{1+\sqrt{5}}{2}, \quad \Delta u^{k+1}
= Ax^k + By^k - c
+ \varphi A(x^k-x^{k-1})
+ \varphi B(y^k-y^{k-1}),\\[0.05in]
&a_{k+1}
= \frac{\|A^\top\Delta u^{k+1}\|}{\| \Delta u^{k+1}\|},
\qquad
b_{k+1}
= \frac{\|B^\top\Delta u^{k+1}\|}{\| \Delta u^{k+1}\|},\\[0.05in]
&\lambda^{A}_{k+1}
= \frac{\langle A^\top \Delta u^{k+1},\, x^{k}-x^{k-1}\rangle}
{\frac{\|A^\top \Delta u^{k+1}\|^2}{8\varphi a_{k+1}^2}
+2\varphi a_{k+1}^2\|x^{k}-x^{k-1}\|^2},\quad \lambda^{B}_{k+1}
= \frac{\langle B^\top \Delta u^{k+1},\, y^{k}-y^{k-1}\rangle}
{\frac{\|B^\top \Delta u^{k+1}\|^2}{8\varphi b_{k+1}^2}
+2\varphi b_{k+1}^2\|y^{k}-y^{k-1}\|^2},\\[0.05in]
&\mu^{A}_{k+1}
= \frac{\langle A^\top \Delta u^{k+1},\, F_k(x^{k-1})-F_k(x^{k})\rangle}
{\frac{\gamma_k\|A^\top \Delta u^{k+1}\|^2}{2a_{k+1}^2}
+\frac{a_{k+1}^2}{2\gamma_k}\|F_k(x^{k-1})-F_k(x^{k})\|^2},\\[0.1in]
&\mu^{B}_{k+1}
= \frac{\langle B^\top \Delta u^{k+1},\, G_k(y^{k-1})-G_k(y^{k})\rangle}
{\frac{\gamma_k\|B^\top \Delta u^{k+1}\|^2}{2b_{k+1}^2}
+\frac{b_{k+1}^2}{2\gamma_k}\|G_k(y^{k-1})-G_k(y^{k})\|^2},\\[0.05in]
&\text{(use convention $0/0=0$ for $a_{k+1},b_{k+1},\lambda^{A}_{k+1},\lambda^{B}_{k+1},\mu^{A}_{k+1},\mu^{B}_{k+1}$)},\\[0.05in]
&\ell_{x,k},\, L_{x,k},\, \ell_{y,k},\, L_{y,k},\, \delta_{x,k},\, \delta_{y,k}
\quad \text{are computed as in Subroutine~\ref{alg:S1}},\\[0.1in]
&p(x)=\dfrac{\sigma a_{k+1}^2\mu_{k+1}^{A}(\delta_{x,k}+1)}{\gamma_k^2}x^3
+\Bigl(2\varphi^2\sigma a_{k+1}^2\lambda^{A}_{k+1}
+\dfrac{\delta_{x,k}}{\gamma_k^2}\Bigr)x^2
+\varphi\ell_{x,k}x-\dfrac{1-2\varepsilon}{2},\\[0.1in]
&q(y)=\dfrac{\sigma b_{k+1}^2\mu_{k+1}^{B}(\delta_{y,k}+1)}{\gamma_k^2}y^3
+\Bigl(2\varphi^2\sigma b_{k+1}^2\lambda^{B}_{k+1}
+\dfrac{\delta_{y,k}}{\gamma_k^2}\Bigr)y^2
+\varphi\ell_{y,k}y-\dfrac{1-2\varepsilon}{2},\\[0.05in]
&\Gamma_x := \min\{t>0:\, p(t)=0\},\qquad
\Gamma_y := \min\{t>0:\, q(t)=0\},
\quad\text{with the convention $\min\emptyset=+\infty$,}\\
&\Theta_{k+1} := \frac{4-\lambda^{A}_{k+1}-\lambda^{B}_{k+1}-8\sigma \varepsilon}{4\sigma},\quad \Psi_{k+1} :=
\frac{\mu^{A}_{k+1}+\mu^{B}_{k+1}}{\sigma}
+\sqrt{
\frac{(\mu^{A}_{k+1}+\mu^{B}_{k+1})^2}{\sigma^2}
+2(a_{k+1}^2+b_{k+1}^2)\Theta_{k+1}
},\\[0.1in]
&\gamma_{k+1}
=\min\Biggl\{
\varphi\gamma_k,\ \frac{\Theta_{k+1}}{\Psi_{k+1}}, \ \Gamma_x,\ \Gamma_y
\Biggr\}.
\end{align*}

\!\!\!\!\!\!\!\!\!\!\!\!\!\!  \Return Next stepsize $\gamma_{k+1}$ and next update direction $\Delta u^{k+1}$
\end{algorithmic}
\end{subroutine}

Like Subroutine~\ref{alg:S1}, \ALiA\ with Subroutine~\ref{alg:S2} satisfies the following convergence guarantee:

\begin{theorem}\label{thm:2}
Assume $f_1$ and $g_1$ are convex, closed, and proper. Assume $f_2$ and $g_2$ are convex and locally smooth.
Assume $\mathbf{L}$ has a saddle point (not necessarily unique). 
Then the sequence $\{(x^k, y^k, u^k)\}_{k=0,1,2,\dots}$ generated by \ALiA\ with Subroutine~\ref{alg:S2} converges to a saddle point of $\mathbf{L}$.
\end{theorem}
The proof of Theorem~\ref{thm:2} is also presented  in Section~\ref{sec:3}.

\paragraph{Tighter upper bound estimate.}
In the proof of Lemma~\ref{lem:A1}, we use the following lemma to eliminate the troublesome cross terms that arise when expanding the square.

\begin{lemma}\label{lem:2}
Let $x,y \in \mathbb{R}^d$ and $a,b>1$ such that $\frac{1}{a}+\frac{1}{b}=1$. Then, 
\[
\|x+y\|^2 =a\|x\|^2 + b \|y\|^2
-\|\sqrt{a-1}x-\sqrt{b-1}y\|^2\le a\|x\|^2 + b \|y\|^2
\]
\end{lemma}
Rather than directly invoking the inequality in Lemma~\ref{lem:2}, we keep the exact identity and explicitly track the negative quadratic term $-\|\sqrt{a-1}x-\sqrt{b-1}y\|^2$. By introducing the auxiliary parameters $\mu^{A}_{k+1}$ and $\mu^{B}_{k+1}$, we can control this term instead of discarding it, which leads to a tighter bound than the one obtained from the inequality.

\paragraph{Cost of solving the cubic.}
Notably, Subroutine~\ref{alg:S2} requires solving a cubic polynomial, whereas Subroutine~\ref{alg:S1} requires solving only a quadratic polynomial. The additional cost incurred by solving the cubic is comparatively negligible. Indeed, in most practical applications, the dominant computational cost of ALiA arises from evaluating the gradients $\nabla f_2$ and $\nabla g_2$, computing the matrix-vector products $Ax^{k+1}$, $By^{k+1}$, $A^\top \Delta u^{k+1}$, and $B^\top \Delta u^{k+1}$, and evaluating the proximal operators $\prox_{f_1}$ and $\prox_{g_1}$.

To compute the real roots of a cubic polynomial
\[
ax^3 + bx^2 + cx + d = 0,
\]
we use the explicit formula provided by \cite{bauschke2023real}.
First, we compute
\[
P=\frac{3ac-b^{2}}{3a^{2}},\qquad
Q=\frac{2b^{3}-9abc+27a^{2}d}{27a^{3}},\qquad
\Delta=\left(\frac{Q}{2}\right)^{2}+\left(\frac{P}{3}\right)^{3}
\]
and then proceed according to the following mutually exclusive cases.
\begin{itemize}
\item[(i)] $b^2 = 3ac$ or $\Delta > 0$.  
Then the cubic has exactly one real root, given by
\begin{equation*}
-\frac{b}{3a}
+ \sqrt[3]{-\frac{Q}{2} + \sqrt{\Delta}}
+ \sqrt[3]{-\frac{Q}{2} - \sqrt{\Delta}} .
\end{equation*}
\item[(ii)] $b^2 > 3ac$ and $\Delta = 0$.  
Then the cubic has exactly two real roots, given by
\begin{equation*}
-\frac{b}{3a} + 2\sqrt[3]{-\frac{Q}{2}},
\qquad
-\frac{b}{3a}- \sqrt[3]{-\frac{Q}{2}} .
\end{equation*}
\item[(iii)] $\Delta < 0$.  
Then the cubic has exactly three distinct real roots $x_0,x_1,x_2$, where
\begin{equation*}
x_k
= -\frac{b}{3a} + 2\sqrt{-\frac{P}{3}}
\cos\!\left(\frac{\theta + 2k\pi}{3}\right),
\qquad
\theta = \arccos\!\left(
\frac{-Q/2}{(-P/3)^{3/2}}
\right),
\quad k=0,1,2.
\end{equation*}
\end{itemize}
There is, however, an important numerical consideration. When $\mu^A_{k+1} \approx 0$ or $\mu^B_{k+1} \approx 0$, the leading coefficient $a$ of the cubic polynomial becomes small. In this regime, the direct application of the analytic formulas above suffers from numerical instability, which in turn makes ALiA unstable. Therefore, when the coefficient of the $x^3$-term is sufficiently small, we instead approximate the solution by first solving the quadratic equation $bx^2 + cx + d = 0$, and then using the resulting root as an initial guess for Newton's method \cite{press2007numerical} applied to the full cubic equation $ax^3 + bx^2 + cx + d = 0$.
This refinement procedure converges rapidly and yields a root accurate to machine precision. With this remedy in place, ALiA becomes numerically stable once again.

\section{Proof of Theorems~\ref{thm:1} and \ref{thm:2}}\label{sec:3}
We now prove Theorems~\ref{thm:1} and \ref{thm:2}.
Both convergence results primarily rely on the ``Lyapunov sequence'' $\{\mathcal{U}_k\}_{k=1,2,\dots}$ and the descent lemma. Throughout the section, let $\rho_{k+1}=\frac{\gamma_{k+1}}{\gamma_k}$ for simplicity. We first begin with properties that derived from the primal updates of $x$ and $y$.
The $x$-update  
\[x^{k+1}\in \argmin_{x\in\mathbb{R}^p}\left\{f_1(x)+\langle \nabla f_2(x^{k}) + A^\top u^{k+1}, x\rangle +\frac{1}{2\gamma_{k+1}}\|x-x^{k}\|^2 \right\}
\]
implies 
\begin{align*}
0 &\in \partial f_1(x^{k+1}) + \nabla f_2 (x^{k}) + A^\top u^{k+1} + \frac{1}{\gamma_{k+1}}\left( x^{k+1}-x^{k}\right)
\end{align*}
and therefore,
\begin{align}\label{eq:A1-1}
\frac{1}{\gamma_{k+1}}\left(F_{k+1}(x^{k}) -  x^{k+1}\right) -A^\top u^{k+1}  \in \partial f_1(x^{k+1}).
\end{align}
Decrement the indices by $1$ to get 

\begin{align}\label{eq:A1-2}
\frac{1}{\gamma_{k}}\left(F_{k}(x^{k-1}) -  x^{k}\right) -A^\top u^{k} \in  \partial f_1(x^{k}) .
\end{align}
Similarly for $y$-update, 
\begin{align}\label{eq:A1-3}
\frac{1}{\gamma_{k+1}}\left(G_{k+1}(y^{k}) -  y^{k+1}\right) -B^\top u^{k+1} \in  \partial g_1(y^{k+1})
\end{align}
and
\begin{align}\label{eq:A1-4}
 \frac{1}{\gamma_{k}}\left(G_k(y^{k-1}) -  y^{k}\right) -B^\top u^{k}\in\partial g_1(y^{k}) .
\end{align}

\subsection{Proof of Lemma~\ref{lem:A1} and Lemma~\ref{lem:gamma1}}

Now we can start with the proof of Lemma~\ref{lem:A1}. We need a few more components to complete the proof. 
The following lemma controls the cross terms.

\begin{lemma}\label{lem:1}
Consider a sequence $\{(x^k, y^k, u^k)\}_{k=0,1,2,\dots}$ generated by Algorithm~\ref{alg:A1}. Then, 
\begin{align*}
&\langle A^\top (u^{k+1}-u^{k}), x^{k}-x^{k+1}\rangle+\frac{1}{\gamma_k}\langle F_k(x^{k-1})-F_k(x^{k}), x^k- x^{k+1} \rangle \nonumber \\
&\quad \le \gamma_{k+1}\left\|\frac{1}{\gamma_k}\left( F_k(x^{k-1})-F_k(x^{k})\right)+ A^\top (u^{k+1}-u^{k})\right\|^2.
\end{align*}
and
\begin{align*}
&\langle B^\top (u^{k+1}-u^{k}), y^{k}-y^{k+1}\rangle+\frac{1}{\gamma_k}\langle G_k(y^{k-1})-G_k(y^{k}), y^k- y^{k+1} \rangle \nonumber \\
&\quad \le \gamma_{k+1}\left\|\frac{1}{\gamma_k}\left( G_k(y^{k-1})-G_k(y^{k})\right)+ B^\top (u^{k+1}-u^{k})\right\|^2.
\end{align*}
\end{lemma}
\begin{proof}
We prove the first inequality since the proof of the second inequality is almost identical, with $F_k$ and $x$-iterates replaced with $G_k$ and $y$-iterates.
Suppose $x^{k+1}\neq x^{k}$, otherwise there is nothing to prove. We first derive the lower bound of the left-hand side by the following
\begin{align*}
    &\frac{1}{\gamma_k}\langle F_k(x^{k-1})-F_k(x^{k}), x^k- x^{k+1} \rangle \\
    & = \left\langle \frac{F_k(x^{k-1})-x^k}{\gamma_k} +\nabla f_2(x^k) , x^k - x^{k+1} \right\rangle \\
     & = \left\langle \frac{F_k(x^{k-1})-x^k}{\gamma_k} +\nabla f_2(x^k) - \frac{x^k-x^{k+1}}{\gamma_{k+1}} , x^k - x^{k+1} \right\rangle +\frac{1}{\gamma_{k+1}}\|x^k-x^{k+1}\|^2\\
       & = \left\langle \frac{F_k(x^{k-1})-x^k}{\gamma_k} - \frac{F_{k+1}(x^{k})-x^{k+1}}{\gamma_{k+1}} , x^k - x^{k+1} \right\rangle +\frac{1}{\gamma_{k+1}}\|x^k-x^{k+1}\|^2\\
         & = \left\langle \underbrace{\frac{F_k(x^{k-1})-x^k}{\gamma_k}-A^\top u^k}_{\in \partial f_1(x^k)}  - \underbrace{\left(\frac{F_{k+1}(x^{k})-x^{k+1}}{\gamma_{k+1}} - A^\top u^{k+1} \right)}_{\in \partial f_1(x^{k+1})}, x^k - x^{k+1} \right\rangle +\frac{1}{\gamma_{k+1}}\|x^k-x^{k+1}\|^2 \\
         &\quad + \langle A^\top (u^k-u^{k+1}), x^{k}-x^{k+1}\rangle \\
         &\ge \frac{1}{\gamma_{k+1}}\|x^k-x^{k+1}\|^2+ \langle A^\top (u^k-u^{k+1}), x^{k}-x^{k+1}\rangle
\end{align*}
where the last inequality is from the monotonicity of $\partial f_1$. So,
\[
\langle A^\top (u^{k+1}-u^{k}), x^{k}-x^{k+1}\rangle+\frac{1}{\gamma_k}\langle F_k(x^{k-1})-F_k(x^{k}), x^k- x^{k+1} \rangle  \ge\frac{1}{\gamma_{k+1}}\|x^k-x^{k+1}\|^2.
\]
Then, we use the Cauchy--Schwarz inequality to get 
\begin{align}\label{eq:lem1-1}
\frac{1}{\gamma_{k+1}}\|x^k-x^{k+1}\|^2 \le\left \|\frac{1}{\gamma_k}\left( F_k(x^{k-1})-F_k(x^{k})\right)+ A^\top (u^{k+1}-u^{k})\right\| \| x^k- x^{k+1}\| 
\end{align}
Then it follows that 
\begin{align}\label{eq:lem1-2}
\|x^k-x^{k+1}\| \le \gamma_{k+1} \left\|\frac{1}{\gamma_k}\left( F_k(x^{k-1})-F_k(x^{k})\right)+ A^\top (u^{k+1}-u^{k})\right\|
\end{align}
Plug in \eqref{eq:lem1-2} to the right hand side of \eqref{eq:lem1-1} to get
\[
\mathrm{LHS} \le \gamma_{k+1} \left\|\frac{1}{\gamma_k}\left( F_k(x^{k-1})-F_k(x^{k})\right)+ A^\top (u^{k+1}-u^{k})\right\|^2.
\]
\qed
\end{proof}
The next lemma follows by the definition of $\ell_{x,k},\ell_{y,k},L_{x,k}$, and $L_{y,k}$.

\begin{lemma}\label{lem:3} The following holds
\[
\| F_k(x^{k-1})-F_k(x^{k})\| ^2 = (\gamma_k^2 L_{x,k}^2 -2\gamma_{k}\ell_{x,k} +1)\|x^{k-1}-x^k\|^2.
\]
and
\[
\| G_k(y^{k-1})-G_k(y^{k})\| ^2 = (\gamma_k^2 L_{y,k}^2 -2\gamma_k\ell_{y,k} +1)\|y^{k-1}-y^k\|^2
\]
\end{lemma}
\begin{proof}
    Expand the left-hand side as 
    \begin{align*}
        \| F_k(x^{k-1})-F_k(x^{k})\| ^2 &= \| x^{k-1}-x^k-\gamma_k \nabla f_2(x^{k-1})+\gamma_k \nabla f_2(x^{k})\|^2\\
        &=\|x^{k-1}-x^k\|^2 - 2\gamma_k \langle  x^{k-1}-x^k ,  \nabla f_2(x^{k-1})-  \nabla f_2(x^{k}) \rangle + \gamma_k^2\| \nabla f_2(x^{k-1})-  \nabla f_2(x^{k}) \|^2\\
         &=\|x^{k-1}-x^k\|^2 - 2\gamma_k\ell_{x,k} \|x^{k-1}-x^k\|^2+ \gamma_k^2 L_{x,k}^2 \|x^{k-1}-x^k\|^2.
    \end{align*}
Repeat the similar process for $y$-iterate. \qed
\end{proof}

The next lemma controls the cross terms with $\lambda^{A}_{k+1}$ and $\lambda^{B}_{k+1}$.
\begin{lemma}\label{lem:4}
Let 
\[
\lambda^{A}_{k+1}=\frac{\langle A^\top \Delta u^{k+1},x^{k}- x^{k-1} \rangle }{\frac{\| A^\top \Delta u^{k+1}\|^2}{16a_{k+1}^2}+4a_{k+1}^2\| x^{k}- x^{k-1}\|^2} , \quad \lambda^{B}_{k+1}=\frac{\langle B^\top \Delta u^{k+1},y^{k}- y^{k-1} \rangle }{\frac{\| B^\top \Delta u^{k+1}\|^2}{16b_{k+1}^2}+4b_{k+1}^2\| y^{k}- y^{k-1}\|^2} 
\]
Then $-1\le \lambda^{A}_{k+1},\lambda^{B}_{k+1}\le1$ and the following holds 
\[
  \langle A^\top(u^{k+1}-u^{k}) , x^{k}- x^{k-1} \rangle = \frac{\lambda_{k+1}^{A}}{16\sigma\gamma_{k+1}}\|u^{k+1}-u^k\|^2  +4\sigma a_{k+1}^2\gamma_{k+1}\lambda^{A}_{k+1}\|x^{k}-x^{k-1}\|^2,
\]
\[
  \langle B^\top(u^{k+1}-u^{k}) , y^{k}- y^{k-1} \rangle = \frac{\lambda_{k+1}^{B}}{16\sigma\gamma_{k+1}}\|u^{k+1}-u^k\|^2  +4\sigma b_{k+1}^2\gamma_{k+1}\lambda^{B}_{k+1}\|y^{k}-y^{k-1}\|^2.
\]
\end{lemma}
\begin{proof}
The bounds $-1\le \lambda^{A}_{k+1},\lambda^{B}_{k+1}\le1$ are from Young's inequality. For equalities, 
\begin{align*}
       \langle A^\top(u^{k+1}-u^{k}) , x^{k}- x^{k-1} \rangle &=\sigma\gamma_{k+1}\langle A^\top \Delta u^{k+1} , x^{k}- x^{k-1} \rangle    \\
       &=\frac{\sigma \gamma_{k+1}\lambda^{A}_{k+1}}{16a_{k+1}^2}\|A^\top \Delta u^{k+1}\|^2 +4\sigma a_{k+1}^2\gamma_{k+1}\lambda^{A}_{k+1}\|x^{k}-x^{k-1}\|^2   \\
       &=\frac{\lambda_{k+1}^{A}}{16\sigma \gamma_{k+1}}\|u^{k+1}-u^k\|^2  +4\sigma a_{k+1}^2\gamma_{k+1}\lambda^{A}_{k+1}\|x^{k}-x^{k-1}\|^2.
\end{align*}
Do the same thing for $y$-iterate. \qed
\end{proof}

The next lemma easily follows from the optimality conditions, hence we present it without proof.
\begin{lemma}\label{lem:5}
Assume $h_1$ is CCP and $h_2$ is differentiable. If $z^\star \in \argmin \left\{ h_1(z) + h_2(z)\right\}$, then $z^\star \in \argmin \left\{ h_1(z) + \langle \nabla h_2(z^\star), z-z^\star \rangle \right\}$.
\end{lemma}

\vspace{0.1in}

Now we are ready to prove the main descent lemma. 

\vspace{0.1in}

\paragraph{Proof of Lemma~\ref{lem:A1}.}
 Recall the $x^{k+1}$ and $y^{k+1}$ updates 
 \begin{align*}
 x^{k+1}&\in \argmin\left\{f_1(x)+\langle  A^\top u^{k+1} + \nabla f_2(x^k) , x\rangle +\frac{1}{2\gamma_{k+1}}\|x-x^{k}\|^2 \right\},\\
 y^{k+1}&\in \argmin\left\{g_1(y)+\langle  B^\top u^{k+1} + \nabla g_2(y^k), y\rangle + \frac{1}{2\gamma_{k+1}}\|y-y^{k}\|^2 \right\}.
 \end{align*}
Apply Lemma~\ref{lem:5} on the first inclusion by setting $h_1$ = $f_1$ and $h_2$ to be the rest of the terms. For the second inclusion, set $h_1$ = $g_1$ and $h_2$ to be the rest of the terms. Then we get the following for any $x$ and $y$
\begin{align}
    0&\le f_1(x)-f_1(x^{k+1}) + \langle \nabla f_2 (x^{k}), x - x^{k+1} \rangle+ \langle A^\top u^{k+1}, x - x^{k+1} \rangle +\frac{1}{\gamma_{k+1}}\langle x^{k+1}-x^k, x - x^{k+1} \rangle. \nonumber \\
        0&\le g_1(y)-g_1(y^{k+1}) + \langle \nabla g_2 (y^{k}), y - y^{k+1} \rangle + \langle B^\top u^{k+1} , y - y^{k+1}\rangle +\frac{1}{\gamma_{k+1}}\langle y^{k+1}-y^k, y - y^{k+1} \rangle. \label{eq:lemA1-1}
\end{align}
Plug in $x=x^\star$, $y=y^\star$ to get
\begin{align}
    0&\le f_1(x^\star)-f_1(x^{k+1}) + \underbrace{\langle \nabla f_2 (x^{k}), x^{\star} - x^{k+1} \rangle }_{\mathrm{(A1)}}+ \langle A^\top u^{k+1} , x^{\star} - x^{k+1} \rangle+\frac{1}{\gamma_{k+1}}\langle x^{k+1}-x^k, x^{\star} - x^{k+1} \rangle. \nonumber \\
        0&\le g_1(y^\star)-g_1(y^{k+1}) + \underbrace{\langle \nabla g_2 (y^{k}), y^{\star} - y^{k+1} \rangle }_{\mathrm{(A2)}}+ \langle B^\top u^{k+1} , y^{\star} - y^{k+1}\rangle +\frac{1}{\gamma_{k+1}}\langle y^{k+1}-y^k, y^{\star} - y^{k+1} \rangle. \nonumber
\end{align}
By using $2\langle a-b , c-a \rangle = -\|a-b\|^2-\|c-a\|^2 + \|b-c\|^2$,
\[
\frac{1}{\gamma_{k+1}}\langle x^{k+1}-x^k, x^{\star} - x^{k+1} \rangle = \frac{1}{2\gamma_{k+1}}\|x^{k}-x^\star\|^2 - \frac{1}{2\gamma_{k+1}}\|x^{k+1}-x^\star\|^2- \frac{1}{2\gamma_{k+1}}\|x^{k+1}-x^{k}\|^2 
\]
and
\[
\frac{1}{\gamma_{k+1}}\langle y^{k+1}-y^k, y^{\star} - y^{k+1} \rangle = \frac{1}{2\gamma_{k+1}}\|y^{k}-y^\star\|^2 - \frac{1}{2\gamma_{k+1}}\|y^{k+1}-y^\star\|^2- \frac{1}{2\gamma_{k+1}}\|y^{k+1}-y^{k}\|^2.
\]
We further bound (A1) as
\begin{align}
    \mathrm{(A1)} &= \langle \nabla f_2(x^k), x^\star- x^k \rangle + \langle \nabla f_2(x^k), x^k- x^{k+1} \rangle \nonumber \\
    &\le f_2(x^\star)-f_2(x^k)  + \langle A^\top u^k, x^{k+1} -x^{k} \rangle\nonumber \\
    &\quad + \left\langle  \frac{F_{k}(x^{k-1})-F_k(x^k)}{\gamma_k} , x^k- x^{k+1}\right\rangle+ \left\langle \underbrace{\frac{F_{k}(x^{k-1}) -  x^k}{\gamma_k} -A^\top u^k}_{\in\partial f_1(x^k)} , x^{k+1}- x^{k} \right\rangle \nonumber \\
    & \le f_2(x^\star)-f_2(x^k) +f_1(x^{k+1})-f_1(x^k) +\frac{1}{\gamma_k}\langle F_{k}(x^{k-1})-F_k(x^k) , x^k-x^{k+1} \rangle +\langle A^\top u^k, x^{k+1} -x^{k} \rangle \nonumber
\end{align}
Similarly for (A2),
\begin{align}
    \mathrm{(A2)} \le g_2(y^\star)-g_2(y^k) + g_1(y^{k+1})-g_1(y^k) +\frac{1}{\gamma_k}\langle G_k(y^{k-1})-G_k(y^k) , y^k-y^{k+1} \rangle +\langle B^\top u^k, y^{k+1} -y^{k} \rangle. \label{eq:lemA1-2}
\end{align}
Combine the results to get
    \begin{align}
        0&\le \frac{1}{2\gamma_{k+1}}\|x^{k}-x^\star\|^2 - \frac{1}{2\gamma_{k+1}}\|x^{k+1}-x^\star\|^2 - \frac{1}{2\gamma_{k+1}}\|x^{k+1}-x^{k}\|^2 \nonumber \\
        &\quad + \frac{1}{2\gamma_{k+1}}\|y^{k}-y^\star\|^2 - \frac{1}{2\gamma_{k+1}}\|y^{k+1}-y^\star\|^2 - \frac{1}{2\gamma_{k+1}}\|y^{k+1}-y^{k}\|^2\nonumber \\
        &\quad +\frac{1}{\gamma_k}\langle  F_k(x^{k-1})-F_k(x^k) , x^k-x^{k+1} \rangle +\langle A^\top u^k, x^{k+1} -x^{k} \rangle \nonumber \\
        &\quad+ \frac{1}{\gamma_k}\langle G_{k}(y^{k-1})-G_k(y^k) , y^k-y^{k+1} \rangle+ \langle B^\top u^k, y^{k+1} -y^{k} \rangle \nonumber \\
        &\quad +\langle u^{k+1} , A(x^\star-x^{k+1})+B(y^{\star}-y^{k+1}) \rangle +f(x^\star)-f(x^{k})+g(y^\star)-g(y^k). \label{eq:lemA1-3}
    \end{align}
Now we decrement the indices of \eqref{eq:lemA1-1} by 1, and substitute $x=x^{k-1}$ and $y=y^{k-1}$ to get 
\begin{align*}
      0&\le f_1(x^{k-1})-f_1(x^{k}) +  \underbrace{\langle \nabla f_2 (x^{k-1}), x^{k-1} - x^{k} \rangle}_{\mathrm{(B1)}}+ \langle A^\top u^{k} , x^{k-1}- x^{k} \rangle +\frac{1}{\gamma_{k}}\langle x^{k}-x^{k-1}, x^{k-1} - x^{k} \rangle,
\end{align*}
\begin{align*}
      0&\le g_1(y^{k-1})-g_1(y^{k}) + \underbrace{\langle \nabla g_2 (y^{k-1}), y^{k-1} - y^{k} \rangle}_{\mathrm{(B2)}} + \langle B^\top u^{k} , y^{k-1}- y^{k} \rangle +\frac{1}{\gamma_{k}}\langle y^{k}-y^{k-1}, y^{k-1} - y^{k} \rangle. 
\end{align*}
We upper bound (B1) as 
\begin{align*}
    \mathrm{(B1)} &= \langle \nabla f_2(x^{k-1})-\nabla f_2(x^{k}), x^{k-1}- x^k \rangle + \langle \nabla f_2(x^k), x^{k-1}- x^k \rangle  \\
    &\le \ell_{x,k} \|x^{k-1}-x^k\|^2+ f_2(x^{k-1})-f_2(x^k)
\end{align*}
Similarly for (B2),
\begin{align*}
    \mathrm{(B2)} &= \langle \nabla g_2(y^{k-1})-\nabla g_2(y^{k}), y^{k-1}- y^k \rangle + \langle \nabla g_2(y^k), y^{k-1}- y^k \rangle  \\
    &\le \ell_{y,k} \|y^{k-1}-y^k\|^2+ g_2(y^{k-1})-g_2(y^k)
\end{align*}
Combine the results to get 
\begin{align}
      0&\le f(x^{k-1})-f(x^{k}) + \langle A^\top u^{k} , x^{k-1}- x^{k} \rangle -\frac{1-\ell_{x,k}\gamma_k}{\gamma_{k}}\| x^{k}-x^{k-1}\|^2 \nonumber \\
      0&\le g(y^{k-1})-g(y^{k}) + \langle B^\top u^{k} , y^{k-1}- y^{k} \rangle-\frac{1-\ell_{y,k}\gamma_k}{\gamma_k}\|y^{k}-y^{k-1}\|^2. \label{eq:lemA1-4}
\end{align}
Multiply \eqref{eq:lemA1-4} by $2$ and add it to \eqref{eq:lemA1-3} to get 
    \begin{align}
        0&\le \frac{1}{2\gamma_{k+1}}\|x^{k}-x^\star\|^2 - \frac{1}{2\gamma_{k+1}}\|x^{k+1}-x^\star\|^2 - \frac{1}{2\gamma_{k+1}}\|x^{k+1}-x^{k}\|^2 -\frac{2(1-\ell_{x,k}\gamma_k)}{\gamma_k}\|x^k-x^{k-1}\|^2\nonumber \\
        &\quad + \frac{1}{2\gamma_{k+1}}\|y^{k}-y^\star\|^2 - \frac{1}{2\gamma_{k+1}}\|y^{k+1}-y^\star\|^2 -  \frac{1}{2\gamma_{k+1}}\|y^{k+1}-y^{k}\|^2-\frac{2(1-\ell_{y,k}\gamma_k)}{\gamma_k}\|y^k-y^{k-1}\|^2\nonumber \\
        &\quad +\frac{1}{\gamma_k}\langle F_k(x^{k-1})-F_k(x^k) , x^k-x^{k+1} \rangle +\langle A^\top u^k, x^{k+1} -x^{k} \rangle \nonumber \\
        &\quad +\frac{1}{\gamma_k}\langle G_{k}(y^{k-1})-G_k(y^k) , y^k-y^{k+1} \rangle+\langle B^\top u^k, y^{k+1} -y^{k} \rangle\nonumber \\
        &\quad+\langle u^{k+1} , A(x^\star-x^{k+1})+B(y^{\star}-y^{k+1}) \rangle  +2\langle A^\top u^{k} , x^{k-1}- x^{k} \rangle+2\langle B^\top u^{k} , y^{k-1}- y^{k} \rangle\nonumber\\
        &\quad +f(x^\star)-f(x^{k})+g(y^\star)-g(y^k)+2 f(x^{k-1}) - 2 f(x^{k})+ 2 g(y^{k-1}) - 2g(y^{k}). \nonumber
    \end{align}
     Add and subtract $\langle u^\star, A(x^\star-x^{k+1})+B(y^\star-y^{k+1}) \rangle$ to get 
       \begin{align}
        0&\le \frac{1}{2\gamma_{k+1}}\|x^{k}-x^\star\|^2 - \frac{1}{2\gamma_{k+1}}\|x^{k+1}-x^\star\|^2 - \frac{1}{2\gamma_{k+1}}\|x^{k+1}-x^{k}\|^2 -\frac{2(1-\ell_{x,k}\gamma_k)}{\gamma_k}\|x^k-x^{k-1}\|^2\nonumber \\
        &\quad + \frac{1}{2\gamma_{k+1}}\|y^{k}-y^\star\|^2 - \frac{1}{2\gamma_{k+1}}\|y^{k+1}-y^\star\|^2 -  \frac{1}{2\gamma_{k+1}}\|y^{k+1}-y^{k}\|^2-\frac{2(1-\ell_{y,k}\gamma_k)}{\gamma_k}\|y^k-y^{k-1}\|^2\nonumber \\
        &\quad +\frac{1}{\gamma_k}\langle F_k(x^{k-1})-F_k(x^k) , x^k-x^{k+1} \rangle +\langle A^\top u^k, x^{k+1} -x^{k} \rangle \nonumber \\
        &\quad +\frac{1}{\gamma_k}\langle G_{k}(y^{k-1})-G_k(y^k) , y^k-y^{k+1} \rangle+\langle B^\top u^k, y^{k+1} -y^{k} \rangle\nonumber \\
        &\quad+ \langle u^{k+1} - u^\star , A(x^\star-x^{k+1})+B(y^{\star}-y^{k+1}) \rangle  +2\langle A^\top u^{k} , x^{k-1}- x^{k} \rangle+2\langle B^\top u^{k} , y^{k-1}- y^{k} \rangle \nonumber\\
        &\quad  +\langle u^\star , A(x^\star-x^{k})+B(y^{\star}-y^{k}) \rangle+\langle u^\star , A(x^{k}-x^{k+1})+B(y^{k}-y^{k+1}) \rangle \nonumber \\
        &\quad+f(x^\star)-f(x^{k})+g(y^\star)-g(y^k) +2 f(x^{k-1}) - 2 f(x^{k})+ 2 g(y^{k-1}) - 2g(y^{k}).  \nonumber 
    \end{align}  
Recall $P_{k} = \mathbf{L}(x^{k}, y^{k}, u^\star)-\mathbf{L}(x^{\star}, y^\star, u^\star) = f(x^k)+g(y^k)-f(x^\star)-g(y^\star)+\langle u^\star, A(x^k-x^\star)+B(y^k-y^\star) \rangle \ge0$ to get
       \begin{align}
        0&\le \frac{1}{2\gamma_{k+1}}\|x^{k}-x^\star\|^2 - \frac{1}{2\gamma_{k+1}}\|x^{k+1}-x^\star\|^2 - \frac{1}{2\gamma_{k+1}}\|x^{k+1}-x^{k}\|^2 -\frac{2(1-\ell_{x,k}\gamma_k)}{\gamma_k}\|x^k-x^{k-1}\|^2\nonumber \\
        &\quad + \frac{1}{2\gamma_{k+1}}\|y^{k}-y^\star\|^2 - \frac{1}{2\gamma_{k+1}}\|y^{k+1}-y^\star\|^2 -  \frac{1}{2\gamma_{k+1}}\|y^{k+1}-y^{k}\|^2-\frac{2(1-\ell_{y,k}\gamma_k)}{\gamma_k}\|y^k-y^{k-1}\|^2\nonumber \\
        &\quad +\frac{1}{\gamma_k}\langle F_k(x^{k-1})-F_k(x^k) , x^k-x^{k+1} \rangle +\frac{1}{\gamma_k}\langle G_{k}(y^{k-1})-G_k(y^k) , y^k-y^{k+1} \rangle\nonumber \\
        &\quad+ \langle u^{k+1} - u^\star , A(x^\star-x^{k+1})+B(y^{\star}-y^{k+1}) \rangle  +2\langle A^\top u^{k} , x^{k-1}- x^{k} \rangle+2\langle B^\top u^{k} , y^{k-1}- y^{k} \rangle \nonumber\\
        &\quad -P_k +\langle u^\star -u^{k} , A(x^{k}-x^{k+1})+B(y^{k}-y^{k+1}) \rangle+2 f(x^{k-1}) - 2 f(x^{k})+ 2 g(y^{k-1}) - 2g(y^{k}).\nonumber
    \end{align}  
    Add and subtract $2\langle u^\star, A(x^{k-1}-x^{k})+B(y^{k-1}-y^{k}) \rangle$ and recall that 
    \[
    2P_{k-1} - 2P_{k} = 2f(x^{k-1})-2f(x^k)+2g(y^{k-1})-2g(y^k)+2\langle u^\star , A(x^{k-1}-x^k)+B(y^{k-1}-y^k) \rangle.
    \]
    Then, 
           \begin{align}
        0&\le \frac{1}{2\gamma_{k+1}}\|x^{k}-x^\star\|^2 - \frac{1}{2\gamma_{k+1}}\|x^{k+1}-x^\star\|^2 - \frac{1}{2\gamma_{k+1}}\|x^{k+1}-x^{k}\|^2 -\frac{2(1-\ell_{x,k}\gamma_k)}{\gamma_k}\|x^k-x^{k-1}\|^2\nonumber \\
        &\quad + \frac{1}{2\gamma_{k+1}}\|y^{k}-y^\star\|^2 - \frac{1}{2\gamma_{k+1}}\|y^{k+1}-y^\star\|^2 -  \frac{1}{2\gamma_{k+1}}\|y^{k+1}-y^{k}\|^2-\frac{2(1-\ell_{y,k}\gamma_k)}{\gamma_k}\|y^k-y^{k-1}\|^2\nonumber \\
        &\quad +\frac{1}{\gamma_k}\langle F_k(x^{k-1})-F_k(x^k) , x^k-x^{k+1} \rangle +\frac{1}{\gamma_k}\langle G_{k}(y^{k-1})-G_k(y^k) , y^k-y^{k+1} \rangle\nonumber \\
        &\quad+ \underbrace{\langle u^{k+1} - u^\star , A(x^\star-x^{k+1})+B(y^{\star}-y^{k+1}) \rangle}_{\mathrm{(C)}}  +\langle u^\star -u^{k} , A(x^{k}-x^{k+1})+B(y^{k}-y^{k+1}) \rangle\nonumber\\
        &\quad +2P_{k-1}-3P_k +2\langle u^{k}-u^\star , A(x^{k-1}- x^{k})+B(y^{k-1}-y^k) \rangle.    \label{eq:lemA1-5}
    \end{align} 
    For $\mathrm{(C)}$, we have
\begin{align*}
    \mathrm{(C)} &=  \langle u^{k+1} -u^\star , A(x^\star-x^{k+1})+B(y^{\star}-y^{k+1}) \rangle \\
    & =  -\langle u^{k+1} -u^\star , Ax^{k+1}+By^{k+1}-c \rangle\\
        & =  -\langle u^{k+1} -u^\star , \frac{1}{\sigma\gamma_{k+1}}(u^{k+1}-u^k) + A(x^{k+1}-x^k) + B(y^{k+1}-y^k) - 2A(x^{k}-x^{k-1}) - 2B(y^{k}-y^{k-1}) \rangle \\
        & = -\frac{1}{2\sigma\gamma_{k+1}}\|u^{k+1}-u^\star\|^2-\frac{1}{2\sigma\gamma_{k+1}}\|u^{k+1}-u^{k}\|^2+\frac{1}{2\sigma\gamma_{k+1}}\|u^{k}-u^\star\|^2\\
        &\quad + \langle u^{k+1}-u^\star, A(x^{k}-x^{k+1}) + B(y^{k}-y^{k+1})\rangle + 2\langle u^{k+1}-u^\star, A(x^{k}-x^{k-1}) + B(y^{k}-y^{k-1})\rangle
\end{align*}
Plug it into \eqref{eq:lemA1-5} to get 
  \begin{align}
        0&\le \frac{1}{2\gamma_{k+1}}\|x^{k}-x^\star\|^2 - \frac{1}{2\gamma_{k+1}}\|x^{k+1}-x^\star\|^2 - \frac{1}{2\gamma_{k+1}}\|x^{k+1}-x^{k}\|^2 -\frac{2(1-\ell_{x,k}\gamma_k)}{\gamma_k}\|x^k-x^{k-1}\|^2\nonumber \\
        &\quad + \frac{1}{2\gamma_{k+1}}\|y^{k}-y^\star\|^2 - \frac{1}{2\gamma_{k+1}}\|y^{k+1}-y^\star\|^2 -  \frac{1}{2\gamma_{k+1}}\|y^{k+1}-y^{k}\|^2-\frac{2(1-\ell_{y,k}\gamma_k)}{\gamma_k}\|y^k-y^{k-1}\|^2\nonumber \\
          & \quad+\frac{1}{2\sigma\gamma_{k+1}}\|u^{k}-u^\star\|^2-\frac{1}{2\sigma\gamma_{k+1}}\|u^{k+1}-u^\star\|^2-\frac{1}{2\sigma\gamma_{k+1}}\|u^{k+1}-u^{k}\|^2 \nonumber \\
        &\quad +\frac{1}{\gamma_k}\langle F_k(x^{k-1})-F_{k}(x^k) , x^k-x^{k+1} \rangle +\frac{1}{\gamma_k}\langle G_{k}(y^{k-1})-G_k(y^k) , y^k-y^{k+1} \rangle\nonumber \\
        &\quad+ \langle u^{k+1} -u^{k} , A(x^{k}-x^{k+1})+B(y^{k}-y^{k+1}) \rangle\nonumber\\
        &\quad +2P_{k-1}-3P_k +2\langle u^{k}-u^{k+1} , A(x^{k-1}- x^{k})+B(y^{k-1}-y^k) \rangle.    \nonumber
    \end{align} 
Now apply Lemma~\ref{lem:1} to get 
  \begin{align}
         0&\le \frac{1}{2\gamma_{k+1}}\|x^{k}-x^\star\|^2 - \frac{1}{2\gamma_{k+1}}\|x^{k+1}-x^\star\|^2 - \frac{1}{2\gamma_{k+1}}\|x^{k+1}-x^{k}\|^2 -\frac{2(1-\ell_{x,k}\gamma_k)}{\gamma_k}\|x^k-x^{k-1}\|^2\nonumber \\
        &\quad + \frac{1}{2\gamma_{k+1}}\|y^{k}-y^\star\|^2 - \frac{1}{2\gamma_{k+1}}\|y^{k+1}-y^\star\|^2 -  \frac{1}{2\gamma_{k+1}}\|y^{k+1}-y^{k}\|^2-\frac{2(1-\ell_{y,k}\gamma_k)}{\gamma_k}\|y^k-y^{k-1}\|^2\nonumber \\
          & \quad+\frac{1}{2\sigma\gamma_{k+1}}\|u^{k}-u^\star\|^2-\frac{1}{2\sigma\gamma_{k+1}}\|u^{k+1}-u^\star\|^2-\frac{1}{2\sigma\gamma_{k+1}}\|u^{k+1}-u^{k}\|^2 \nonumber \\
        &\quad + \gamma_{k+1}\left\|\frac{1}{\gamma_k}\left(F_k(x^{k-1})-F_k(x^{k})\right)+ A^\top (u^{k+1}-u^{k})\right\|^2. \nonumber \\
        &\quad+ \gamma_{k+1}\left\|\frac{1}{\gamma_k}\left( G_k(y^{k-1})-G_k(y^{k})\right)+ B^\top (u^{k+1}-u^{k})\right\|^2.\nonumber \\
        &\quad +2P_{k-1}-3P_k +2\langle u^{k}-u^{k+1} , A(x^{k-1}- x^{k})+B(y^{k-1}-y^k) \rangle.    \nonumber
    \end{align} 
    Apply Lemma~\ref{lem:2} with $a =\frac{4}{3}$ and $b=4$.
    \begin{align}
        0&\le \frac{1}{2\gamma_{k+1}}\|x^{k}-x^\star\|^2 - \frac{1}{2\gamma_{k+1}}\|x^{k+1}-x^\star\|^2 - \frac{1}{2\gamma_{k+1}}\|x^{k+1}-x^{k}\|^2 -\frac{2(1-\ell_{x,k}\gamma_k)}{\gamma_k}\|x^k-x^{k-1}\|^2\nonumber \\
        &\quad + \frac{1}{2\gamma_{k+1}}\|y^{k}-y^\star\|^2 - \frac{1}{2\gamma_{k+1}}\|y^{k+1}-y^\star\|^2 -  \frac{1}{2\gamma_{k+1}}\|y^{k+1}-y^{k}\|^2-\frac{2(1-\ell_{y,k}\gamma_k)}{\gamma_k}\|y^k-y^{k-1}\|^2\nonumber \\
          & \quad+\frac{1}{2\sigma\gamma_{k+1}}\|u^{k}-u^\star\|^2-\frac{1}{2\sigma\gamma_{k+1}}\|u^{k+1}-u^\star\|^2-\frac{1}{2\sigma\gamma_{k+1}}\|u^{k+1}-u^{k}\|^2 \nonumber \\
        &\quad +\frac{4\rho_{k+1}}{3\gamma_k}\|F_k(x^{k-1})-F_k(x^{k})\|^2+ 4\gamma_{k+1}a_{k+1}^2\|u^{k+1}-u^k\|^2 \nonumber \\
  &\quad +\frac{4\rho_{k+1}}{3\gamma_k}\|G_k(y^{k-1})-G_k(y^{k})\|^2+ 4\gamma_{k+1}b_{k+1}^2\|u^{k+1}-u^k\|^2 \nonumber \\
        &\quad +2P_{k-1}-3P_k +2\langle u^{k}-u^{k+1} , A(x^{k-1}- x^{k})+B(y^{k-1}-y^k) \rangle.    \nonumber
    \end{align} 
    Apply Lemma~\ref{lem:3} to get
        \begin{align}
        0&\le \frac{1}{2\gamma_{k+1}}\|x^{k}-x^\star\|^2 - \frac{1}{2\gamma_{k+1}}\|x^{k+1}-x^\star\|^2 - \frac{1}{2\gamma_{k+1}}\|x^{k+1}-x^{k}\|^2 -\frac{2(1-\ell_{x,k}\gamma_k)}{\gamma_k}\|x^k-x^{k-1}\|^2\nonumber \\
        &\quad + \frac{1}{2\gamma_{k+1}}\|y^{k}-y^\star\|^2 - \frac{1}{2\gamma_{k+1}}\|y^{k+1}-y^\star\|^2 -  \frac{1}{2\gamma_{k+1}}\|y^{k+1}-y^{k}\|^2-\frac{2(1-\ell_{y,k}\gamma_k)}{\gamma_k}\|y^k-y^{k-1}\|^2\nonumber \\
          & \quad+\frac{1}{2\sigma\gamma_{k+1}}\|u^{k}-u^\star\|^2-\frac{1}{2\sigma\gamma_{k+1}}\|u^{k+1}-u^\star\|^2-\frac{1}{2\sigma\gamma_{k+1}}\|u^{k+1}-u^{k}\|^2 \nonumber \\
        &\quad +\frac{4\rho_{k+1}(\gamma_{k}^2L_{x,k}^2 -2\gamma_{k}\ell_{x,k}+1)}{3\gamma_k}\|x^{k-1}-x^{k}\|^2+ 4\gamma_{k+1}a_{k+1}^2\|u^{k+1}-u^k\|^2 \nonumber \\
  &\quad +\frac{4\rho_{k+1}(\gamma_{k}^2L_{y,k}^2 -2\gamma_{k}\ell_{y,k}+1)}{3\gamma_k}\|y^{k-1}-y^{k}\|^2+ 4\gamma_{k+1}b_{k+1}^2\|u^{k+1}-u^k\|^2 \nonumber \\
        &\quad +2P_{k-1}-3P_k +2\langle u^{k}-u^{k+1} , A(x^{k-1}- x^{k})+B(y^{k-1}-y^k) \rangle.    \nonumber
    \end{align} 
    Apply Lemma~\ref{lem:4} to get
       \begin{align}
        0&\le \frac{1}{2\gamma_{k+1}}\|x^{k}-x^\star\|^2 - \frac{1}{2\gamma_{k+1}}\|x^{k+1}-x^\star\|^2 - \frac{1}{2\gamma_{k+1}}\|x^{k+1}-x^{k}\|^2 -\frac{2(1-\ell_{x,k}\gamma_k)}{\gamma_k}\|x^k-x^{k-1}\|^2\nonumber \\
        &\quad + \frac{1}{2\gamma_{k+1}}\|y^{k}-y^\star\|^2 - \frac{1}{2\gamma_{k+1}}\|y^{k+1}-y^\star\|^2 -  \frac{1}{2\gamma_{k+1}}\|y^{k+1}-y^{k}\|^2-\frac{2(1-\ell_{y,k}\gamma_k)}{\gamma_k}\|y^k-y^{k-1}\|^2\nonumber \\
          & \quad+\frac{1}{2\sigma\gamma_{k+1}}\|u^{k}-u^\star\|^2-\frac{1}{2\sigma\gamma_{k+1}}\|u^{k+1}-u^\star\|^2-\frac{1}{2\sigma\gamma_{k+1}}\|u^{k+1}-u^{k}\|^2 \nonumber \\
        &\quad +\frac{4\rho_{k+1}(\gamma_{k}^2L_{x,k}^2 -2\gamma_{k}\ell_{x,k}+1)}{3\gamma_k}\|x^{k-1}-x^{k}\|^2+ 4\gamma_{k+1}a_{k+1}^2\|u^{k+1}-u^k\|^2 \nonumber \\
  &\quad +\frac{4\rho_{k+1}(\gamma_{k}^2L_{y,k}^2 -2\gamma_{k}\ell_{y,k}+1)}{3\gamma_k}\|y^{k-1}-y^{k}\|^2+ 4\gamma_{k+1}b_{k+1}^2\|u^{k+1}-u^k\|^2 \nonumber \\
        &\quad +2P_{k-1}-3P_k +\frac{\lambda_{k+1}^{A}}{8\sigma\gamma_{k+1}}\|u^{k+1}-u^k\|^2  +8\sigma a_{k+1}^2\gamma_{k+1}\lambda^{A}_{k+1}\|x^{k}-x^{k-1}\|^2.    \nonumber \\
        &\quad + \frac{\lambda_{k+1}^{B}}{8\sigma\gamma_{k+1}}\|u^{k+1}-u^k\|^2  +8\sigma b_{k+1}^2\gamma_{k+1}\lambda^{B}_{k+1}\|y^{k}-y^{k-1}\|^2. \label{eq:lemA1-6}
    \end{align} 
    Now by the stepsize update rule $\frac{4\rho_{k+1}}{3}\le 2$, therefore
    \[
    \frac{4\rho_{k+1}(\gamma_{k}^2L_{x,k}^2 -2\gamma_{k}\ell_{x,k}+1)}{3\gamma_k} - \frac{2}{\gamma_k} \le  \frac{4\rho_{k+1}(\gamma_{k}^2L_{x,k}^2 -2\gamma_{k}\ell_{x,k})}{3\gamma_k} = \frac{4\rho_{k+1}\delta_{x,k}}{3\gamma_{k}}.
    \]
    and
    \[
    \frac{4\rho_{k+1}(\gamma_{k}^2L_{y,k}^2 -2\gamma_{k}\ell_{y,k}+1)}{3\gamma_k} - \frac{2}{\gamma_k} \le  \frac{4\rho_{k+1}(\gamma_{k}^2L_{y,k}^2 -2\gamma_{k}\ell_{y,k})}{3\gamma_k} = \frac{4\rho_{k+1}\delta_{y,k}}{3\gamma_{k}}.
    \]
    Using this result, we get the following from \eqref{eq:lemA1-6}
     \begin{align}
        0&\le \frac{1}{2\gamma_{k+1}}\|x^{k}-x^\star\|^2 - \frac{1}{2\gamma_{k+1}}\|x^{k+1}-x^\star\|^2 + \frac{1}{2\gamma_{k+1}}\|y^{k}-y^\star\|^2 - \frac{1}{2\gamma_{k+1}}\|y^{k+1}-y^\star\|^2 \nonumber \\
          & \quad+\frac{1}{2\sigma\gamma_{k+1}}\|u^{k}-u^\star\|^2-\frac{1}{2\sigma\gamma_{k+1}}\|u^{k+1}-u^\star\|^2 \nonumber \\
        &\quad -\frac{1}{2\gamma_{k+1}}\|x^{k+1}-x^k\|^2+\left(\frac{4\rho_{k+1}\delta_{x,k}}{3\gamma_k}+2\ell_{x,k}+8\sigma a_{k+1}^2\gamma_{k+1}\lambda^{A}_{k+1}\right)\|x^{k-1}-x^{k}\|^2 \nonumber \\
  &\quad  -\frac{1}{2\gamma_{k+1}}\|y^{k+1}-y^k\|^2+\left(\frac{4\rho_{k+1}\delta_{y,k}}{3\gamma_k}+2\ell_{y,k}+8\sigma b_{k+1}^2\gamma_{k+1}\lambda^{B}_{k+1}\right)\|y^{k-1}-y^{k}\|^2 \nonumber \\
  & \quad -\left( \frac{1}{2\sigma\gamma_{k+1}} - \frac{\lambda_{k+1}^{A}}{8\sigma\gamma_{k+1}}-\frac{\lambda_{k+1}^{B}}{8\sigma\gamma_{k+1}}-4\gamma_{k+1}a_{k+1}^2-4\gamma_{k+1}b_{k+1}^2\right) \|u^{k+1}-u^k\|^2+2P_{k-1}-3P_k . \nonumber
    \end{align} 
    Finally, multiply by $\gamma_{k+1}$ and using $\rho_{k+1}=\frac{\gamma_{k+1}}{\gamma_k}$
         \begin{align}
        0&\le \frac{1}{2}\|x^{k}-x^\star\|^2 - \frac{1}{2}\|x^{k+1}-x^\star\|^2+ \frac{1}{2}\|y^{k}-y^\star\|^2 - \frac{1}{2}\|y^{k+1}-y^\star\|^2 +\frac{1}{2\sigma}\|u^{k}-u^\star\|^2-\frac{1}{2\sigma}\|u^{k+1}-u^\star\|^2 \nonumber \\
        &\quad - \frac{1}{2}\|x^{k+1}-x^k\|^2+\left(\frac{4\rho_{k+1}^2\delta_{x,k}}{3}+2\gamma_k\ell_{x,k}\rho_{k+1}+ 8\sigma a_{k+1}^2\gamma_{k+1}^2\lambda^{A}_{k+1}\right)\|x^{k-1}-x^{k}\|^2\nonumber \\
  &\quad - \frac{1}{2}\|y^{k+1}-y^k\|^2+\left(\frac{4\rho_{k+1}^2\delta_{y,k}}{3}+2\gamma_k\ell_{y,k}\rho_{k+1} +8\sigma b_{k+1}^2\gamma_{k+1}^2\lambda^{B}_{k+1}\right)\|y^{k-1}-y^{k}\|^2 \nonumber \\
  & \quad -\left(\frac{1}{2\sigma} - \frac{\lambda_{k+1}^{A}+\lambda_{k+1}^{B}}{8\sigma}-4\gamma_{k+1}^2(a_{k+1}^2+b_{k+1}^2)\right) \|u^{k+1}-u^k\|^2+2\gamma_{k+1}P_{k-1}-3\gamma_{k+1}P_k . \label{lem:A1-7}
    \end{align} 

Here, \eqref{lem:A1-7} provides a one-step descent inequality. We now exploit this inequality with choosing appropriate $\{\gamma_k\}_{k \ge 0 }$ to derive the monotonicity of the Lyapunov sequence $\{\mathcal{U}_k\}_{k\ge1}$. Recall the Lyapunov sequence: 
\begin{align*}
\mathcal{U}_k &= \frac{1}{2}\|x^{k}-x^\star\|^2 + \frac{1}{2}\|x^{k}-x^{k-1}\|^2 + \frac{1}{2}\|y^{k}-y^\star\|^2 +  \frac{1}{2}\|y^{k}-y^{k-1}\|^2 \nonumber \\
&+\frac{1}{2\sigma}\|u^{k}-u^\star\|^2 + 3\gamma_{k}P_{k-1}, 
\end{align*}
Expressing \eqref{lem:A1-7} in terms of $\mathcal{U}_k$ and  $\mathcal{U}_{k+1}$, we get
\begin{align*}
    \mathcal{U}_{k+1} &\le \mathcal{U}_k-\left(\frac{1}{2}-\frac{4\rho_{k+1}^2\delta_{x,k}}{3} -2\gamma_k\ell_{x,k}\rho_{k+1}-8\sigma a_{k+1}^2\gamma_{k+1}^2\lambda^{A}_{k+1}\right)\|x^{k}-x^{k-1}\|^2\\
    &\quad -\left(\frac{1}{2}-\frac{4\rho_{k+1}^2\delta_{y,k}}{3}- 2\gamma_k\ell_{y,k}\rho_{k+1}-8 \sigma b_{k+1}^2\gamma_{k+1}^2\lambda^{B}_{k+1}\right)\|y^{k}-y^{k-1}\|^2\\
    &\quad -\left(\frac{1}{2\sigma} - \frac{\lambda_{k+1}^{A}+\lambda_{k+1}^{B}}{8\sigma}-4\gamma_{k+1}^2(a_{k+1}^2+b_{k+1}^2)\right) \|u^{k+1}-u^k\|^2\\
    &\quad -(3\gamma_k - 2\gamma_{k+1}) P_{k-1} .
\end{align*}
Hence, it is enough to show
\[
\frac{1}{2}-\frac{4\rho_{k+1}^2\delta_{x,k}}{3}-2\gamma_k\ell_{x,k}\rho_{k+1} -8\sigma a_{k+1}^2\gamma_{k+1}^2\lambda^{A}_{k+1} \ge \varepsilon ,
\]
\[
\frac{1}{2}-\frac{4\rho_{k+1}^2\delta_{y,k}}{3}- 2\gamma_k\ell_{y,k}\rho_{k+1} -8\sigma b_{k+1}^2\gamma_{k+1}^2\lambda^{B}_{k+1} \ge \varepsilon,
\]
\[
\frac{1}{2\sigma} - \frac{\lambda_{k+1}^{A}+\lambda_{k+1}^{B}}{8\sigma}-4\gamma_{k+1}^2(a_{k+1}^2+b_{k+1}^2) \ge \varepsilon .
\]
Rearrange the first two inequalities as
\begin{equation}\label{eq:quadx}(\delta_{x,k}+6\sigma a_{k+1}^2\gamma_k^2\lambda_{k+1}^{A})\rho_{k+1}^2+\frac{3}{2}\gamma_k\ell_{x,k}\rho_{k+1}-\left(\frac{3-6\varepsilon}{8}\right) \le 0 ,
\end{equation}
and
\begin{equation}\label{eq:quady}
(\delta_{y,k}+6\sigma b_{k+1}^2\gamma_k^2\lambda_{k+1}^{B})\rho_{k+1}^2+\frac{3}{2}\gamma_k\ell_{y,k}\rho_{k+1}-\left(\frac{3-6\varepsilon}{8}\right) \le 0 .
\end{equation}
If $a_{k+1}^2+b_{k+1}^2=0$, the third inequality always holds. Otherwise, it is equivalent to 

\[
\gamma_{k+1} \le \sqrt{\frac{4-\lambda_{k+1}^{A}-\lambda_{k+1}^{B}-8\sigma\varepsilon}{32\sigma(a_{k+1}^2+b_{k+1}^2)}}.
\]
We observe that, irrespective of the sign of the leading coefficient, \eqref{eq:quadx} and \eqref{eq:quady} are satisfied by any $\rho_{k+1}$ that satisfies the following inequality if the square roots are real-valued:
\[
\rho_{k+1} \le \frac{3-6\varepsilon}{4}\cdot\frac{1}{\frac{3}{2}\gamma_k\ell_{x,k}+\sqrt{(\frac{3}{2}\gamma_k\ell_{x,k})^2 +\frac{3-6\varepsilon}{2}\left(\delta_{x,k} + 6\sigma a_{k+1}^2\gamma_k^2\lambda_{k+1}^{A}\right)}},
\]
and
\[
\rho_{k+1} \le \frac{3-6\varepsilon}{4}\cdot\frac{1}{\frac{3}{2}\gamma_k\ell_{y,k}+\sqrt{(\frac{3}{2}\gamma_k\ell_{y,k})^2 +\frac{3-6\varepsilon}{2}\left(\delta_{y,k} + 6\sigma b_{k+1}^2\gamma_k^2\lambda_{k+1}^{B}\right)}}.
\]
If the square roots are not real-valued, then both inequalities are satisfied independently of $\rho_{k+1}$.

Now note that these inequalities are equivalent to the defining terms of $\gamma_{k+1}$:
\[
\gamma_{k+1} \le \frac{1-2\varepsilon}{2}\cdot\frac{\gamma_k}{\gamma_k\ell_{x,k}+\sqrt{(\gamma_k\ell_{x,k})^2 +\frac{2-4\varepsilon}{3}\left(\delta_{x,k} + 6\sigma a_{k+1}^2\gamma_k^2\lambda_{k+1}^{A}\right)}},
\]

\[
\gamma_{k+1} \le \frac{1-2\varepsilon}{2}\cdot\frac{\gamma_k}{\gamma_k\ell_{y,k}+\sqrt{(\gamma_k\ell_{y,k})^2 +\frac{2-4\varepsilon}{3}\left(\delta_{y,k} + 6\sigma b_{k+1}^2\gamma_k^2\lambda_{k+1}^{B}\right)}}.
\]
Hence, we have the nonincreasing property if 
 \[
                          \gamma_{k+1} = \min\left\{ \frac{3}{2}\gamma_{k} , \quad \sqrt{\frac{4-\lambda^{A}_{k+1}-\lambda^{B}_{k+1}-8\sigma\varepsilon}{32\sigma(a_{k+1}^2+b^2_{k+1})}}, \quad \Gamma_x , \quad \Gamma_y\right\}.
\]
If the square root that defines $\Gamma_x$ or $\Gamma_y$ is not real, then \eqref{eq:quadx} or \eqref{eq:quady} holds for every 
$\rho_{k+1}$, and then the nonincreasing property is guaranteed without assuming $\gamma_{k+1}\le \Gamma_x$ or $\gamma_{k+1}\le \Gamma_y$. In other words, it is equivalent to taking $\Gamma_x=+\infty$ or $\Gamma_y=+\infty$. This completes the verification that, under the stepsize rule defining $\gamma_{k+1}$, all the required inequalities hold simultaneously. Consequently, the Lyapunov sequence $\{\mathcal{U}_k\}_{k\ge1}$ is nonincreasing, as claimed. \qed

\paragraph{Proof of Lemma~\ref{lem:gamma1}.} The proof starts with the following lemma:
\begin{lemma}\label{lem:7}
Let $L_C>0$ is a local smoothness constant of $f_2$ with respect to a compact convex set $C$. Then for any $x^k,x^{k-1}\in C$, the following holds 
\[
0\le\ell_{x,k} \le L_{x,k} \le L_C.
\]
\end{lemma}
\begin{proof}
The convexity of $f_2$ implies $\ell_{x,k}$ is nonnegative~\cite[Section~17]{bauschke2017convex}. The rest of the proof is by Cauchy -- Schwarz inequality:
    \begin{align*}
        \ell_{x,k}& =  \frac{\langle \nabla f_2(x^{k-1})-\nabla f_2(x^{k}), x^{k-1}-x^k \rangle}{\|x^{k-1}-x^k\|^2}\\
        &\le\frac{\|\nabla f_2(x^{k-1})-\nabla f_2(x^{k})\|\| x^{k-1}-x^k \|}{\|x^{k-1}-x^k\|^2}\\
        &\le\frac{\|\nabla f_2(x^{k-1})-\nabla f_2(x^{k})\|}{\|x^{k-1}-x^k\|}\\
        &=L_{x,k}\le L_C.
    \end{align*}\qed
\end{proof}
By the exact same argument, we get an analogous lemma for $y$-iterates. Observe that $\{x^k\}_{k\ge0}$, $\{y^k\}_{k\ge0}$, and $\{u^k\}_{k\ge0}$ are bounded since
\[
\frac{1}{2}\|x^{k}-x^\star\|^2 + \frac{1}{2}\|y^{k}-y^\star\|^2 +\frac{1}{2\sigma}\|u^{k}-u^\star\|^2  \le \mathcal{U}_k \le \cdots \le \mathcal{U}_1.
\]
For sufficiently large compact sets $U\subseteq\mathbb{R}^p$ and $V\subseteq\mathbb{R}^q$ that contain $\{x^k\}_{k\ge0}$ and $\{y^k\}_{k\ge0}$, denote $L_U>0$ and $L_V>0$ as the corresponding local smoothness constants for $f_2$ and $g_2$ respectively.
Now we will show that the last three terms in the $\mathrm{min}$ operator that defines $\gamma_{k+1}$ are bounded away from zero. The first among the three terms is bounded as follows:
\[
\sqrt{\frac{4-\lambda_{k+1}^{A}-\lambda_{k+1}^{B}-8\sigma\varepsilon}{32\sigma(a_{k+1}^2+b_{k+1}^2)}} \ge \sqrt{\frac{2-8\sigma\varepsilon}{32\sigma(\|A\|^2+\|B\|^2)}},
\]
where $\|A\|$ and $\|B\|$ are operator norms of $A$ and $B$ respectively. 
For $\Gamma_x$ with real-valued square root, 
\begin{align*}
    \Gamma_x &= \frac{1-2\varepsilon}{2}\cdot\frac{\gamma_k}{\gamma_k\ell_{x,k}+\sqrt{(\gamma_k\ell_{x,k})^2 +\frac{2-4\varepsilon}{3}\left(\gamma_k^2L_{x,k}^2-2\ell_{x,k}\gamma_k + 6\sigma a_{k+1}^2\gamma_k^2\lambda_{k+1}^{A}\right)}}\\
     &\ge \frac{1-2\varepsilon}{2}\cdot\frac{\gamma_k}{\gamma_k\ell_{x,k}+\sqrt{(\gamma_k\ell_{x,k})^2 +\frac{2-4\varepsilon}{3}\left(\gamma_k^2L_{x,k}^2 + 6\sigma a_{k+1}^2\gamma_k^2\lambda_{k+1}^{A}\right)}}\\
      &\ge \frac{1-2\varepsilon}{2}\cdot\frac{1}{\ell_{x,k}+\sqrt{(\ell_{x,k})^2 +\frac{2-4\varepsilon}{3}\left(L_{x,k}^2 + 6\sigma a_{k+1}^2\lambda_{k+1}^{A}\right)}}\\
       &\ge\frac{1-2\varepsilon}{2}\cdot\frac{1}{L_U+\sqrt{L_U^2 +\frac{2-4\varepsilon}{3}\left(L_U^2 + 6\sigma \|A\|^2\right)}}:=\gamma_x,
\end{align*}
where the last inequality follows from $\ell_{x,k}\le L_U$ from Lemma~\ref{lem:7}.
Similarly if $\gamma_{k+1} = \Gamma_y $, then 
\begin{align*}
    \Gamma_{y} &\ge \frac{1-2\varepsilon}{2}\cdot\frac{1}{L_V+\sqrt{L_V^2 +\frac{2-4\varepsilon}{3}\left(L_V^2 + 6\sigma\|B\|^2\right)}}:=\gamma_y.
\end{align*}
Hence, from the simple induction argument, we have
\[
\gamma_{k} \ge \min \left\{ \gamma_0 ,\quad \sqrt{\frac{2-8\sigma\varepsilon}{32\sigma(\|A\|^2+\|B\|^2)}},\quad \gamma_x ,\quad \gamma_y \right\}:=\gamma>0
\]
for any $k\ge0$. \qed

\subsection{Convergence result and proof of Theorem~\ref{thm:1}}\label{sec:3-2}
\paragraph{Proof of Theorem~\ref{thm:1}.}
Since the sequence is bounded, there exists a convergent subsequence with limit $(x^\infty, y^\infty, u^\infty)$ since $\{(x^k,y^k,u^k)\}_{k=0,1,\cdot}$ is bounded in $\mathbb{R}^p\times\mathbb{R}^q\times\mathbb{R}^{r}$. The telescoping sum argument gives 
\[
0\le\mathcal{U}_{k+1}\le \mathcal{U}_1 - \varepsilon\sum_{i=1}^{k}\|x^i-x^{i-1}\|^2  - \varepsilon\sum_{i=1}^{k}\|y^i-y^{i-1}\|^2-\varepsilon\sum_{i=1}^{k}\|u^i-u^{i-1}\|^2.
\]
By letting $k\to\infty$, we have $\|x^k-x^{k-1}\|$, $\|y^k-y^{k-1}\|$, and $\|u^k-u^{k-1}\|$ converge to $0$. By the definition of $x^{k+1}$,$y^{k+1}$ and $u^{k+1}$, we have
\begin{align*}
0 &\in \partial f_1(x^{k+1}) + \nabla f_2 (x^{k}) + A^\top u^{k+1} + \frac{1}{\gamma_{k+1}}\left( x^{k+1}-x^{k}\right),
\end{align*}
\begin{align*}
0 &\in \partial g_1(y^{k+1}) + \nabla g_2 (y^{k}) + B^\top u^{k+1} + \frac{1}{\gamma_{k+1}}\left( y^{k+1}-y^{k}\right)
\end{align*}
and
\begin{align*}
0=-\frac{u^{k+1}-u^k}{\sigma\gamma_{k+1}}+(Ax^k + By^k - c) + 2\left(A(x^k-x^{k-1})+B(y^k-y^{k-1})\right)
\end{align*}
By passing them to the limit along the subsequences that converge to $(x^\infty, y^\infty, u^\infty)$, we get
\begin{align*}
0 &\in \partial f_1(x^{\infty}) + \nabla f_2 (x^{\infty}) + A^\top u^{\infty}  = \partial f(x^\infty) +A^\top u^{\infty},
\end{align*}
\begin{align*}
0 &\in \partial g_1(y^{\infty}) + \nabla g_2 (y^{\infty}) + B^\top u^{\infty} = \partial g(y^\infty) + B^\top u^{\infty},
\end{align*}
and
\begin{align*}
0= Ax^{\infty} + By^{\infty} - c.
\end{align*}
Note that the results follow from $\{\gamma_k\}_{k=0,1,\cdots}$ being bounded away from zero. Thus, the limit point $(x^\infty, y^\infty, u^\infty)$ is a saddle point. For uniqueness, assume $(\tilde{x}^\infty, \tilde{y}^\infty, \tilde{u}^\infty)$ is another limit point, which is a saddle point. Let $\{\mathcal{U}_k\}$ and  $\{\tilde{\mathcal{U}}_k\}$ be sequences defined in Lemma~\ref{lem:A1} by their respective saddle points $(x^\infty, y^\infty, u^\infty)$ and $(\tilde{x}^\infty, \tilde{y}^\infty, \tilde{u}^\infty)$. Then,
\begin{align*}
\lim_{k\to \infty} \left(\mathcal{U}_k-\tilde{\mathcal{U}}_k \right)&= \frac{1}{2}\|x^\infty\|^2-\frac{1}{2}\|\tilde{x}^\infty\|^2 +\frac{1}{2}\|y^\infty\|^2-\frac{1}{2}\|\tilde{y}^\infty\|^2+\frac{1}{2\sigma}\|u^\infty\|^2-\frac{1}{2\sigma}\|\tilde{u}^\infty\|^2 \\
&\quad-\lim_{k\to \infty} \left(\langle x^k, x^\infty-\tilde{x}^\infty\rangle+\langle y^k, y^\infty-\tilde{y}^\infty\rangle+\frac{1}{\sigma}\langle u^k, u^\infty-\tilde{u}^\infty\rangle\right).
\end{align*}
Recall that $\mathcal{U}_k-\tilde{\mathcal{U}}_k$ is convergent. Passing $k$ to the limit along two subsequences $K$ and $\tilde{K}$ that converge to $(x^\infty, y^\infty, u^\infty)$ and $(\tilde{x}^\infty, \tilde{y}^\infty, \tilde{u}^\infty)$, we must have 
\begin{align*}
\langle x^\infty-\tilde{x}^\infty, x^\infty-\tilde{x}^\infty\rangle+\langle y^\infty-\tilde{y}^\infty, y^\infty-\tilde{y}^\infty\rangle+\frac{1}{\sigma}\langle u^\infty-\tilde{u}^\infty, u^\infty-\tilde{u}^\infty\rangle = 0.
\end{align*}
This implies 
\[
\|x^\infty-\tilde{x}^\infty\|^2+\|y^\infty-\tilde{y}^\infty\|^2+\frac{1}{\sigma}\|u^\infty-\tilde{u}^\infty\|^2=0
\]
and thus $(x^\infty, y^\infty, u^\infty)=(\tilde{x}^\infty, \tilde{y}^\infty, \tilde{u}^\infty)$. So, the limit point is unique and $(x^k,y^k,u^k)$ converges. \qed

\paragraph{Convergence rate.}
We now briefly discuss the convergence rate. Then the telescopic sum argument also gives
\[
3\gamma_{K+1}P_K + \sum_{k=1}^{K}\gamma_k \left(3-2\rho_{k+1}\right)P_{k-1} \le \mathcal{U}_1.
\]
Then,
\begin{align*}
\min_{k=0,\cdots,K}P_k &\le \frac{\mathcal{U}_1}{3\gamma_{K+1}+\sum_{k=1}^{K}\gamma_k \left(3-2\rho_{k+1}\right)}\\
&=\frac{\mathcal{U}_1}{3\gamma_{K+1}+\sum_{k=1}^{K}\gamma_k+\sum_{k=1}^{K} \left(2\gamma_k-2\gamma_{k+1}\right)}\\
&=\frac{\mathcal{U}_1}{2\gamma_1+\sum_{k=1}^{K}\gamma_k+\gamma_{K+1}}\\
&\le\frac{\mathcal{U}_1}{(K+3)\gamma}.
\end{align*}

\subsection{Proof of Theorem~\ref{thm:2}}
From this subsection, we prove the analogous result for Subroutine~\ref{alg:S2}. Let $\varphi=\frac{1+\sqrt{5}}{2}$, the golden ratio. Then, corresponding descent lemma for Subroutine~\ref{alg:S2} is as follows: 
\begin{lemma} \label{lem:A3}
Consider a sequence $\{(x^k, y^k, u^k)\}_{k=0,1,2,\dots}$ generated by \ALiA\ with Subroutine~\ref{alg:S2}. For $k=1,2,\dots$, let
\begin{align*}
\mathcal{V}_k &= \frac{1}{2}\|x^{k}-x^\star\|^2 + \frac{1}{2}\|x^{k}-x^{k-1}\|^2 + \frac{1}{2}\|y^{k}-y^\star\|^2 +  \frac{1}{2}\|y^{k}-y^{k-1}\|^2 \nonumber \\
&+\frac{1}{2\sigma}\|u^{k}-u^\star\|^2 + (1+\varphi)\gamma_{k}P_{k-1}, 
\end{align*}
where $(x^\star,y^\star,u^\star)$ is a saddle point of the Lagrangian. Then, following holds:
\begin{align}\label{eq:lya2}
    \mathcal{V}_{k+1} \le \mathcal{V}_{k} -\varepsilon\|x^k-x^{k-1}\|^2 -\varepsilon\|y^k-y^{k-1}\|^2 -\varepsilon\|u^k-u^{k-1}\|^2-\underbrace{((1+\varphi)\gamma_k - \varphi\gamma_{k+1})}_{\ge 0} P_{k-1}. 
\end{align}
\end{lemma}

The majority of the proof proceeds in the same way, but we need a few additional lemmas to handle the cross terms that arise only in the analysis of Subroutine~\ref{alg:S2}. The following lemma is almost identical to Lemma~\ref{lem:4}, differing only in the choice of coefficients.

\begin{lemma}\label{lem:4-3}
Let 
\[
\lambda^{A}_{k+1}=\frac{\langle A^\top \Delta u^{k+1},x^{k}- x^{k-1} \rangle }{\frac{\| A^\top \Delta u^{k+1}\|^2}{8\varphi a_{k+1}^2}+2\varphi a_{k+1}^2\| x^{k}- x^{k-1}\|^2} , \quad \lambda^{B}_{k+1}=\frac{\langle B^\top \Delta u^{k+1},y^{k}- y^{k-1} \rangle }{\frac{\| B^\top \Delta u^{k+1}\|^2}{8\varphi b_{k+1}^2}+2\varphi b_{k+1}^2\| y^{k}- y^{k-1}\|^2} 
\]
Then $-1\le \lambda^{A}_{k+1},\lambda^{B}_{k+1}\le1$ and the following holds 
\[
  \langle A^\top(u^{k+1}-u^{k}) , x^{k}- x^{k-1} \rangle = \frac{\lambda_{k+1}^{A}}{8\varphi \sigma\gamma_{k+1}}\|u^{k+1}-u^k\|^2  +2\varphi \sigma a_{k+1}^2\gamma_{k+1}\lambda^{A}_{k+1}\|x^{k}-x^{k-1}\|^2,
\]
\[
  \langle B^\top(u^{k+1}-u^{k}) , y^{k}- y^{k-1} \rangle = \frac{\lambda_{k+1}^{B}}{8\varphi \sigma\gamma_{k+1}}\|u^{k+1}-u^k\|^2  +2\varphi \sigma b_{k+1}^2\gamma_{k+1}\lambda^{B}_{k+1}\|y^{k}-y^{k-1}\|^2.
\]
\end{lemma}
\begin{proof}
The bounds $-1\le \lambda^{A}_{k+1},\lambda^{B}_{k+1}\le1$ are from Young's inequality. For equalities, 
\begin{align*}
       \langle A^\top(u^{k+1}-u^{k}) , x^{k}- x^{k-1} \rangle &=\sigma\gamma_{k+1}\langle A^\top \Delta u^{k+1} , x^{k}- x^{k-1} \rangle    \\
       &=\frac{\sigma\gamma_{k+1}\lambda^{A}_{k+1}}{8\varphi a_{k+1}^2}\|A^\top \Delta u^{k+1}\|^2 +2\varphi \sigma a_{k+1}^2\gamma_{k+1}\lambda^{A}_{k+1}\|x^{k}-x^{k-1}\|^2   \\
       &=\frac{\lambda_{k+1}^{A}}{8\varphi \sigma\gamma_{k+1}}\|u^{k+1}-u^k\|^2  +2\varphi \sigma a_{k+1}^2\gamma_{k+1}\lambda^{A}_{k+1}\|x^{k}-x^{k-1}\|^2.
\end{align*}
Do the same thing for $y$-iterate. \qed
\end{proof}

The next lemma is analogous to the previous one, but with $\mu^{A}_{k+1}$ and $\mu^{B}_{k+1}$.
\begin{lemma}\label{lem:4-4}
Let 
\[
\mu^{A}_{k+1}=\frac{\langle A^\top \Delta u^{k+1},F_k(x^{k-1}) - F_k(x^{k}) \rangle }{\frac{\gamma_k\| A^\top \Delta u^{k+1}\|^2}{2a_{k+1}^2}+\frac{a_{k+1}^2}{2\gamma_k}\| F_k(x^{k-1}) - F_k(x^{k})\|^2} , \quad \mu^{B}_{k+1}=\frac{\langle B^\top \Delta u^{k+1}, G_k(y^{k-1})- G_k(y^{k}) \rangle }{\frac{\gamma_k\| B^\top \Delta u^{k+1}\|^2}{2b_{k+1}^2}+\frac{b_{k+1}^2}{2\gamma_k}\| G_k(y^{k-1})- G_k(y^{k})\|^2} 
\]
Then $-1\le \mu^{A}_{k+1},\mu^{B}_{k+1}\le1$ and the following holds 
\[
  \langle A^\top(u^{k+1}-u^{k}) , F_k(x^{k-1})- F_k(x^{k}) \rangle =\frac{\mu_{k+1}^{A}}{2\sigma\rho_{k+1}}\|u^{k+1}-u^k\|^2  +\frac{\sigma a_{k+1}^2\rho_{k+1}\mu^{A}_{k+1}}{2}\|F_k(x^{k-1})- F_k(x^{k})\|^2,
\]
\[
  \langle B^\top(u^{k+1}-u^{k}) , G_k(y^{k-1})- G_k(y^{k}) \rangle = \frac{\mu_{k+1}^{B}}{2\sigma\rho_{k+1}}\|u^{k+1}-u^k\|^2  +\frac{\sigma b_{k+1}^2\rho_{k+1}\mu^{B}_{k+1}}{2}\|G_k(y^{k-1})- G_k(y^{k})\|^2.
\]
\end{lemma}
\begin{proof}
The bounds $-1\le \mu^{A}_{k+1},\mu^{B}_{k+1}\le1$ are from Young's inequality. For equalities, 
\begin{align*}
       \langle A^\top(u^{k+1}-u^{k}) , F_k(x^{k-1})- F_k(x^{k}) \rangle &=\sigma\gamma_{k+1}\langle A^\top \Delta u^{k+1} , F_k(x^{k-1})- F_k(x^{k}) \rangle    \\
       &=\frac{\sigma\gamma_k\gamma_{k+1}\mu^{A}_{k+1}}{2a_{k+1}^2}\|A^\top \Delta u^{k+1}\|^2 +\frac{\sigma a_{k+1}^2\gamma_{k+1}\mu^{A}_{k+1}}{2\gamma_k}\|F_k(x^{k-1})- F_k(x^{k})\|^2   \\
       &=\frac{\mu_{k+1}^{A}}{2\sigma\rho_{k+1}}\|u^{k+1}-u^k\|^2  +\frac{\sigma a_{k+1}^2\rho_{k+1}\mu^{A}_{k+1}}{2}\|F_k(x^{k-1})- F_k(x^{k})\|^2.
\end{align*}
Do the same thing for $y$-iterate. \qed
\end{proof}

Now we are ready to prove Lemma~\ref{lem:A3}.

\vspace{0.1in}

\paragraph{Proof of Lemma~\ref{lem:A3}.}
Note that in the proof of Lemma~\ref{lem:A1} we did not invoke any properties specific to the subroutine before inequality \eqref{eq:lemA1-4}, so we will begin our argument from that point.

Multiply \eqref{eq:lemA1-4} by $\varphi$ and add it to \eqref{eq:lemA1-3} to get 
    \begin{align}
        0&\le \frac{1}{2\gamma_{k+1}}\|x^{k}-x^\star\|^2 - \frac{1}{2\gamma_{k+1}}\|x^{k+1}-x^\star\|^2 - \frac{1}{2\gamma_{k+1}}\|x^{k+1}-x^{k}\|^2 -\frac{\varphi(1-\ell_{x,k}\gamma_k)}{\gamma_k}\|x^k-x^{k-1}\|^2\nonumber \\
        &\quad + \frac{1}{2\gamma_{k+1}}\|y^{k}-y^\star\|^2 - \frac{1}{2\gamma_{k+1}}\|y^{k+1}-y^\star\|^2 -  \frac{1}{2\gamma_{k+1}}\|y^{k+1}-y^{k}\|^2-\frac{\varphi(1-\ell_{y,k}\gamma_k)}{\gamma_k}\|y^k-y^{k-1}\|^2\nonumber \\
        &\quad +\frac{1}{\gamma_k}\langle F_k(x^{k-1})-F_k(x^k) , x^k-x^{k+1} \rangle +\langle A^\top u^k, x^{k+1} -x^{k} \rangle \nonumber \\
        &\quad +\frac{1}{\gamma_k}\langle G_{k}(y^{k-1})-G_k(y^k) , y^k-y^{k+1} \rangle+\langle B^\top u^k, y^{k+1} -y^{k} \rangle\nonumber \\
        &\quad+\langle u^{k+1} , A(x^\star-x^{k+1})+B(y^{\star}-y^{k+1}) \rangle  +\varphi\langle A^\top u^{k} , x^{k-1}- x^{k} \rangle+\varphi\langle B^\top u^{k} , y^{k-1}- y^{k} \rangle  \rangle\nonumber\\
        &\quad +f(x^\star)-f(x^{k})+g(y^\star)-g(y^k)+\varphi f(x^{k-1}) - \varphi f(x^{k})+ \varphi g(y^{k-1}) - \varphi g(y^{k}). \nonumber 
    \end{align}
     Add and subtract $\langle u^\star, A(x^\star-x^{k+1})+B(y^\star-y^{k+1}) \rangle$ to get 
       \begin{align}
        0&\le \frac{1}{2\gamma_{k+1}}\|x^{k}-x^\star\|^2 - \frac{1}{2\gamma_{k+1}}\|x^{k+1}-x^\star\|^2 - \frac{1}{2\gamma_{k+1}}\|x^{k+1}-x^{k}\|^2 -\frac{\varphi(1-\ell_{x,k}\gamma_k)}{\gamma_k}\|x^k-x^{k-1}\|^2\nonumber \\
        &\quad + \frac{1}{2\gamma_{k+1}}\|y^{k}-y^\star\|^2 - \frac{1}{2\gamma_{k+1}}\|y^{k+1}-y^\star\|^2 -  \frac{1}{2\gamma_{k+1}}\|y^{k+1}-y^{k}\|^2-\frac{\varphi(1-\ell_{y,k}\gamma_k)}{\gamma_k}\|y^k-y^{k-1}\|^2\nonumber \\
        &\quad +\frac{1}{\gamma_k}\langle F_k(x^{k-1})-F_k(x^k) , x^k-x^{k+1} \rangle +\langle A^\top u^k, x^{k+1} -x^{k} \rangle \nonumber \\
        &\quad +\frac{1}{\gamma_k}\langle G_{k}(y^{k-1})-G_k(y^k) , y^k-y^{k+1} \rangle+\langle B^\top u^k, y^{k+1} -y^{k} \rangle\nonumber \\
        &\quad+ \langle u^{k+1} - u^\star , A(x^\star-x^{k+1})+B(y^{\star}-y^{k+1}) \rangle  +\varphi\langle A^\top u^{k} , x^{k-1}- x^{k} \rangle+\varphi\langle B^\top u^{k} , y^{k-1}- y^{k} \rangle  \rangle\nonumber\\
        &\quad  +\langle u^\star , A(x^\star-x^{k})+B(y^{\star}-y^{k}) \rangle+\langle u^\star , A(x^{k}-x^{k+1})+B(y^{k}-y^{k+1}) \rangle \nonumber \\
        &\quad+f(x^\star)-f(x^{k})+g(y^\star)-g(y^k) +\varphi f(x^{k-1}) - \varphi f(x^{k})+ \varphi g(y^{k-1}) - \varphi g(y^{k}). \nonumber 
    \end{align}  
Recall $P_{k} = \mathbf{L}(x^{k}, y^{k}, u^\star)-\mathbf{L}(x^{\star}, y^\star, u^\star) = f(x^k)+g(y^k)-f(x^\star)-g(y^\star)+\langle u^\star, A(x^k-x^\star)+B(y^k-y^\star) \rangle \ge0$ to get
       \begin{align}
        0&\le \frac{1}{2\gamma_{k+1}}\|x^{k}-x^\star\|^2 - \frac{1}{2\gamma_{k+1}}\|x^{k+1}-x^\star\|^2 - \frac{1}{2\gamma_{k+1}}\|x^{k+1}-x^{k}\|^2 -\frac{\varphi(1-\ell_{x,k}\gamma_k)}{\gamma_k}\|x^k-x^{k-1}\|^2\nonumber \\
        &\quad + \frac{1}{2\gamma_{k+1}}\|y^{k}-y^\star\|^2 - \frac{1}{2\gamma_{k+1}}\|y^{k+1}-y^\star\|^2 -  \frac{1}{2\gamma_{k+1}}\|y^{k+1}-y^{k}\|^2-\frac{\varphi(1-\ell_{y,k}\gamma_k)}{\gamma_k}\|y^k-y^{k-1}\|^2\nonumber \\
        &\quad +\frac{1}{\gamma_k}\langle F_k(x^{k-1})-F_k(x^k) , x^k-x^{k+1} \rangle +\frac{1}{\gamma_k}\langle G_{k}(y^{k-1})-G_k(y^k) , y^k-y^{k+1} \rangle\nonumber \\
        &\quad+ \langle u^{k+1} - u^\star , A(x^\star-x^{k+1})+B(y^{\star}-y^{k+1}) \rangle  +\varphi\langle A^\top u^{k} , x^{k-1}- x^{k} \rangle+\varphi\langle B^\top u^{k} , y^{k-1}- y^{k} \rangle  \rangle\nonumber\\
        &\quad -P_k +\langle u^\star -u^{k} , A(x^{k}-x^{k+1})+B(y^{k}-y^{k+1}) \rangle+\varphi f(x^{k-1}) - \varphi f(x^{k})+ \varphi g(y^{k-1}) - \varphi g(y^{k}). \nonumber 
    \end{align}  
    Add and subtract $\varphi\langle u^\star, A(x^{k-1}-x^{k})+B(y^{k-1}-y^{k}) \rangle$ and recall that 
    \[
    \varphi P_{k-1} - \varphi P_{k} = \varphi f(x^{k-1})-\varphi f(x^k)+\varphi g(y^{k-1})-\varphi g(y^k)+\varphi \langle u^\star , A(x^{k-1}-x^k)+B(y^{k-1}-y^k) \rangle.
    \]
    Then, 
           \begin{align}
        0&\le \frac{1}{2\gamma_{k+1}}\|x^{k}-x^\star\|^2 - \frac{1}{2\gamma_{k+1}}\|x^{k+1}-x^\star\|^2 - \frac{1}{2\gamma_{k+1}}\|x^{k+1}-x^{k}\|^2 -\frac{\varphi(1-\ell_{x,k}\gamma_k)}{\gamma_k}\|x^k-x^{k-1}\|^2\nonumber \\
        &\quad + \frac{1}{2\gamma_{k+1}}\|y^{k}-y^\star\|^2 - \frac{1}{2\gamma_{k+1}}\|y^{k+1}-y^\star\|^2 -  \frac{1}{2\gamma_{k+1}}\|y^{k+1}-y^{k}\|^2-\frac{\varphi(1-\ell_{y,k}\gamma_k)}{\gamma_k}\|y^k-y^{k-1}\|^2\nonumber \\
        &\quad +\frac{1}{\gamma_k}\langle F_k(x^{k-1})-F_k(x^k) , x^k-x^{k+1} \rangle +\frac{1}{\gamma_k}\langle G_{k}(y^{k-1})-G_k(y^k) , y^k-y^{k+1} \rangle\nonumber \\
        &\quad+ \underbrace{\langle u^{k+1} - u^\star , A(x^\star-x^{k+1})+B(y^{\star}-y^{k+1}) \rangle}_{\mathrm{(C)}}  +\langle u^\star -u^{k} , A(x^{k}-x^{k+1})+B(y^{k}-y^{k+1}) \rangle\nonumber\\
        &\quad +\varphi P_{k-1}-(1+\varphi)P_k +\varphi \langle u^{k}-u^\star , A(x^{k-1}- x^{k})+B(y^{k-1}-y^k) \rangle.   \label{eq:lemA3-1}
    \end{align} 
    For $\mathrm{(C)}$, we have
\begin{align*}
    \mathrm{(C)} &=  \langle u^{k+1} -u^\star , A(x^\star-x^{k+1})+B(y^{\star}-y^{k+1}) \rangle \\
    & =  -\langle u^{k+1} -u^\star , Ax^{k+1}+By^{k+1}-c \rangle\\
        & =  -\langle u^{k+1} -u^\star , \frac{1}{\sigma\gamma_{k+1}}(u^{k+1}-u^k) + A(x^{k+1}-x^k) + B(y^{k+1}-y^k) - \varphi A(x^{k}-x^{k-1}) - \varphi B(y^{k}-y^{k-1}) \rangle \\
        & = -\frac{1}{2\sigma\gamma_{k+1}}\|u^{k+1}-u^\star\|^2-\frac{1}{2\sigma\gamma_{k+1}}\|u^{k+1}-u^{k}\|^2+\frac{1}{2\sigma\gamma_{k+1}}\|u^{k}-u^\star\|^2\\
        &\quad + \langle u^{k+1}-u^\star, A(x^{k}-x^{k+1}) + B(y^{k}-y^{k+1})\rangle +  \varphi \langle u^{k+1}-u^\star, A(x^{k}-x^{k-1}) + B(y^{k}-y^{k-1})\rangle
\end{align*}
Plug in it \eqref{eq:lemA3-1} to get 
  \begin{align}
        0&\le \frac{1}{2\gamma_{k+1}}\|x^{k}-x^\star\|^2 - \frac{1}{2\gamma_{k+1}}\|x^{k+1}-x^\star\|^2 - \frac{1}{2\gamma_{k+1}}\|x^{k+1}-x^{k}\|^2 -\frac{\varphi(1-\ell_{x,k}\gamma_k)}{\gamma_k}\|x^k-x^{k-1}\|^2\nonumber \\
        &\quad + \frac{1}{2\gamma_{k+1}}\|y^{k}-y^\star\|^2 - \frac{1}{2\gamma_{k+1}}\|y^{k+1}-y^\star\|^2 -  \frac{1}{2\gamma_{k+1}}\|y^{k+1}-y^{k}\|^2-\frac{\varphi(1-\ell_{y,k}\gamma_k)}{\gamma_k}\|y^k-y^{k-1}\|^2\nonumber \\
          & \quad+\frac{1}{2\sigma\gamma_{k+1}}\|u^{k}-u^\star\|^2-\frac{1}{2\sigma\gamma_{k+1}}\|u^{k+1}-u^\star\|^2-\frac{1}{2\sigma\gamma_{k+1}}\|u^{k+1}-u^{k}\|^2 \nonumber \\
        &\quad +\frac{1}{\gamma_k}\langle F_k(x^{k-1})-F_k(x^k) , x^k-x^{k+1} \rangle +\frac{1}{\gamma_k}\langle G_{k}(y^{k-1})-G_k(y^k) , y^k-y^{k+1} \rangle\nonumber \\
        &\quad+ \langle u^{k+1} -u^{k} , A(x^{k}-x^{k+1})+B(y^{k}-y^{k+1}) \rangle\nonumber\\
        &\quad +\varphi P_{k-1}-(1+\varphi) P_k +\varphi\langle u^{k}-u^{k+1} , A(x^{k-1}- x^{k})+B(y^{k-1}-y^k) \rangle.    \nonumber
    \end{align} 
Now apply Lemma~\ref{lem:1} to get 
  \begin{align}
         0&\le \frac{1}{2\gamma_{k+1}}\|x^{k}-x^\star\|^2 - \frac{1}{2\gamma_{k+1}}\|x^{k+1}-x^\star\|^2 - \frac{1}{2\gamma_{k+1}}\|x^{k+1}-x^{k}\|^2 -\frac{\varphi(1-\ell_{x,k}\gamma_k)}{\gamma_k}\|x^k-x^{k-1}\|^2\nonumber \\
        &\quad + \frac{1}{2\gamma_{k+1}}\|y^{k}-y^\star\|^2 - \frac{1}{2\gamma_{k+1}}\|y^{k+1}-y^\star\|^2 -  \frac{1}{2\gamma_{k+1}}\|y^{k+1}-y^{k}\|^2-\frac{\varphi(1-\ell_{y,k}\gamma_k)}{\gamma_k}\|y^k-y^{k-1}\|^2\nonumber \\
          & \quad+\frac{1}{2\sigma\gamma_{k+1}}\|u^{k}-u^\star\|^2-\frac{1}{2\sigma\gamma_{k+1}}\|u^{k+1}-u^\star\|^2-\frac{1}{2\sigma\gamma_{k+1}}\|u^{k+1}-u^{k}\|^2 \nonumber \\
        &\quad + \gamma_{k+1}\left\|\frac{1}{\gamma_k}\left(F_k(x^{k-1})-F_k(x^{k})\right)+ A^\top (u^{k+1}-u^{k})\right\|^2. \nonumber \\
        &\quad+ \gamma_{k+1}\left\|\frac{1}{\gamma_k}\left( G_k(y^{k-1})-G_k(y^{k})\right)+ B^\top (u^{k+1}-u^{k})\right\|^2.\nonumber \\
        &\quad +\varphi P_{k-1}-(1+\varphi) P_k +\varphi\langle u^{k}-u^{k+1} , A(x^{k-1}- x^{k})+B(y^{k-1}-y^k) \rangle.    \nonumber
    \end{align} 
Now, instead of applying inequality $\|x+y\|^2 \le 2\|x\|^2 + 2 \|y\|^2$ in Lemma~\ref{lem:2} with $a=b=2$, we use the equality $
\|x+y\|^2=2\|x\|^2 + 2\|y\|^2
-\|x-y\|^2=\|x\|^2+\|y\|^2+2\langle x, y\rangle$,
which is simply the expansion of the squared norm:
    \begin{align}
        0&\le \frac{1}{2\gamma_{k+1}}\|x^{k}-x^\star\|^2 - \frac{1}{2\gamma_{k+1}}\|x^{k+1}-x^\star\|^2 - \frac{1}{2\gamma_{k+1}}\|x^{k+1}-x^{k}\|^2 -\frac{\varphi(1-\ell_{x,k}\gamma_k)}{\gamma_k}\|x^k-x^{k-1}\|^2\nonumber \\
        &\quad + \frac{1}{2\gamma_{k+1}}\|y^{k}-y^\star\|^2 - \frac{1}{2\gamma_{k+1}}\|y^{k+1}-y^\star\|^2 -  \frac{1}{2\gamma_{k+1}}\|y^{k+1}-y^{k}\|^2-\frac{\varphi(1-\ell_{y,k}\gamma_k)}{\gamma_k}\|y^k-y^{k-1}\|^2\nonumber \\
          & \quad+\frac{1}{2\sigma\gamma_{k+1}}\|u^{k}-u^\star\|^2-\frac{1}{2\sigma\gamma_{k+1}}\|u^{k+1}-u^\star\|^2-\frac{1}{2\sigma\gamma_{k+1}}\|u^{k+1}-u^{k}\|^2 \nonumber \\
        &\quad +\frac{\rho_{k+1}}{\gamma_k}\|F_k(x^{k-1})-F_k(x^{k})\|^2+ \gamma_{k+1}a_{k+1}^2\|u^{k+1}-u^k\|^2  + 2\rho_{k+1} \langle  F_k(x^{k-1})-F_k(x^{k}) , A^\top (u^{k+1}-u^k) \rangle \nonumber \\
  &\quad +\frac{\rho_{k+1}}{\gamma_k}\|G_k(y^{k-1})-G_k(y^{k})\|^2+ \gamma_{k+1}b_{k+1}^2\|u^{k+1}-u^k\|^2 + 2\rho_{k+1} \langle  G_k(y^{k-1})-G_k(y^{k}) , B^\top (u^{k+1}-u^k) \rangle\nonumber \\
        &\quad +\varphi P_{k-1}-(1+\varphi) P_k +\varphi \langle u^{k}-u^{k+1} , A(x^{k-1}- x^{k})+B(y^{k-1}-y^k) \rangle.    \nonumber
    \end{align} 
    Apply Lemma~\ref{lem:3} to get
        \begin{align}
        0&\le \frac{1}{2\gamma_{k+1}}\|x^{k}-x^\star\|^2 - \frac{1}{2\gamma_{k+1}}\|x^{k+1}-x^\star\|^2 - \frac{1}{2\gamma_{k+1}}\|x^{k+1}-x^{k}\|^2 -\frac{\varphi(1-\ell_{x,k}\gamma_k)}{\gamma_k}\|x^k-x^{k-1}\|^2\nonumber \\
        &\quad + \frac{1}{2\gamma_{k+1}}\|y^{k}-y^\star\|^2 - \frac{1}{2\gamma_{k+1}}\|y^{k+1}-y^\star\|^2 -  \frac{1}{2\gamma_{k+1}}\|y^{k+1}-y^{k}\|^2-\frac{\varphi(1-\ell_{y,k}\gamma_k)}{\gamma_k}\|y^k-y^{k-1}\|^2\nonumber \\
          & \quad+\frac{1}{2\sigma\gamma_{k+1}}\|u^{k}-u^\star\|^2-\frac{1}{2\sigma\gamma_{k+1}}\|u^{k+1}-u^\star\|^2-\frac{1}{2\sigma\gamma_{k+1}}\|u^{k+1}-u^{k}\|^2 \nonumber \\
        &\quad +\frac{\rho_{k+1}(\gamma_{k}^2L_{x,k}^2 -2\gamma_{k}\ell_{x,k}+1)}{\gamma_k}\|x^{k-1}-x^{k}\|^2+ \gamma_{k+1}a_{k+1}^2\|u^{k+1}-u^k\|^2 \nonumber \\
  &\quad +\frac{\rho_{k+1}(\gamma_{k}^2L_{y,k}^2 -2\gamma_{k}\ell_{y,k}+1)}{\gamma_k}\|y^{k-1}-y^{k}\|^2+ \gamma_{k+1}b_{k+1}^2\|u^{k+1}-u^k\|^2 \nonumber \\
  & \quad +2\rho_{k+1} \langle  F_k(x^{k-1})-F_k(x^{k}) , A^\top (u^{k+1}-u^k) \rangle + 2\rho_{k+1} \langle  G_k(y^{k-1})-G_k(y^{k}) , B^\top (u^{k+1}-u^k) \rangle  \nonumber \\
        &\quad +\varphi P_{k-1}-(1+\varphi)P_k +\varphi\langle u^{k}-u^{k+1} , A(x^{k-1}- x^{k})+B(y^{k-1}-y^k) \rangle.    \nonumber
    \end{align} 
    Apply Lemma~\ref{lem:4-3} to get
       \begin{align}
        0&\le \frac{1}{2\gamma_{k+1}}\|x^{k}-x^\star\|^2 - \frac{1}{2\gamma_{k+1}}\|x^{k+1}-x^\star\|^2 - \frac{1}{2\gamma_{k+1}}\|x^{k+1}-x^{k}\|^2 -\frac{\varphi(1-\ell_{x,k}\gamma_k)}{\gamma_k}\|x^k-x^{k-1}\|^2\nonumber \\
        &\quad + \frac{1}{2\gamma_{k+1}}\|y^{k}-y^\star\|^2 - \frac{1}{2\gamma_{k+1}}\|y^{k+1}-y^\star\|^2 -  \frac{1}{2\gamma_{k+1}}\|y^{k+1}-y^{k}\|^2-\frac{\varphi(1-\ell_{y,k}\gamma_k)}{\gamma_k}\|y^k-y^{k-1}\|^2\nonumber \\
          & \quad+\frac{1}{2\sigma\gamma_{k+1}}\|u^{k}-u^\star\|^2-\frac{1}{2\sigma\gamma_{k+1}}\|u^{k+1}-u^\star\|^2-\frac{1}{2\sigma\gamma_{k+1}}\|u^{k+1}-u^{k}\|^2 \nonumber \\
        &\quad +\frac{\rho_{k+1}(\gamma_{k}^2L_{x,k}^2 -2\gamma_{k}\ell_{x,k}+1)}{\gamma_k}\|x^{k-1}-x^{k}\|^2+ \gamma_{k+1}a_{k+1}^2\|u^{k+1}-u^k\|^2 \nonumber \\
  &\quad +\frac{\rho_{k+1}(\gamma_{k}^2L_{y,k}^2 -2\gamma_{k}\ell_{y,k}+1)}{\gamma_k}\|y^{k-1}-y^{k}\|^2+ \gamma_{k+1}b_{k+1}^2\|u^{k+1}-u^k\|^2 \nonumber \\
        &\quad +2\rho_{k+1} \langle  F_k(x^{k-1})-F_k(x^{k}) , A^\top (u^{k+1}-u^k) \rangle + 2\rho_{k+1} \langle  G_k(y^{k-1})-G_k(y^{k}) , B^\top (u^{k+1}-u^k) \rangle   \nonumber \\
           &\quad + \frac{\lambda_{k+1}^{A}}{8\sigma\gamma_{k+1}}\|u^{k+1}-u^k\|^2  +2\varphi^2 \sigma a_{k+1}^2\gamma_{k+1}\lambda^{A}_{k+1}\|x^{k}-x^{k-1}\|^2 \nonumber \\
        &\quad + \frac{\lambda_{k+1}^{B}}{8\sigma\gamma_{k+1}}\|u^{k+1}-u^k\|^2  +2\varphi^2 \sigma b_{k+1}^2\gamma_{k+1}\lambda^{B}_{k+1}\|y^{k}-y^{k-1}\|^2 +\varphi P_{k-1}-(1+\varphi)P_k.  \nonumber
    \end{align} 
    Now by the stepsize update rule $\rho_{k+1}\le \varphi $, 
    \[
    \frac{\rho_{k+1}(\gamma_{k}^2L_{x,k}^2 -2\gamma_{k}\ell_{x,k}+1)}{\gamma_k} - \frac{\varphi}{\gamma_k} \le  \frac{\rho_{k+1}(\gamma_{k}^2L_{x,k}^2 -2\gamma_{k}\ell_{x,k})}{\gamma_k} = \frac{\rho_{k+1}\delta_{x,k}}{\gamma_k}.
    \]
    and
    \[
    \frac{\rho_{k+1}(\gamma_{k}^2L_{y,k}^2 -2\gamma_{k}\ell_{y,k}+1)}{\gamma_k} - \frac{\varphi}{\gamma_k} \le  \frac{\rho_{k+1}(\gamma_{k}^2L_{y,k}^2 -2\gamma_{k}\ell_{y,k})}{\gamma_k} = \frac{\rho_{k+1}\delta_{y,k}}{\gamma_k}.
    \]
    Using this result, we get the following 
     \begin{align}
        0&\le \frac{1}{2\gamma_{k+1}}\|x^{k}-x^\star\|^2 - \frac{1}{2\gamma_{k+1}}\|x^{k+1}-x^\star\|^2 + \frac{1}{2\gamma_{k+1}}\|y^{k}-y^\star\|^2 - \frac{1}{2\gamma_{k+1}}\|y^{k+1}-y^\star\|^2 \nonumber \\
          & \quad+\frac{1}{2\sigma\gamma_{k+1}}\|u^{k}-u^\star\|^2-\frac{1}{2\sigma\gamma_{k+1}}\|u^{k+1}-u^\star\|^2 \nonumber \\
        &\quad -\frac{1}{2\gamma_{k+1}}\|x^{k+1}-x^k\|^2+\left(\frac{\rho_{k+1}\delta_{x,k}}{\gamma_k}+\varphi\ell_{x,k}+2\varphi^2 \sigma a_{k+1}^2\gamma_{k+1}\lambda^{A}_{k+1}\right)\|x^{k-1}-x^{k}\|^2 \nonumber \\
  &\quad  -\frac{1}{2\gamma_{k+1}}\|y^{k+1}-y^k\|^2+\left(\frac{\rho_{k+1}\delta_{y,k}}{\gamma_k}+\varphi \ell_{y,k}+2\varphi^2 \sigma b_{k+1}^2\gamma_{k+1}\lambda^{B}_{k+1}\right)\|y^{k-1}-y^{k}\|^2 \nonumber \\
   &\quad +2\rho_{k+1} \langle  F_k(x^{k-1})-F_k(x^{k}) , A^\top (u^{k+1}-u^k) \rangle + 2\rho_{k+1} \langle  G_k(y^{k-1})-G_k(y^{k}) , B^\top (u^{k+1}-u^k) \rangle   \nonumber \\
  & \quad -\left( \frac{1}{2\sigma \gamma_{k+1}} - \frac{\lambda_{k+1}^{A}}{8\sigma\gamma_{k+1}}-\frac{\lambda_{k+1}^{B}}{8\sigma\gamma_{k+1}}-\gamma_{k+1}a_{k+1}^2-\gamma_{k+1}b_{k+1}^2\right) \|u^{k+1}-u^k\|^2+\varphi P_{k-1}-(1+\varphi) P_k . \nonumber
    \end{align} 
     Apply Lemma~\ref{lem:4-4} to get
     \begin{align}
        0&\le \frac{1}{2\gamma_{k+1}}\|x^{k}-x^\star\|^2 - \frac{1}{2\gamma_{k+1}}\|x^{k+1}-x^\star\|^2 + \frac{1}{2\gamma_{k+1}}\|y^{k}-y^\star\|^2 - \frac{1}{2\gamma_{k+1}}\|y^{k+1}-y^\star\|^2 \nonumber \\
          & \quad+\frac{1}{2\sigma\gamma_{k+1}}\|u^{k}-u^\star\|^2-\frac{1}{2\sigma\gamma_{k+1}}\|u^{k+1}-u^\star\|^2 \nonumber \\
        &\quad -\frac{1}{2\gamma_{k+1}}\|x^{k+1}-x^k\|^2+\left(\frac{\rho_{k+1}\delta_{x,k}}{\gamma_k}+\varphi\ell_{x,k}+2\varphi^2 \sigma a_{k+1}^2\gamma_{k+1}\lambda^{A}_{k+1}\right)\|x^{k-1}-x^{k}\|^2 \nonumber \\
  &\quad  -\frac{1}{2\gamma_{k+1}}\|y^{k+1}-y^k\|^2+\left(\frac{\rho_{k+1}\delta_{y,k}}{\gamma_k}+\varphi \ell_{y,k}+2\varphi^2 \sigma b_{k+1}^2\gamma_{k+1}\lambda^{B}_{k+1}\right)\|y^{k-1}-y^{k}\|^2 \nonumber \\
   & \quad + \frac{\mu_{k+1}^{A}}{\sigma}\|u^{k+1}-u^k\|^2  +\sigma a_{k+1}^2\rho_{k+1}^2\mu^{A}_{k+1}\|F_k(x^{k-1}) - F_k(x^{k})\|^2 \label{eq: lemA3-2} \\
 & \quad+ \frac{\mu_{k+1}^{B}}{\sigma}\|u^{k+1}-u^k\|^2  +\sigma b_{k+1}^2\rho_{k+1}^2\mu^{B}_{k+1}\|G_k(y^{k-1})- G_k(y^{k})\|^2  \label{eq: lemA3-3} \\
  & \quad -\left( \frac{1}{2\sigma\gamma_{k+1}} - \frac{\lambda_{k+1}^{A}}{8\sigma\gamma_{k+1}}-\frac{\lambda_{k+1}^{B}}{8\sigma\gamma_{k+1}}-\gamma_{k+1}a_{k+1}^2-\gamma_{k+1}b_{k+1}^2\right) \|u^{k+1}-u^k\|^2+\varphi P_{k-1}-(1+\varphi) P_k . \nonumber
    \end{align} 
    Multiply by $\gamma_{k+1}$ and use Lemma~\ref{lem:3} on \eqref{eq: lemA3-2} and \eqref{eq: lemA3-3} to get 
         \begin{align}
        0&\le \frac{1}{2}\|x^{k}-x^\star\|^2 - \frac{1}{2}\|x^{k+1}-x^\star\|^2+ \frac{1}{2}\|y^{k}-y^\star\|^2 - \frac{1}{2}\|y^{k+1}-y^\star\|^2 +\frac{1}{2\sigma}\|u^{k}-u^\star\|^2-\frac{1}{2\sigma}\|u^{k+1}-u^\star\|^2 \nonumber \\
        &\quad - \frac{1}{2}\|x^{k+1}-x^k\|^2- \frac{1}{2}\|y^{k+1}-y^k\|^2 \nonumber \\
        &\quad +\left(\rho_{k+1}^2\delta_{x,k}+\varphi\gamma_k\ell_{x,k}\rho_{k+1}+ 2 \varphi^2 \sigma a_{k+1}^2\gamma_{k+1}^2\lambda^{A}_{k+1} +\sigma a_{k+1}^2\gamma_k\rho_{k+1}^3\mu_{k+1}^{A}\left(\delta_{x,k}+1\right)\right)\|x^{k-1}-x^{k}\|^2\nonumber \\
  &\quad +\left(\rho_{k+1}^2\delta_{y,k}+\varphi\gamma_k\ell_{y,k}\rho_{k+1} +2\varphi^2 \sigma b_{k+1}^2\gamma_{k+1}^2\lambda^{B}_{k+1}+\sigma b_{k+1}^2\gamma_k\rho_{k+1}^3\mu_{k+1}^{B}\left(\delta_{y,k}+1\right)\right)\|y^{k-1}-y^{k}\|^2 \nonumber \\
  & \quad -\left(\frac{1}{2\sigma} - \frac{\lambda_{k+1}^{A}+\lambda_{k+1}^{B}}{8\sigma}-\gamma_{k+1}\frac{\mu_{k+1}^{A}+\mu_{k+1}^{B}}{\sigma}-\gamma_{k+1}^2(a_{k+1}^2+b_{k+1}^2)\right) \|u^{k+1}-u^k\|^2 \nonumber \\
  &\quad +\varphi\gamma_{k+1} P_{k-1}-(1+\varphi)\gamma_{k+1}P_k . \label{eq:A3-4}
    \end{align} 
    
Just as we did in the proof of Lemma~\ref{lem:A1}, \eqref{eq:A3-4} provides a one-step descent inequality. We now derive the monotonicity of the Lyapunov sequence $\{\mathcal{V}_k\}_{k\ge1}$. Recall the Lyapunov sequence: 
\begin{align*}
\mathcal{V}_k &= \frac{1}{2}\|x^{k}-x^\star\|^2 + \frac{1}{2}\|x^{k}-x^{k-1}\|^2 + \frac{1}{2}\|y^{k}-y^\star\|^2 +  \frac{1}{2}\|y^{k}-y^{k-1}\|^2 \nonumber \\
&+\frac{1}{2\sigma}\|u^{k}-u^\star\|^2 + (1+\varphi)\gamma_{k}P_{k-1}, 
\end{align*}
Then expressing \eqref{eq:A3-4} in terms of $\mathcal{V}_k$ and  $\mathcal{V}_{k+1}$,
\begin{align*}
    \mathcal{V}_{k+1} &\le \mathcal{V}_k-\left(\frac{1}{2}-\rho_{k+1}^2\delta_{x,k}-\varphi\gamma_k\ell_{x,k}\rho_{k+1} -2\varphi^2 \sigma a_{k+1}^2\gamma_{k+1}^2\lambda^{A}_{k+1}-\sigma a_{k+1}^2\gamma_k\rho_{k+1}^3\mu_{k+1}^{A}\left(\delta_{x,k}+1\right)\right)\|x^{k}-x^{k-1}\|^2\\
    &\quad -\left(\frac{1}{2}-\rho_{k+1}^2\delta_{y,k}-\varphi\gamma_k\ell_{y,k}\rho_{k+1} -2\varphi^2 \sigma b_{k+1}^2\gamma_{k+1}^2\lambda^{B}_{k+1}-\sigma b_{k+1}^2\gamma_k\rho_{k+1}^3\mu_{k+1}^{B}\left(\delta_{y,k}+1\right)\right)\|y^{k}-y^{k-1}\|^2\\
    &\quad -\left(\frac{1}{2\sigma} - \frac{\lambda_{k+1}^{A}+\lambda_{k+1}^{B}}{8\sigma}-\gamma_{k+1}\frac{\mu_{k+1}^{A}+\mu_{k+1}^{B}}{\sigma}-\gamma_{k+1}^2(a_{k+1}^2+b_{k+1}^2)\right) \|u^{k+1}-u^k\|^2\\
    &\quad -\gamma_k(1+\varphi-\varphi \rho_{k+1})P_{k-1} .
\end{align*}
Hence, it is enough to show
\[
\frac{1}{2}-\rho_{k+1}^2\delta_{x,k}-\varphi\gamma_k\ell_{x,k}\rho_{k+1} -2\varphi^2 \sigma a_{k+1}^2\gamma_{k+1}^2\lambda^{A}_{k+1}-\sigma a_{k+1}^2\gamma_k\rho_{k+1}^3\mu_{k+1}^{A}\left(\delta_{x,k}+1\right) \ge \varepsilon ,
\]
\[
\frac{1}{2}-\rho_{k+1}^2\delta_{y,k}-\varphi\gamma_k\ell_{y,k}\rho_{k+1} -2\varphi^2 \sigma b_{k+1}^2\gamma_{k+1}^2\lambda^{B}_{k+1}-\sigma b_{k+1}^2\gamma_k\rho_{k+1}^3\mu_{k+1}^{B}\left(\delta_{y,k}+1\right) \ge \varepsilon,
\]
\[
\frac{1}{2\sigma} - \frac{\lambda_{k+1}^{A}+\lambda_{k+1}^{B}}{8\sigma}-\gamma_{k+1}\frac{\mu_{k+1}^{A}+\mu_{k+1}^{B}}{\sigma}-\gamma_{k+1}^2(a_{k+1}^2+b_{k+1}^2) \ge \varepsilon .
\]
Rearrange the first two inequalities as
\[
\frac{\sigma a_{k+1}^2\mu_{k+1}^{A}\left(\delta_{x,k}+1\right)}{\gamma_k^2} \gamma_{k+1}^3 + \left(2\varphi^2 \sigma a_{k+1}^2\lambda^{A}_{k+1}+\frac{\delta_{x,k}}{\gamma_k^2}\right)\gamma_{k+1}^2+\varphi\ell_{x,k}\gamma_{k+1}-\frac{1-2\varepsilon}{2} \le 0
\]
and
\[
\frac{\sigma b_{k+1}^2\mu_{k+1}^{B}\left(\delta_{y,k}+1\right)}{\gamma_k^2} \gamma_{k+1}^3 + \left(2\varphi^2 \sigma b_{k+1}^2\lambda^{B}_{k+1}+\frac{\delta_{y,k}}{\gamma_k^2}\right)\gamma_{k+1}^2+\varphi\ell_{y,k}\gamma_{k+1}-\frac{1-2\varepsilon}{2} \le 0.
\]
 For the last inequality, solve the quadratic to get 
\[
\gamma_{k+1} \le \frac{4-\lambda_{k+1}^{A}-\lambda_{k+1}^{B}-8\sigma\varepsilon}{4\sigma} \cdot \frac{1}{\frac{\mu_{k+1}^{A}+\mu_{k+1}^{B}}{\sigma}+\sqrt{\frac{(\mu_{k+1}^{A}+\mu_{k+1}^{B})^2}{\sigma^2}+(a_{k+1}^2+b_{k+1}^2)\frac{4-\lambda_{k+1}^{A}-\lambda_{k+1}^{B}-8\sigma \varepsilon}{2\sigma}}}.
\]
The nonincreasing property holds provided that $\gamma_{k+1}$ satisfies the above three inequalities, which is precisely enforced by the stepsize selection rule in Subroutine~\ref{alg:S2}. If either of the cubic inequalities admits no positive real root, then the corresponding constraint on $\gamma_{k+1}$ is vacuous. In this case, we may equivalently set the associated bound to $+\infty$, without affecting the nonincreasing property. \qed

The following result is an analogue of Lemma~\ref{lem:gamma1} for Subroutine~\ref{alg:S2}:
\begin{lemma}\label{lem:gamma2}
        The sequence $\{\gamma_k\}_{k \ge 0}$ generated by Subroutine~\ref{alg:S2} is bounded away from zero. In other words, there exists $\gamma>0$ such that $ \gamma_k \ge \gamma >0$ for every $k\ge0$.
\end{lemma}
\begin{proof}
    Starting from the nonincreasing property \eqref{eq:lya2}, we get the boundedness of $\{x_k\}_{k\ge0}$, $\{y_k\}_{k\ge0}$, and $\{u_k\}_{k\ge0}$:
    \[
\frac{1}{2}\|x^{k}-x^\star\|^2 + \frac{1}{2}\|y^{k}-y^\star\|^2 +\frac{1}{2\sigma}\|u^{k}-u^\star\|^2  \le \mathcal{V}_k \le \cdots \le \mathcal{V}_1.
\]
 For sufficiently large convex compact set $U\subseteq\mathbb{R}^p$ and $V\subseteq\mathbb{R}^q$ that contain $(x^k,y^k)$, denote $L_U>0$ and $L_V>0$ be the corresponding smoothness constant for $f_2$ and $g_2$ respectively. Now consider when 
\[
\gamma_{k+1} = \frac{4-\lambda_{k+1}^{A}-\lambda_{k+1}^{B}-8\sigma\varepsilon}{4\sigma} \cdot \frac{1}{\frac{\mu_{k+1}^{A}+\mu_{k+1}^{B}}{\sigma}+\sqrt{\frac{(\mu_{k+1}^{A}+\mu_{k+1}^{B})^2}{\sigma^2}+(a_{k+1}^2+b_{k+1}^2)\frac{4-\lambda_{k+1}^{A}-\lambda_{k+1}^{B}-8\sigma\varepsilon}{2\sigma}}}.
\]
Then,
\begin{align*}
 \gamma_{k+1} &= \frac{4-\lambda_{k+1}^{A}-\lambda_{k+1}^{B}-8\sigma\varepsilon}{4\sigma} \cdot \frac{1}{\frac{\mu_{k+1}^{A}+\mu_{k+1}^{B}}{\sigma}+\sqrt{\frac{(\mu_{k+1}^{A}+\mu_{k+1}^{B})^2}{\sigma^2}+(a_{k+1}^2+b_{k+1}^2)\frac{4-\lambda_{k+1}^{A}-\lambda_{k+1}^{B}-8\sigma\varepsilon}{2\sigma}}}\\
 & \ge \frac{2-8\sigma\varepsilon}{4\sigma} \cdot \frac{1}{\frac{2}{\sigma}+\sqrt{\frac{4}{\sigma^2}+(\|A\|^2+\|B\|^2)\frac{6-8\sigma\varepsilon}{2\sigma}}}.
\end{align*}
Now consider the case when $\gamma_{k+1}$ is smallest positive root of 
\[
\frac{\sigma a_{k+1}^2\mu_{k+1}^{A}\left(\delta_{x,k}+1\right)}{\gamma_k^2} x^3 + \left(2\varphi^2 \sigma a_{k+1}^2\lambda^{A}_{k+1}+\frac{\delta_{x,k}}{\gamma_k^2}\right)x^2+\varphi\ell_{x,k}x-\frac{1-2\varepsilon}{2}=0.
\]
It suffices to show that $\gamma_{k+1}$ must be bounded away from zero whenever $\gamma_{k+1}\le \gamma_k$. That is, when $0<\rho_{k+1}\le 1$. Denote $p(x)$ to be the polynomial on the left-hand side. We will use the fact that
\[
\delta_{x,k}+1= \gamma_k^2 L_{x,k}^2 - 2\gamma_k \ell_{x,k}+1\ge  \gamma_k^2 L_{x,k}^2 - 2\gamma_k  L_{x,k}+1=(\gamma_k L_{x,k}-1)^2\ge0\]
by Lemma~\ref{lem:7}. Then,
\begin{align*}
\!\!\!\!\!\!\!\!\!\!\!\!\!\! 0& = \sigma a_{k+1}^2\mu_{k+1}^{A}\gamma_k \left(\delta_{x,k}+1\right)\rho_{k+1}^3  + \left(2\varphi^2 \sigma a_{k+1}^2\gamma_k^2\lambda^{A}_{k+1}+\delta_{x,k}\right)\rho_{k+1}^2+\varphi\gamma_k\ell_{x,k}\rho_{k+1}-\frac{1-2\varepsilon}{2} \\
          &\quad \le \sigma a_{k+1}^2|\mu_{k+1}^{A}|\gamma_k \left(\gamma_k^2L_{x,k}^2+1\right)\rho_{k+1}^3  + \left(2\varphi^2 \sigma a_{k+1}^2\gamma_k^2\lambda^{A}_{k+1}+\gamma_k^2L_{x,k}^2\right)\rho_{k+1}^2+\varphi\gamma_k\ell_{x,k}\rho_{k+1}-\frac{1-2\varepsilon}{2}\\
           &\quad \le \sigma a_{k+1}^2\gamma_k^3L_{x,k}^2\rho_{k+1}^3  + \sigma a_{k+1}^2\gamma_k\rho_{k+1}^3+\left(2\varphi^2 \sigma a_{k+1}^2\gamma_k^2\lambda^{A}_{k+1}+\gamma_k^2L_{x,k}^2\right)\rho_{k+1}^2+\varphi\gamma_k\ell_{x,k}\rho_{k+1}-\frac{1-2\varepsilon}{2}\\
&\quad \le \sigma a_{k+1}^2\gamma_k^3L_{x,k}^2\rho_{k+1}^3  + \sigma a_{k+1}^2\gamma_k\rho_{k+1}+\left(2\varphi^2 \sigma a_{k+1}^2\gamma_k^2\lambda^{A}_{k+1}+\gamma_k^2L_{x,k}^2\right)\rho_{k+1}^2+\varphi\gamma_k\ell_{x,k}\rho_{k+1}-\frac{1-2\varepsilon}{2}\\
&\quad= \sigma a_{k+1}^2L_{x,k}^2\gamma_{k+1}^3  + \sigma a_{k+1}^2\gamma_{k+1}+\left(2\varphi^2 \sigma a_{k+1}^2\lambda^{A}_{k+1}+L_{x,k}^2\right)\gamma_{k+1}^2+\varphi\ell_{x,k}\gamma_{k+1}-\frac{1-2\varepsilon}{2}\\
&\quad \le \sigma \|A\|^2L_U^2\gamma_{k+1}^3 +\left(2\varphi^2 \sigma\|A\|^2+L_U^2\right)\gamma_{k+1}^2+\left(\varphi L_U+\sigma\|A\|^2\right)\gamma_{k+1}-\frac{1-2\varepsilon}{2}.
\end{align*}
The first inequality follows from $0\le \delta_{x,k}+1\le\gamma_k^2 L_{x,k}^2+1 $, the second follows from $|\mu_{k+1}^{A}|\le 1$, and for the third inequality, we used $\rho^3_{k+1}\le \rho_{k+1}$ for $0<\rho_{k+1}\le 1$. Let 
\[
p_1(x)= \sigma\|A\|^2L_U^2x^3 +\left(2\varphi^2 \sigma\|A\|^2+L_U^2\right)x^2+\left(\varphi L_U+\sigma\|A\|^2\right)x-\frac{1-2\varepsilon}{2}
\]
and observe that $p_1$ is a strictly increasing cubic for $x\ge0$ with $p_1(0)<0$ that is independent of iterate count $k$. Denote $\gamma_x>0$ to be the value that crosses the $x$-axis; i.e., $p_1(\gamma_x)=0$. Then since $0=p_1(\gamma_x)\le p_1(\gamma_{k+1})$, it must be that $\gamma_x \le \gamma_{k+1}$. For $y$-iterate, apply the same argument with corresponding cubic root $\gamma_y>0$.  Then we have
\[
\gamma_k \ge  \min \left\{ \gamma_0 , \quad \frac{2-8\sigma\varepsilon}{4\sigma} \cdot \frac{1}{\frac{2}{\sigma}+\sqrt{\frac{4}{\sigma^2}+(\|A\|^2+\|B\|^2)\frac{6-8\sigma\varepsilon}{2\sigma}}} , \quad  \gamma_x , \quad \gamma_y\right\}
\]
for any $k\ge 0$. \qed
\end{proof}

\paragraph{Convergence and the proof of Theorem~\ref{thm:2}.} 
The remainder of the convergence proof is identical to that in Section~\ref{sec:3-2} for Subroutine~\ref{alg:S1} and is therefore omitted. The $O(1/k)$ convergence rate analysis follows the same arguments as in Section~\ref{sec:3-2} as well, with the only difference being modified numerical constants, and is omitted for brevity.

\section{Experiments}\label{sec:4}
In this section, we empirically evaluate \ALiA\  on a collection of real-world problems: Depth map estimation in Section~\ref{sec: depth map estimation} and hyperspectral image unmixing in Section~\ref{sec: hyperspectral image unmixing}. We also consider some simpler tasks in Section~\ref{sec: compare with PDM}. We compare \ALiA\ against the (non-adaptive) FLiP-ADMM method, and, when the problem structure permits, the adaptive methods of Malitsky--Pock \cite[Algorithm~4]{malitsky2018first} and adaPDM/adaPDM+ \cite{latafat2024adaptive}, as well as the classical (non-adaptive) Condat--V\~u method \cite{condat2013primal}. The results demonstrate the strong empirical performance of \ALiA.

To clarify, adaPDM and adaPDM+  differ in their adaptivity with respect to the linear operator in the coupling constraint. Specifically, adaPDM+ employs a line search to achieve adaptivity with respect to the operator norm, whereas adaPDM directly calculates the operator norm and therefore does not require a line search. In the special case where the mapping is an identity,  computing the norm incurs no additional cost, making adaPDM computationally preferable. We therefore use adaPDM in the identity-mapping case and adaPDM+ otherwise.

All algorithms are implemented in the open-source Julia package \texttt{PDMO.jl} \cite{sun2026automating}, and the code used to generate results in this paper is available at branch \texttt{test/alia}:
\begin{center}
\url{github.com/alibaba-damo-academy/PDMO.jl/tree/test/alia}.
\end{center}
\medskip

\paragraph{Stopping criterion.}
Note that $(x^\star,y^\star,u^\star)$ is a saddle point of $\mathbf{L}(x,y,u)$ if  
\[
(0,0,0)\in 
(\partial f_1(x^\star)+\nabla f_2(x^\star)+ A^\top u^\star , \partial g_1(y^\star)+\nabla g_2(y^\star) + B^\top u^\star, Ax^\star+By^\star -c ). 
\]
By definition of $x^{k+1}$ and $y^{k+1}$,
\begin{align*}
0 &\in \partial f_1(x^{k+1}) + \nabla f_2 (x^{k}) + A^\top u^{k+1} + \frac{1}{\gamma_{k+1}}\left( x^{k+1}-x^{k}\right)
\\
0 &\in \partial g_1(y^{k+1}) + \nabla g_2 (y^{k}) + B^\top u^{k+1} + \frac{1}{\gamma_{k+1}}\left( y^{k+1}-y^{k}\right). 
\end{align*}
Therefore, the magnitude of the following dual and primal residuals can be used as a stopping criterion:
\begin{align*}
w_1^{k+1}:=& \frac{1}{\gamma_{k+1}}\left( x^{k}-x^{k+1}\right) -\nabla f_2 (x^{k}) + \nabla f_2 (x^{k+1})\in \partial f_1 (x^{k+1}) + \nabla f_2 (x^{k+1}) + A^\top u^{k+1} \\
w_2^{k+1}:= & \frac{1}{\gamma_{k+1}}\left( y^{k}-y^{k+1}\right) -\nabla g_2 (y^{k}) + \nabla g_2 (y^{k+1}) \in \partial g_1 (y^{k+1}) + \nabla g_2 (y^{k+1}) + B^\top u^{k+1}\\
w_3^{k+1}:= & Ax^{k+1}+By^{k+1}-c.
\end{align*}
In our experiments, we used the stopping criterion 
\[
\max\big\{\|(w_1^{k+1}, w_2^{k+1})\|_2,\|w_3^{k+1}\|_2\big\}\le 10^{-4} \quad \text{and} \quad \max\big\{\|(w_1^{k+1}, w_2^{k+1})\|_\infty,\|w_3^{k+1}\|_\infty\big\}\le 10^{-6}.
\]

\subsection{Depth map estimation}\label{sec: depth map estimation}
We consider the depth map estimation problem \cite{pustelnik2017proximity}, whose goal is to estimate a depth map from a pair of images depicting the same scene from slightly different viewpoints, such as those captured by the left and right eyes. Let the pixel domain be $\Omega = \{1,\ldots, N_1\} \times \{1,\ldots, N_2\}$.
A variational formulation seeks to estimate, for each pixel, one of a finite set of possible disparity levels \(q = 1,\ldots, Q\). For each pixel \(n \in \Omega\) and each disparity level \(1 \leq q \leq Q\), a data-fidelity cost \(\eta^{(q)}_n\) is defined, quantifying how well the right image, shifted according to disparity level \(q\), aligns with the left image at pixel \(n\). The objective is to assign each pixel to one of the \(Q\) disparity levels in a way that balances data fidelity with spatial coherence.

A convex relaxation of the discrete assignment problem can be formulated as \cite{chambolle2012convex}:
\begin{align}\label{eq: depth_map_estimation_original}
\begin{array}{ll}
    \underset{\bm{\Theta} = (\bm{\theta}^{(1)},\cdots, \bm{\theta}^{(Q-1)})} 
    {\mbox{minimize}}
    &\displaystyle{\sum_{q=1}^{Q-1}\langle \alpha^{(q)}, \bm{\theta}^{(q)}\rangle + \delta_E(\bm{\Theta}) + \lambda \sum_{q=1}^Q \|(DH_q)(\bm{\Theta})\|_{2,1}}. 
\end{array}
\end{align}
The decision variable $\bm{\Theta}$ consists of matrices $\bm{\theta}^{(q)}\in \R^\Omega$ for $q = 1, \dots, Q-1$. The first term in \eqref{eq: depth_map_estimation_original} is linear with coefficient $\alpha^{(q)} := \eta^{(q+1)} - \eta^{(q)}$ (recall that $\eta^{(q)}$ is a matrix in $\R^{\Omega}$ whose entries are $\eta^{(q)}_n$ for $n\in \Omega$).  For notational convenience, we let $\bm{\theta}^{(0)}$ be the matrix of all ones and $\bm{\theta}^{(Q)}$ be the matrix of all zeros. The second objective component is an indicator function of the set 
\begin{align}\label{eq: monotone_matrices}
    E := \{\bm{\Theta} = (\bm{\theta}^{(1)}, \dots, \bm{\theta}^{(Q-1)}):~\bm{1} \equiv \bm{\theta}^{(0)} \geq \bm{\theta}^{(1)}\geq \dots\geq \bm{\theta}^{(Q-1)} \geq \bm{\theta}^{(Q)}\equiv \bm{0}\}. 
\end{align}
We note that projection onto $E$ can be computed exactly and efficiently by the PAVA algorithm \cite{ayer1955empirical} applied pixel-wise. 
The last term in the objective represents a 2-dimensional discrete total variation, where $D$ is the linear operator taking horizontal and vertical finite differences, and $H_q$ takes the difference of two adjacent indexed matrices, i.e.,  $H_q(\bm{\Theta}) = \bm{\theta}^{(q-1)} - \bm{\theta}^{(q)}$. We note that the composition of $H_q$ with $D$ remains linear. Introducing auxiliary variables, we rewrite problem \eqref{eq: depth_map_estimation_original} as 
\begin{align*}
\begin{array}{ll}
    \underset{\stackrel{\bm{\theta} = (\bm{\theta}^{(1)},\dots, \bm{\theta}^{(Q-1)})}{\bm{\Phi} = (\bm\phi^{(1)}, \dots, \bm\phi^{(Q)})}}  
    {\mbox{minimize}}
 & \displaystyle{\sum_{q=1}^{Q-1}  \langle \alpha^{(q)}, \bm{\theta}^{(q)}\rangle + \delta_E(\bm{\Theta}) + \lambda \sum_{q=1}^Q \|\bm\phi^{(q)}\|_{2,1}}\\
\mbox{subject to}& (DH_q)(\bm{\Theta}) = \bm\phi^{(q)}, \quad\text{for } q = 1, \dots, Q.
\end{array}
\end{align*}

Figure~\ref{fig:depthmap} compares \ALiA\ against other baselines applied to this problem with three different NYU Depth V2 images \cite{Silberman:ECCV12}. We see that \ALiA\ consistently outperforms the baselines, and \ALiA\ with Subroutine~\ref{alg:S2} is consistently better or not worse than \ALiA\ with Subroutine~\ref{alg:S1}.

\begin{figure}[ht]
  \centering
  \begin{subfigure}{0.32\textwidth}
    \centering
    \includegraphics[width=\linewidth]{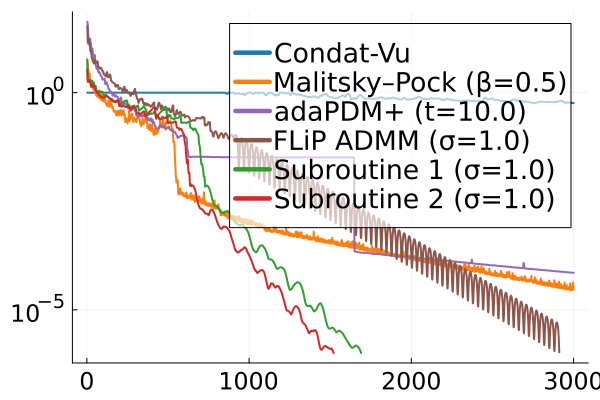}
    \caption{$Q=15, \lambda=20$}
    \label{fig:depth1}
  \end{subfigure}
  \hfill
  \begin{subfigure}{0.32\textwidth}
    \centering
    \includegraphics[width=\linewidth]{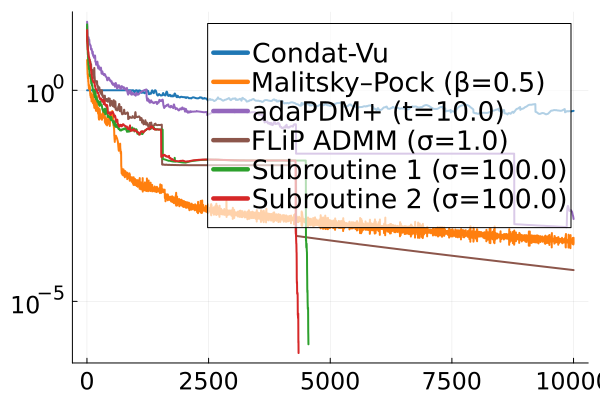}
    \caption{$Q=15, \lambda=20$}
    \label{fig:depth2}
  \end{subfigure}
    \hfill
  \begin{subfigure}{0.32\textwidth}
    \centering
    \includegraphics[width=\linewidth]{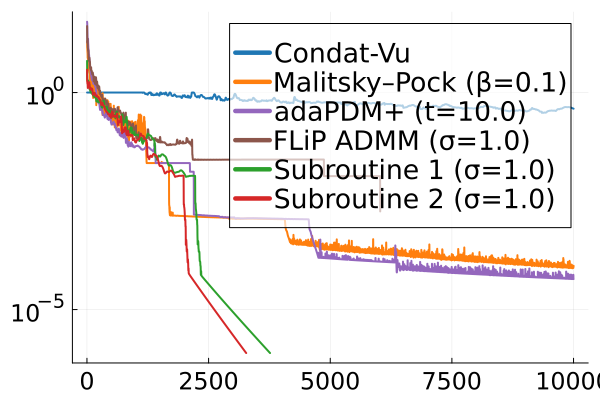}
    \caption{$Q=15, \lambda=20$}
    \label{fig:depth3}
  \end{subfigure}

   \vspace{2ex} 

     \begin{subfigure}{0.32\textwidth}
    \centering
    \includegraphics[width=\linewidth]{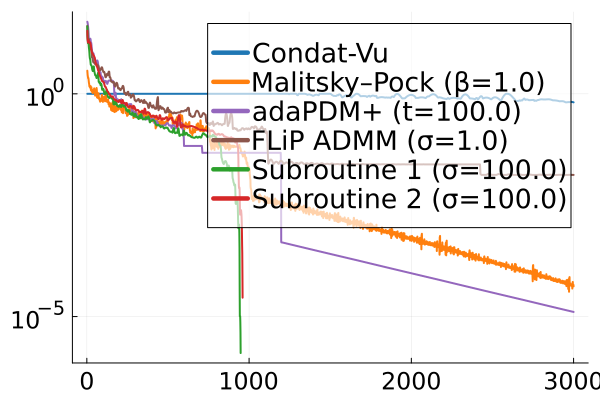}
    \caption{$Q=20, \lambda=20$}
    \label{fig:depth4}
  \end{subfigure}
  \hfill
  \begin{subfigure}{0.32\textwidth}
    \centering
    \includegraphics[width=\linewidth]{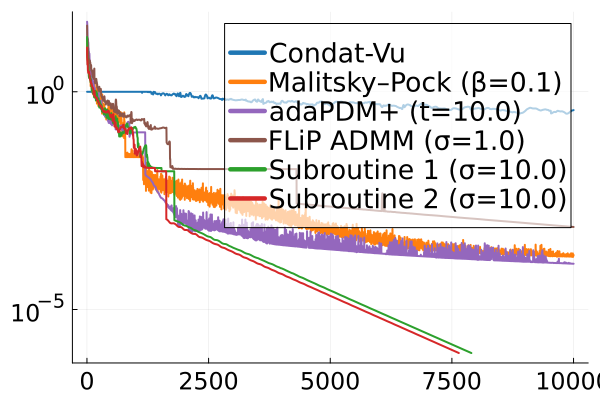}
    \caption{$Q=20, \lambda=20$}
    \label{fig:depth5}
  \end{subfigure}
    \hfill
  \begin{subfigure}{0.32\textwidth}
    \centering
    \includegraphics[width=\linewidth]{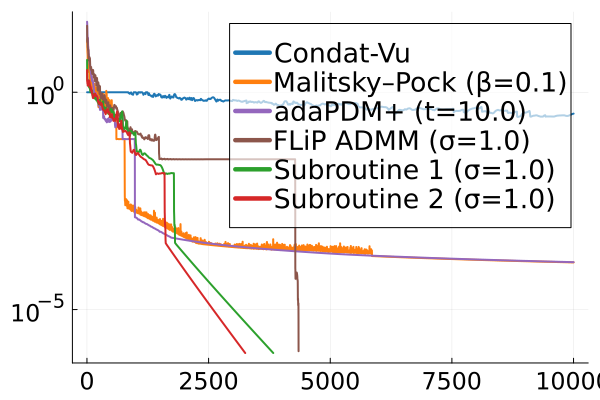}
    \caption{$Q=20, \lambda=20$}
    \label{fig:depth6}
  \end{subfigure}
     \caption{Performance of \ALiA\ vs.\ other baselines on the depth map estimation problem with three different NYU Depth V2 images. \ALiA\ (labeled as Subroutine~\ref{alg:S1} and Subroutine~\ref{alg:S2}) consistently outperforms the non-adaptive and adaptive baselines. We also observe that \ALiA\ with Subroutine~\ref{alg:S2} performs no worse than \ALiA\ with Subroutine~\ref{alg:S1}.}
  \label{fig:depthmap}
\end{figure}

\subsection{Hyperspectral image unmixing}\label{sec: hyperspectral image unmixing}
We consider the hyperspectral image unmixing problem \cite{giampouras2016simultaneously}. Hyperspectral imaging sensors record a high-resolution spectrum at each pixel in a scene, consisting of hundreds of contiguous wavelength bands. Because real-world materials are rarely pure, each pixel's spectrum is typically a mixture of several \textit{endmember} spectra (e.g., vegetation, soil, water), weighted by their fractional abundances. Spectral unmixing is the process of recovering these abundance fractions, which is a core inverse problem in remote sensing and environmental monitoring.

Let $\bm Y\in \R^{L \times K}$ be the observed spectra in a small patch of $K$ neighboring pixels, where $L$ is the number of spectral bands. Let $\bm\Phi \in \R^{L\times N}$ be a dictionary of $N$ endmembers. Our goal is to find a nonnegative matrix $\bm W \in \R^{N \times K}_+$ so that $\bm Y \approx \bm \Phi \bm W$. Two additional challenges arise: each column of $\bm W$ should have only a small number of nonzeros since a given scene typically contains only a few active materials per pixels, and the abundance matrix $\bm W$ over a small window should exhibit a low-rank structure since neighboring pixels often share similar mixtures.

To address these challenges, the authors of \cite{giampouras2016simultaneously} proposed to solve the unmixing problem via the following optimization formulation:
\begin{align}\label{eq: hyperspectral_image_unmixing}
    \min_{\bm W} f_1(\bm W) +  f_2(\bm W) +  f_3(\bm W) +  f_4(\bm W),
\end{align}
where 
\begin{gather*}
    f_1(\bm W) = \frac{1}{2} \|\bm Y - \bm \Phi \bm W\|_F^2,\qquad
    f_2(\bm W) =\gamma \|\bm A \odot \bm W\|_1:= \gamma \sum_{i=1}^N \sum_{j=1}^K a_{ij}|w_{ij}|,\\
    f_3(\bm W) = \tau \|\bm W\|_{\bm b, *} := \tau \sum_{i=1}^{\mathrm{rank}(W)}b_i\sigma_i(\bm W) ,
     \qquad  f_4(\bm W) = \delta_{\{\bm W:~\bm W\geq 0\}}(\bm W).
\end{gather*}
Parameters $\gamma\geq 0$ and $\tau \geq 0$ balance the sparsity (weighted $\ell_1$ norm) and rank (weighted nuclear norm) regularity; the matrix $\bm A$ and the vector $\bm b$ have nonnegative entries $a_{ij}, b_i\geq 0$, and $\sigma_i(\bm W)$ denotes the $i$-th singular value of $\bm W$.

Proximal mappings of $f_1$, $f_2$, $f_3$, and $f_4$ can all be computed exactly and efficiently. Proximal mappings of $f_1$ and $f_4$ are standard.  The proximal mapping of $f_2$ can be calculated through the soft-thresholding operator on $\bm W$ with thresholding parameters $\gamma \bm A$, and the proximal mapping of $f_3$ can be calculated by applying the soft-thresholding operation on singular values of $\bm W$ (after computing SVD of $\bm W$) with thresholding parameters $\tau \bm b$, both in an element-wise manner. In addition, we note the following facts:
\begin{enumerate}
    \item Function $f_1$ is a smooth function.
    \item Function $f_2$ is always convex for any nonnegative $\bm A$, while $f_3$ is convex when all weights are nonegative and  monotonically nonincreasing, i.e., $b_1\ge b_2\ge b_3 \ge\cdots\ge 0$
    \cite[Theorem~3.5.5]{horn1994topics}. 
    When the monotonicity condition is violated, $f_3$ is nonconvex.
    \item Define $\tilde{f}_2: = f_2 + f_4$. The proximal mapping of $\tilde{f}_2$ can be computed exactly as well; it is a one-sided implementation of the proximal mapping of $f_2$.  
\end{enumerate}

We largely follow \cite{giampouras2016simultaneously} to prepare the data for our experiments. We generate $\bm \Phi\in \R_+^{L \times N}$ from the USGS Spectral Library (splib06a) AVIRIS 1995 convolved data \cite{clark2007usgs}, where $L=224$ rows represent spectral bands uniformly distributed in the range $0.4\text{--}2.5\mu m$ and $N=50$ columns represent endmembers randomly selected from the library. Then we generate $K=9$ (corresponding to a $3\times 3$ square sliding window) random abundance vectors on the simplex to form $\bm W^0 \in \R^{N \times K}$, and contaminate the product $\bm \Phi \bm W^0$ by Gaussian noise to obtain $Y$ such that SNR $=30dB$. In view of definitions of $f_2$ and $f_3$, we compute a least square estimate $\bm W^{\mathrm{LS}}$ of $\bm W$ and prepare $\bm A$ and $\bm b$ as $a_{ij} = (w^{\mathrm{LS}}_{ij} + \epsilon)^{-1}$ and 
\[
b_i = b=\mathop{\mathrm{mean}}_{i=1,\dots,\mathrm{rank}(W)}\Big\{\frac{1}{\sigma_{i}(\bm W^{\mathrm{LS}}) + \epsilon}\Big\}
\]
with $\epsilon = 10^{-16}$ so that $f_3$ is convex.

\paragraph{Convex formulations.}
We compare the performance of the proposed methods with FLiP ADMM on several reformulations of \eqref{eq: hyperspectral_image_unmixing} with different numbers of variable blocks.

In the \texttt{2blocks} reformulation, we treat $f_1$ and $\tilde{f_2}$ as smooth and proximable functions of the first block, respectively, and $f_3$ as the proximable function of the second block: 
\begin{align*}
\begin{array}{ll}
\underset{\bm W, \bm Z}{\mbox{minimize}}
&f_1(\bm W) + \tilde{f}_2(\bm W) + f_3(\bm Z)\\
\mbox{subject to}&\bm W - \bm Z = 0.
\end{array}
\end{align*}
In the \texttt{3blocks} reformulation, we treat $f_1$ as the proximable function of the first block, $\tilde{f}_2$ as the proximable function of the second block, and $f_3$ as the proximable function of the third block: 
\begin{align*}\
\begin{array}{ll}
\underset{\bm W_1, \bm W_2, \bm Z}{\mbox{minimize}}
&f_1(\bm W_1) + \tilde{f}_2(\bm W_2) + f_3(\bm Z)\\
\mbox{subject to}&
\bm W_1 - \bm Z = 0\\
&\bm W_2 - \bm Z = 0
\end{array}
\end{align*}
Finally, the \texttt{4blocks} reformulation is similar to the \texttt{3blocks} reformulation, except that we use an additional fourth block to enforce consensus: 
\begin{align*}
\begin{array}{ll}
\underset{\bm W_1, \bm W_2, \bm W_3, \bm Z} {\mbox{minimize}}
&f_1(\bm W_1) + \tilde{f}_2(\bm W_2) + f_3(\bm W_3)\\
\mbox{subject to}&
\bm W_1 - \bm Z = 0\\
&\bm W_2 - \bm Z = 0\\
&\bm W_3 - \bm Z = 0.
\end{array}
\end{align*}

We apply \ALiA\ and other baseline methods to solve this problem. Figure~\ref{fig:hyper-all} reports the results with the three different convex formulations. We see that \ALiA\ performs competitively across the \texttt{3blocks} and \texttt{4blocks} formulations, but not in the \texttt{2blocks} reformulation. In the following, we discuss the reasons for this phenomenon.

\paragraph{Nonconvex formulations.}
We also experiment \ALiA\ when $f_3$ is nonconvex, i.e., when $b_i<b_{i+1}$ for some $i$. We take $b_i = \frac{1}{\sigma_i(\bm W^{\mathrm{LS}})+\epsilon}$ for $i=1,\dots,\mathrm{rank}(W)$, which makes $\{b_i\}_{i=1,\dots,\mathrm{rank}(W)}$ non-descending, and the proximal mapping of $f_3$ can be computed exactly the same way as in the convex case since $\{b_i\}_{i=1,\dots,\mathrm{rank}(W)}$ is monotone \cite{gu2014weighted}.

More generally, proximal mappings can be well-defined for some convex functions, and the notion of \emph{prox-regularity} provides a generalization under which proximal operators remain well-defined \cite{poliquin1996prox}.  Whether \ALiA\ is effective in such nonconvex setups, despite our theory not being directly applicable, is an interesting direction for future work, and one that we preliminarily explore in this experiment.

Figure~\ref{fig:hyper-all} compares \ALiA\ against other baselines applied three convex and three nonconvex formulations of this problems. We see that \ALiA\ is generally competitive with the best-performing baselines in the \texttt{3blocks} and \texttt{4blocks} formulations, but not necessarily in the \texttt{2blocks} reformulation. In the following, we discuss the reasons for this phenomenon.

\begin{figure}[H]
  \centering
    \begin{subfigure}{0.32\textwidth}
      \includegraphics[width=\linewidth]{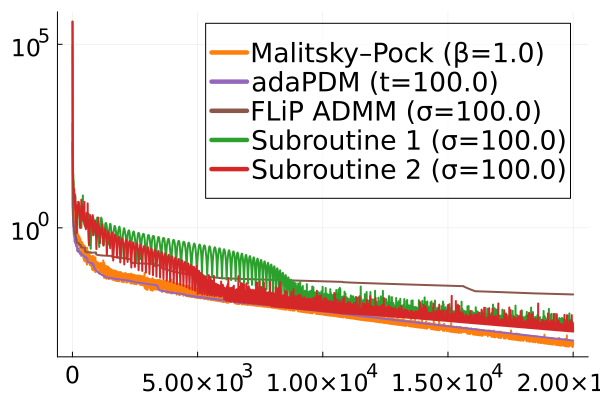}
      \caption{Convex, 2 blocks, $\gamma=\tau=0.1$}
    \end{subfigure}\hfill
    \begin{subfigure}{0.32\textwidth}
      \includegraphics[width=\linewidth]{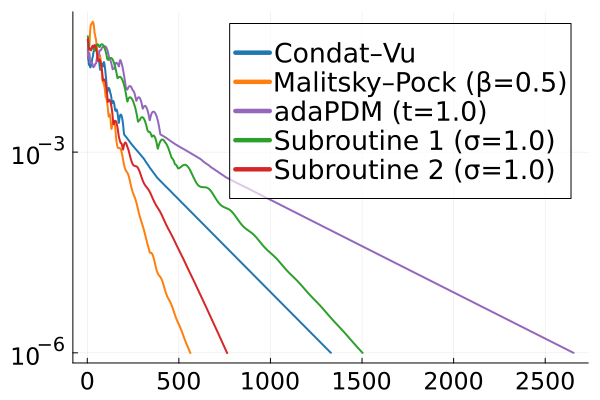}
      \caption{Convex, 3 blocks, $\gamma=\tau=0.1$}
    \end{subfigure}\hfill
    \begin{subfigure}{0.32\textwidth}
      \includegraphics[width=\linewidth]{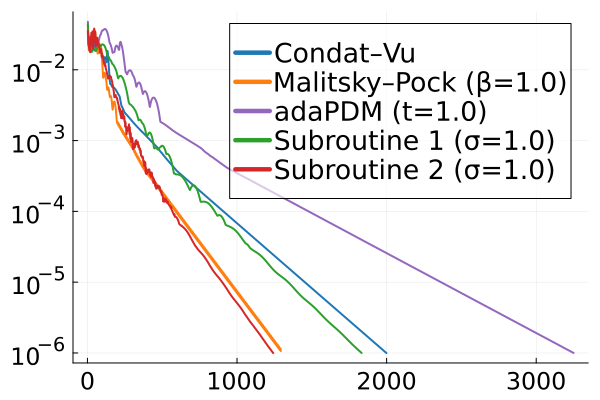}
      \caption{Convex, 4 blocks, $\gamma=\tau=0.1$}
    \end{subfigure}

    \begin{subfigure}{0.32\textwidth}
      \includegraphics[width=\linewidth]{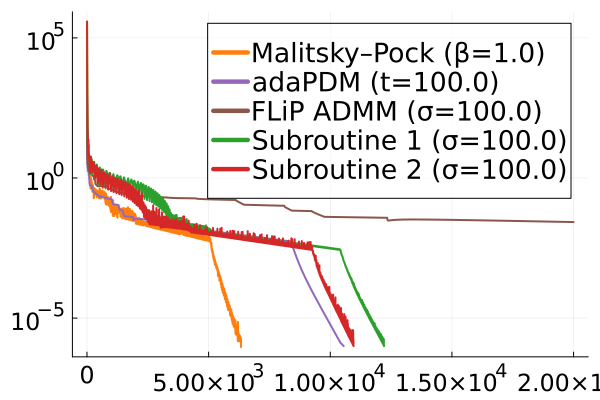}
      \caption{Convex, 2 blocks, $\gamma=0.5, \tau=0.05$}
    \end{subfigure}\hfill
    \begin{subfigure}{0.32\textwidth}
      \includegraphics[width=\linewidth]{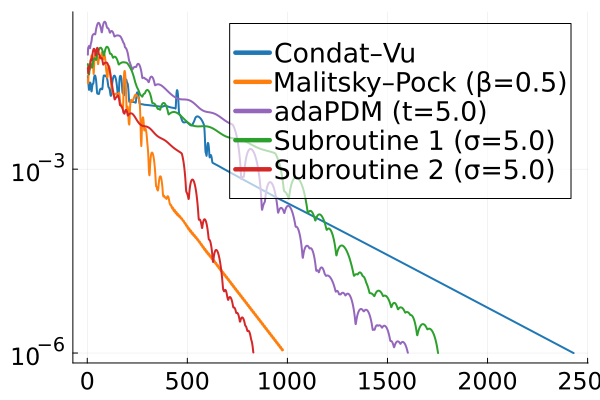}
      \caption{Convex, 3 blocks, $\gamma=0.5, \tau=0.05$}
    \end{subfigure}\hfill
    \begin{subfigure}{0.32\textwidth}
      \includegraphics[width=\linewidth]{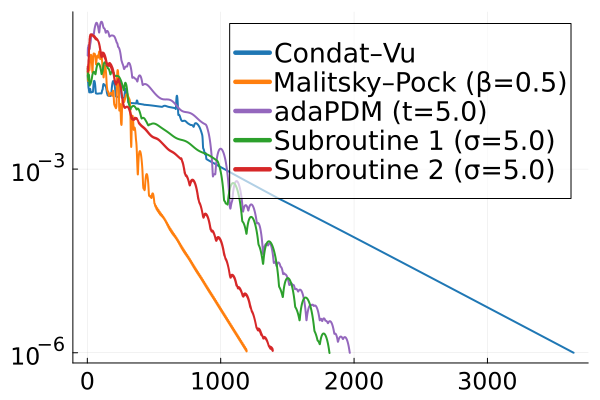}
      \caption{Convex, 4 blocks, $\gamma=0.5, \tau=0.05$}
    \end{subfigure}

  \begin{subfigure}{\textwidth}
    \centering
    \begin{subfigure}{0.32\textwidth}
      \includegraphics[width=\linewidth]{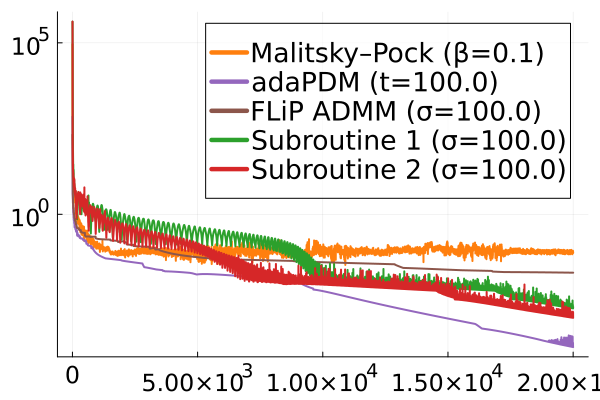}
      \caption{Nonconvex, 2 blocks, $\gamma=\tau=0.1$}
    \end{subfigure}\hfill
    \begin{subfigure}{0.32\textwidth}
      \includegraphics[width=\linewidth]{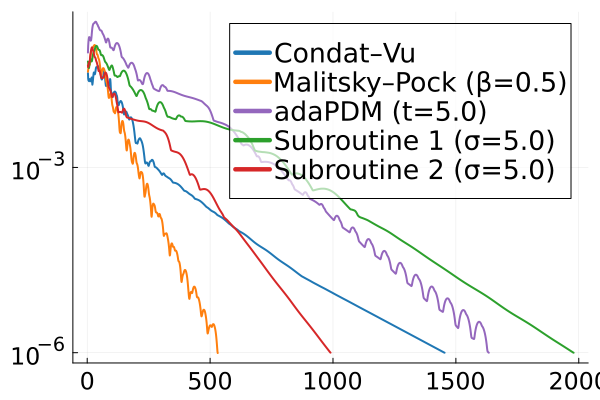}
      \caption{Nonconvex, 3 blocks, $\gamma=\tau=0.1$}
    \end{subfigure}\hfill
    \begin{subfigure}{0.32\textwidth}
      \includegraphics[width=\linewidth]{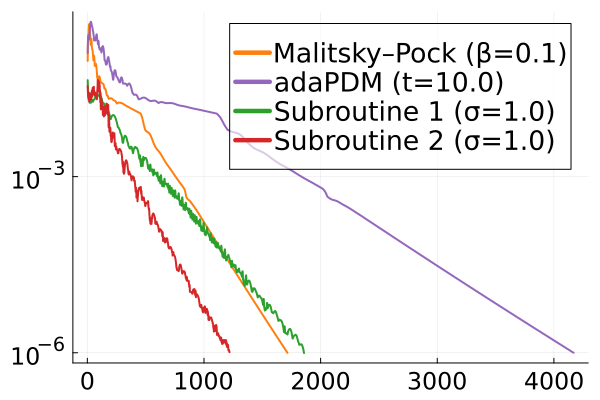}
      \caption{Nonconvex, 4 blocks, $\gamma=\tau=0.1$}
    \end{subfigure}

    \begin{subfigure}{0.32\textwidth}
      \includegraphics[width=\linewidth]{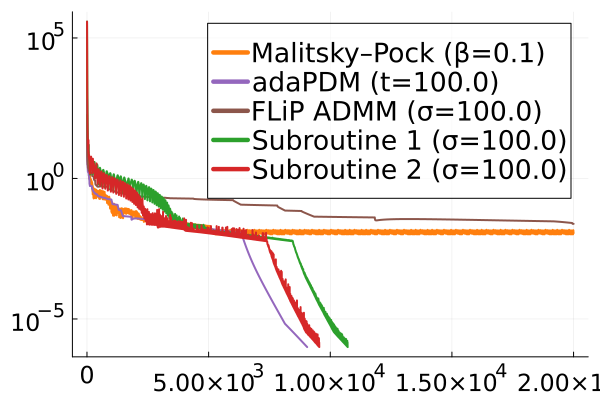}
      \caption{Nonconvex, 2 blocks, $\gamma=0.5, \tau=0.05$}
    \end{subfigure}\hfill
    \begin{subfigure}{0.32\textwidth}
      \includegraphics[width=\linewidth]{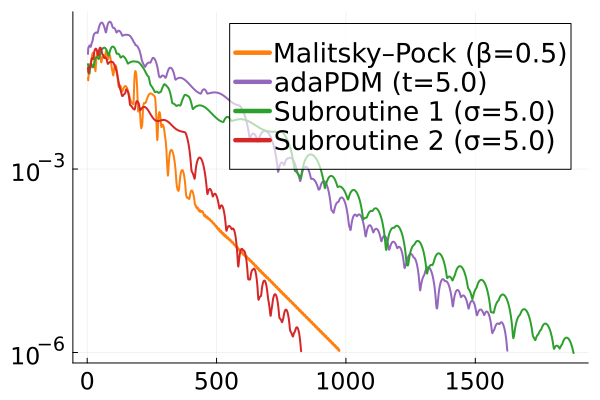}
      \caption{Nonconvex, 3 blocks, $\gamma=0.5, \tau=0.05$}
    \end{subfigure}\hfill
    \begin{subfigure}{0.32\textwidth}
      \includegraphics[width=\linewidth]{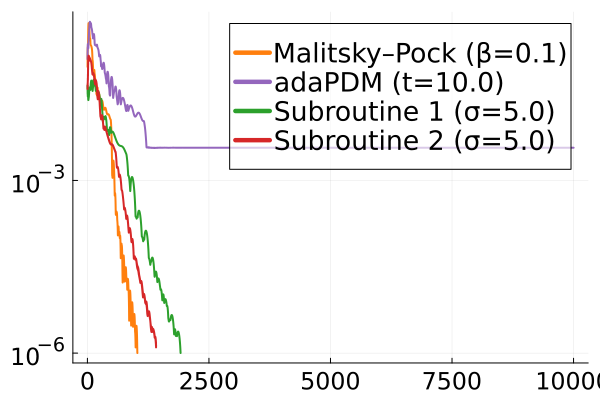}
      \caption{Nonconvex, 4 blocks, $\gamma=0.5, \tau=0.05$}
    \end{subfigure}

    \label{fig:hyper-nonconvex}
  \end{subfigure}

  \caption{Performance of \ALiA\ versus the baseline methods on the hyperspectral
  image unmixing problem under both convex and nonconvex block formulations.
  \ALiA\ converges competitively across all settings.
Moreover, its convergence behavior is stable with respect to the choice of
  block formulation, whereas the competing methods often exhibit significant
  slowdowns for certain block formulations.  We
note that the Malistky--Pock method relies on backtracking line searches, which incur a higher per-iteration
cost, whereas ALiA does not.
 We also observe that \ALiA\ with Subroutine~\ref{alg:S2} performs no worse than \ALiA\ with Subroutine~\ref{alg:S1}. Methods not shown in the plot did not converge within the given number of iterations.
  }
  \label{fig:hyper-all}
\end{figure}

\paragraph{Loss of adaptivity.} 
In the \texttt{2blocks} formulation, the square-root term in $\gamma_{k+1}=\min\{\cdots, \, \Gamma_x, \,  \Gamma_y\}$ of Subroutine~\ref{alg:S1} and  Subroutine~\ref{alg:S2} rapidly converges to a fixed constant. As a result, $\gamma_k$ also converges to a constant, effectively eliminating adaptivity and leading to slower overall convergence. This loss of adaptivity may help explain why \ALiA\ is slower than other baselines for the \texttt{2blocks} formulation. Developing a modification to preserve adaptivity in such cases is a promising direction for future research.

\subsection{Simpler synthetic tasks} \label{sec: compare with PDM}

We consider 
\[
\begin{array}{ll}
\underset{x\in \mathbb{R}^p}{\mbox{minimize}}
&f(x)+g(x)+h(Ax),
\end{array} 
\]
where $f$ is convex differentiable, and $g$ and $h$ are CCP. This problem class was discussed earlier as \eqref{eq:condatvu} and is a strict subclass of the problem class that \ALiA\ targets, namely \eqref{eq:linadmm}. Following the experimental setups described in \cite{latafat2024adaptive}, we utilize several benchmark datasets from the LIBSVM library \cite{chang2011libsvm}. Specifically, we evaluate adaPDM, adaPDM+ \cite{latafat2024adaptive}, the Malitsky--Pock method \cite[Algorithm 4]{malitsky2018first}, and the classical (non-adpative) Condat--V\~u method \cite{condat2013primal}.

When \(A\) is a mere vector, computing \(\|A\|\) is straightforward (it is the Euclidean norm), making adaPDM computationally preferable. Following the convention of \cite{latafat2024adaptive}, we use adaPDM when \(A\) is vector-valued and adaPDM+ otherwise.

\paragraph{Dual lasso.}
We consider the dual lasso problem
\[
\begin{array}{ll}
 \underset{x\in \mathbb{R}^m}{\mathrm{minimize}}  \quad  &\frac{1}{4}\|x\|_2^2 -b^\top x\\
 \mbox{subject to \,} \ &\|A^\top x\|_{\infty}\le \lambda,
\end{array}
\]
where $A\in \mathbb{R}^{m\times n}$, $b\in \mathbb{R}^m$, and $\lambda>0$.
Figure~\ref{fig:lasso} reports the results.

\begin{figure}[ht]
  \centering
     \begin{subfigure}{0.32\textwidth}
    \centering
    \includegraphics[width=\linewidth]{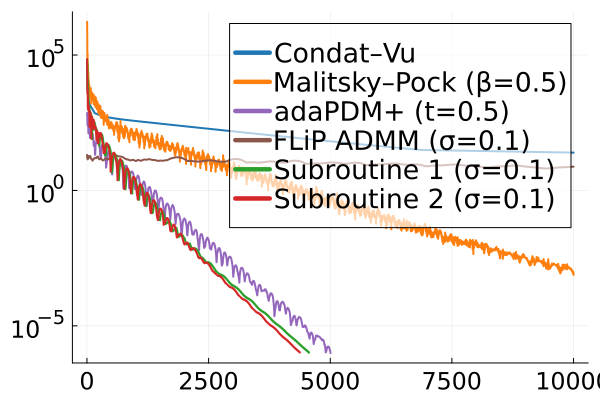}
    \caption{abalone, $m=4177$, $n=8$}
    \label{fig:Duallasso_abalone}
  \end{subfigure}
  \hfill
  \begin{subfigure}{0.32\textwidth}
    \centering
    \includegraphics[width=\linewidth]{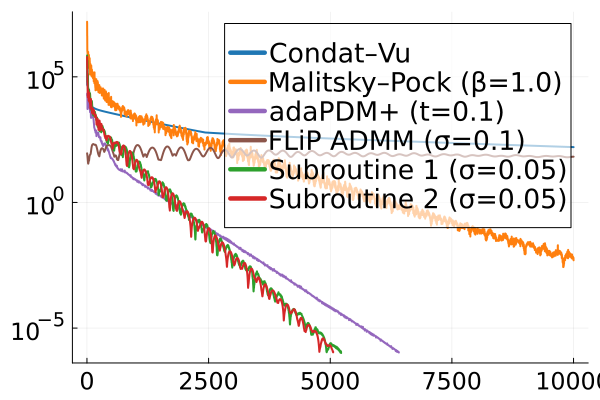}
    \caption{cpusmall scale, $m=8192$, $n=12$}
    \label{fig:Duallasso_cpu}
  \end{subfigure}
    \hfill
  \begin{subfigure}{0.32\textwidth}
    \centering
    \includegraphics[width=\linewidth]{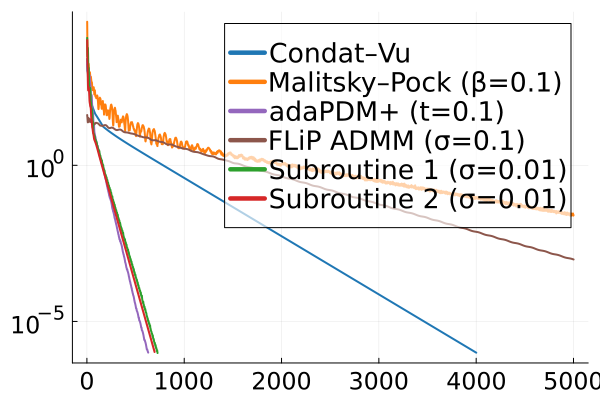}
    \caption{housing scale, $m=506$, $n=13$}
    \label{fig:Duallasso_housing_scale}
  \end{subfigure}
     \caption{Performance of \ALiA\ vs.\ baseline methods on the dual lasso problem with $\lambda=0.1$ and datasets abalone, cpusmall scale, and housing scale. \ALiA\ generally outperforms, or is competitive with, the other adaptive baselines. We also observe that \ALiA\ with Subroutine~\ref{alg:S2} performs no worse than \ALiA\ with Subroutine~\ref{alg:S1}.}
  \label{fig:lasso}
\end{figure}  

\paragraph{Dual least absolute deviation.}
We consider the dual least absolute deviation (LAD) problem 
\[
\begin{array}{ll}
    \underset{x\in\mathbb{R}^m}{\mathrm{minimize}}  \quad  &b^\top x \\
    \mbox{subject to} &  \|x\|_{\infty}\le 1 \\
    &\|A^\top x\|_{\infty}\le \lambda,
\end{array}
\]
where the same data as in the dual lasso problem is used. Figure~\ref{fig:LAD} reports the results. 

\begin{figure}[ht]
  \centering
     \begin{subfigure}{0.32\textwidth}
    \centering
    \includegraphics[width=\linewidth]{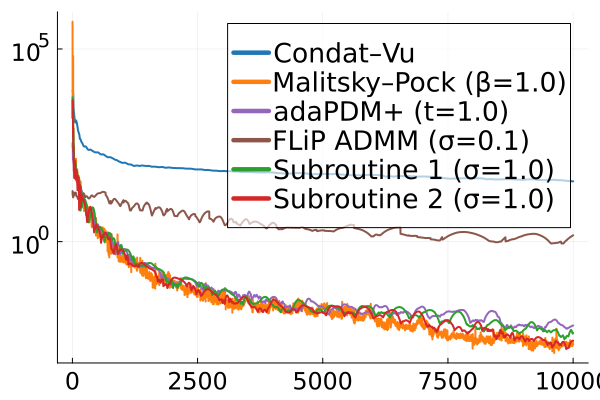}
    \caption{abalone, $m=4177$, $n=8$}
    \label{fig:Duallad_abalone}
  \end{subfigure}
  \hfill
  \begin{subfigure}{0.32\textwidth}
    \centering
    \includegraphics[width=\linewidth]{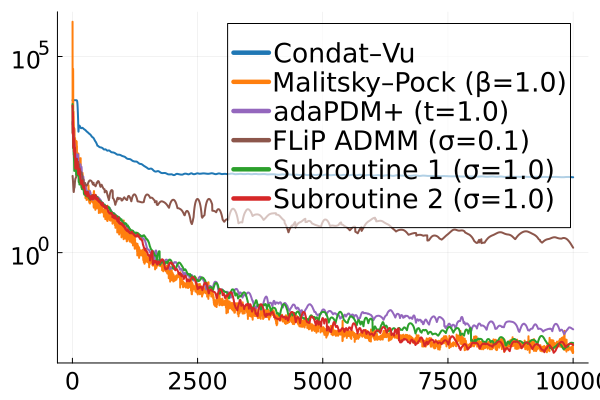}
    \caption{cpusmall scale, $m=8192$, $n=12$}
    \label{fig:Duallad_cpu}
  \end{subfigure}
    \hfill
  \begin{subfigure}{0.32\textwidth}
    \centering
    \includegraphics[width=\linewidth]{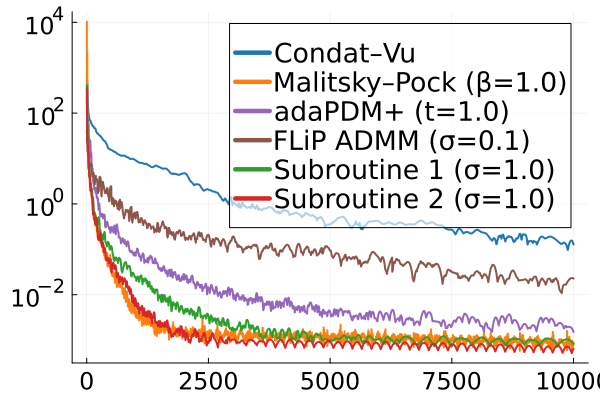}
    \caption{housing scale, $m=506$, $n=13$}
    \label{fig:Duallad_housing_scale}
  \end{subfigure}
     \caption{Performance of \ALiA\ vs.\ baseline methods on the least absolute deviation problem with $\lambda=0.1$ and datasets abalone, cpusmall scale, and housing scale. \ALiA\ is competitive with other adaptive methods. We note that other adaptive methods rely on backtracking line searches, which incur a higher per-iteration cost, whereas \ALiA\ does not. We also observe that \ALiA\ with Subroutine~\ref{alg:S2} performs no worse than \ALiA\ with Subroutine~\ref{alg:S1}.}
  \label{fig:LAD}
\end{figure}  

\paragraph{Dual SVM.}
Finally, we consider the dual support vector machine problem
\[
\begin{array}{ll}
    \underset{x\in\mathbb{R}^m}{\mathrm{minimize}} \quad 
    &\frac{1}{2}x^\top Qx - \mathbf{1}^\top x\\
    \mbox{subject to} \quad 
    &0\le x_i \le C,\quad i=1,\ldots,m\\
    &y^\top x = 0,
\end{array}
\]
where $Q\in\mathbb{R}^{m\times m}$, $y\in \mathbb{R}^m$,  and $C>0$.
Figure~\ref{fig:SVM} reports the results.

\begin{figure}[ht]
  \centering
     \begin{subfigure}{0.32\textwidth}
    \centering
    \includegraphics[width=\linewidth]{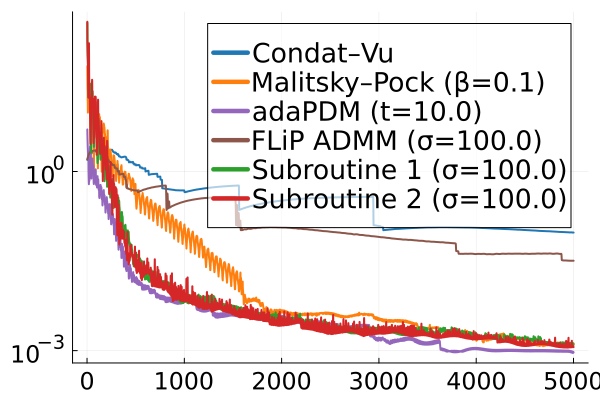}
    \caption{svmguide3, $m=1243$, $n=22$}
    \label{fig:DualSVM_svmguide3}
  \end{subfigure}
  \hfill
  \begin{subfigure}{0.32\textwidth}
    \centering
    \includegraphics[width=\linewidth]{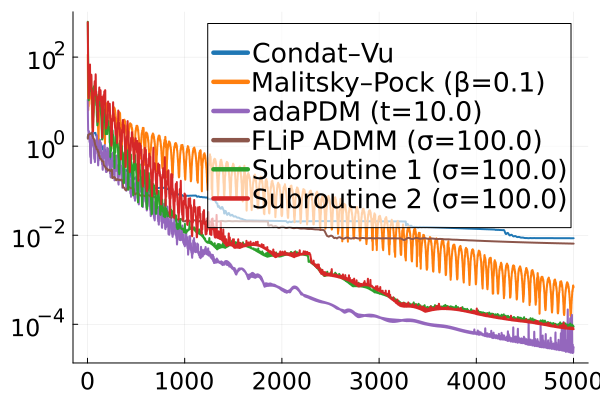}
    \caption{heart scale, $m=270$, $n=13$}
    \label{fig:DualSVM_heart_scale}
  \end{subfigure}
    \hfill
  \begin{subfigure}{0.32\textwidth}
    \centering
    \includegraphics[width=\linewidth]{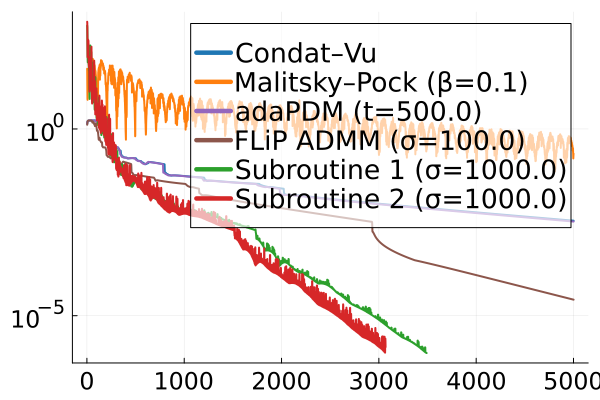}
    \caption{synthetic data, $m=400$, $n=20$}
    \label{fig:DualSVM_synthetic}
  \end{subfigure}
     \caption{Performance of \ALiA\ vs.\ baseline methods on the dual SVM problem with $C=0.1$ and datasets svmguide3, heart scale, and synthetic data.  \ALiA\ is competitive with other adaptive methods. We also observe that \ALiA\ with Subroutine~\ref{alg:S2} performs no worse than \ALiA\ with Subroutine~\ref{alg:S1}.}
     \label{fig:SVM}
\end{figure}  

\section{Conclusion}\label{sec:5}
In this paper, we present \ALiA, an enhancement of FLiP-ADMM that adaptively adjusts its stepsizes. Under convexity and differentiability assumptions (but without any smoothness or Lipschitz assumptions), we prove that the iterates of \ALiA\ converge to a primal-dual solution. In particular, convergence is ensured without requiring knowledge of problem-specific parameters, and the user only needs to know qualitative properties of the objective, namely convexity and differentiability. In this sense, the convergence guarantee of \ALiA\ mirrors the generality of the guarantee of classical ADMM. Empirically, \ALiA\ also consistently outperforms FLiP-ADMM and other adaptive baselines across a range of real-world and synthetic benchmark problems.

\ALiA\ still has one hyperparameter that must be tuned by the user, namely the primal-dual stepsize ratio $\sigma$. (So, \ALiA\ is fully adaptive with respect to the dual stepsize $\gamma_k$, but not with respect to $\sigma$.) This type of hyperparameter is inherent to essentially all ADMM- and primal-dual-type methods; other adaptive primal-dual algorithms, such as the Malitsky--Pock method \cite[Algorithm~4]{malitsky2018first} and adaPDM \cite{latafat2024adaptive}, have analogous parameters governing the relative weighting of primal versus dual errors. Although convergence is often guaranteed for any choice of this parameter, the constants of the convergence rate can be significantly affected.

Due to the importance of choosing this primal-dual stepsize ratio hyperparameter $\sigma$, there is a body of prior work on its automatic or adaptive selection \cite{he2000alternating,wang2001decomposition,ghadimi2014optimal,raghunathan2014alternating,nishihara2015general,wohlberg2017admm,xu2017adaptive,mccann2024robust}. However, to the best of our knowledge, there is no method that adaptively chooses this hyperparameter for the general class of convex (not necessarily quadratic) problem instances with a convergence guarantee. Therefore, achieving adaptivity with respect to $\sigma$ would be a valuable and promising direction for future work.

Another promising direction for future work is to pursue more refined convergence analyses and establish convergence \emph{rates}. As mentioned (albeit not emphasized) in Section~\ref{ss:thm1-sketch}, we can obtain the following rate on the saddle function value:
\[
\min_{k=0,\dots,K}\mathbf{L}({x}^k,{y}^k,u^\star)- \mathbf{L}(x^\star,y^\star,u^\star)\le
O(1/K).
\]
Prior works on ADMM also derive rates for other quantities, such as objective function values and constraint residuals. Obtaining such guarantees for \ALiA\ would therefore be a natural and interesting avenue of future work.

\bibliographystyle{spmpsci} 
\bibliography{admm}

@article{pustelnik2017proximity,
  title={Proximity operator of a sum of functions; application to depth map estimation},
  author={Pustelnik, Nelly and Condat, Laurent},
  journal={IEEE Signal Processing Letters},
  volume={24},
  number={12},
  pages={1827--1831},
  year={2017},
  publisher={IEEE}
}

@article{chambolle2012convex,
  title={A convex approach to minimal partitions},
  author={Chambolle, Antonin and Cremers, Daniel and Pock, Thomas},
  journal={SIAM Journal on Imaging Sciences},
  volume={5},
  number={4},
  pages={1113--1158},
  year={2012},
  publisher={SIAM}
}

@book{horn1994topics,
  title={Topics in matrix analysis},
  author={Horn, Roger A and Johnson, Charles R},
  year={1994},
  publisher={Cambridge university press}
}

@article{malitsky2018first,
  title={A first-order primal-dual algorithm with linesearch},
  author={Malitsky, Yura and Pock, Thomas},
  journal={SIAM Journal on Optimization},
  volume={28},
  number={1},
  pages={411--432},
  year={2018},
  publisher={SIAM}
}

@article{chang2022golden,
  title={Golden ratio primal-dual algorithm with linesearch},
  author={Chang, Xiao-Kai and Yang, Junfeng and Zhang, Hongchao},
  journal={SIAM Journal on Optimization},
  volume={32},
  number={3},
  pages={1584--1613},
  year={2022},
  publisher={SIAM}
}

@book{nesterov2018lectures,
  title   = {Lectures on Convex Optimization},
  author  = {Y. Nesterov},
  volume  = {137},
  edition   = {2},
  year    = {2018},
  publisher = {Springer}
}

@book{press2007numerical,
  title={Numerical recipes 3rd edition: The art of scientific computing},
  author={Press, William H},
  year={2007},
  publisher={Cambridge university press}
}

@incollection{fortin1983chapter,
  title={Chapter III on decomposition-coordination methods using an augmented lagrangian},
  author={Fortin, Michel and Glowinski, Roland},
  booktitle={Studies in Mathematics and Its Applications},
  volume={15},
  pages={97--146},
  year={1983},
  publisher={Elsevier}
}

@article{wang2001decomposition,
  title={Decomposition method with a variable parameter for a class of monotone variational inequality problems},
  author={Wang, SL and Liao, LZ},
  journal={Journal of optimization theory and applications},
  volume={109},
  number={2},
  pages={415--429},
  year={2001},
  publisher={Springer}
}

@book{bauschke2017convex,
  author    = {Heinz H. Bauschke and Patrick L. Combettes},
  title     = {Convex Analysis and Monotone Operator Theory in Hilbert Spaces},
  edition   = {2nd},
  publisher = {Springer},
  year      = {2017}
}

@book{rockafellar1970convex,
  title={Convex Analysis},
  author={Rockafellar, Ralph Tyrell},
  year={1970},
  publisher={Princeton University Press},
  address={Princeton, NJ}
}

@article{suh2025adaptive,
  title={An Adaptive and Parameter-Free {N}esterov's Accelerated Gradient Method for Convex Optimization},
  author={Suh, Jaewook J and Ma, Shiqian},
  journal={arXiv preprint arXiv:2505.11670},
  year={2025}
}

@article{vladarean2021first,
  title={A first-order primal-dual method with adaptivity to local smoothness},
  author={Vladarean, Maria-Luiza and Malitsky, Yura and Cevher, Volkan},
  journal={Advances in neural information processing systems},
  volume={34},
  pages={6171--6182},
  year={2021}
}

@article{armijo1966minimization,
  title={Minimization of functions having Lipschitz continuous first partial derivatives},
  author={Armijo, Larry},
  journal={Pacific Journal of mathematics},
  volume={16},
  number={1},
  pages={1--3},
  year={1966},
  publisher={Mathematical Sciences Publishers}
}

@article{ayer1955empirical,
  title={An empirical distribution function for sampling with incomplete information},
  author={Ayer, Miriam and Brunk, H Daniel and Ewing, George M and Reid, William T and Silverman, Edward},
  journal={The annals of mathematical statistics},
  pages={641--647},
  year={1955},
  publisher={JSTOR}
}

@article{giampouras2016simultaneously,
  title={Simultaneously sparse and low-rank abundance matrix estimation for hyperspectral image unmixing},
  author={Giampouras, Paris V and Themelis, Konstantinos E and Rontogiannis, Athanasios A and Koutroumbas, Konstantinos D},
  journal={IEEE Transactions on Geoscience and Remote Sensing},
  volume={54},
  number={8},
  pages={4775--4789},
  year={2016},
  publisher={IEEE}
}

@techreport{clark2007usgs,
  title={USGS digital spectral library splib06a},
  author={Clark, Roger N and Swayze, Gregg A and Wise, Richard A and Livo, K Eric and Hoefen, Todd M and Kokaly, Raymond F and Sutley, Stephen J},
  year={2007},
  institution={US Geological Survey}
}

@article{GlowinskiMarroco1975_lapproximation,
  title = {Sur l'approximation, Par \'El\'ements Finis d'ordre Un, et La R\'esolution, Par P\'enalisation-Dualit\'e d'une Classe de Probl\`emes de {{Dirichlet}} Non Lin\'eaires},
  author = {Glowinski, R. and Marroco, A.},
  year = {1975},
  journal = {Revue Fran\c{c}aise d'Automatique, Informatique, Recherche Op\'erationnelle. Analyse Num\'erique},
  volume = {9},
  number = {2},
  pages = {41--76},
  ajournal = {Rev. Fr. d'Autom. Inform. Rech. Oper. Anal. Numer.}
}

@article{lan2024auto,
  title={Auto-conditioned primal-dual hybrid gradient method and alternating direction method of multipliers},
  author={Lan, Guanghui and Li, Tianjiao},
  journal={arXiv preprint arXiv:2410.01979},
  year={2024}
}

@article{GabayMercier1976_dual,
  title = {A Dual Algorithm for the Solution of Nonlinear Variational Problems via Finite Element Approximation},
  author = {Gabay, D. and Mercier, B.},
  year = {1976},
  journal = {Computers and Mathematics with Applications},
  volume = {2},
  number = {1},
  pages = {17--40}
}

@article{boyd2011distributed,
  title={Distributed optimization and statistical learning via the alternating direction method of multipliers},
  author={Boyd, S. and Parikh, N. and Chu, E. and Peleato, B. and Eckstein, J.},
  journal={Foundations and Trends{\textregistered} in Machine learning},
  volume={3},
  number={1},
  pages={1--122},
  year={2011},
 }

@article{xu2017adaptive,
  title={Adaptive {ADMM} with spectral penalty parameter selection},
  author={Xu, Z. and Figueiredo, M. and Goldstein, T.},
  journal={Artificial Intelligence and Statistics},
  year={2017},
}

@article{xu2017adaptive2,
  title={Adaptive consensus {ADMM} for distributed optimization},
  author={Xu, Z. and Taylor, G. and Li, H. and Figueiredo, M. A. T. and Yuan, X. and Goldstein, T.},
  journal={International Conference on Machine Learning},
  year={2017},
}

@article{mccann2024robust,
  title={Robust and Simple {ADMM} Penalty Parameter Selection},
  author={McCann, M. T. and Wohlberg, B.},
  journal={IEEE Open Journal of Signal Processing},
  volume={5},
  pages={402--420},
  year={2024},
  publisher={IEEE}
}

@book{ryu2022large,
  title={Large-Scale Convex Optimization: Algorithms \& Analyses via Monotone Operators},
  author={E. K. Ryu and Yin, W.},
  year={2022},
  publisher={Cambridge University Press}
}

@article{latafat2024adaptive,
  title={Adaptive proximal algorithms for convex optimization under local Lipschitz continuity of the gradient},
  author={Latafat, Puya and Themelis, Andreas and Stella, Lorenzo and Patrinos, Panagiotis},
  journal={Mathematical Programming},
  pages={433--471},
  year={2025},
  volume = {213},
}

@article{malitsky2024adaptive,
  title={Adaptive proximal gradient method for convex optimization},
  author={Malitsky, Yura and Mishchenko, Konstantin},
  journal={Advances in Neural Information Processing Systems},
  volume={37},
  pages={100670--100697},
  year={2024}
}

@article{vu2013splitting,
  title={A splitting algorithm for dual monotone inclusions involving cocoercive operators},
  author={V{\~u}, Bằng C{\^o}ng},
  journal={Advances in Computational Mathematics},
  volume={38},
  number={3},
  pages={667--681},
  year={2013},
  publisher={Springer}
}

@article{malitsky2020golden,
  title={Golden ratio algorithms for variational inequalities},
  author={Malitsky, Yura},
  journal={Mathematical Programming},
  volume={184},
  number={1},
  pages={383--410},
  year={2020},
  publisher={Springer}
}

@article{he20121,
  title={On the ${O}(1/n)$ convergence rate of the Douglas--Rachford alternating direction method},
  author={He, Bingsheng and Yuan, Xiaoming},
  journal={SIAM Journal on Numerical Analysis},
  volume={50},
  number={2},
  pages={700--709},
  year={2012},
  publisher={SIAM}
}

@article{deng2016global,
  title={On the global and linear convergence of the generalized alternating direction method of multipliers},
  author={Deng, Wei and Yin, Wotao},
  journal={Journal of Scientific Computing},
  volume={66},
  pages={889--916},
  year={2016},
  publisher={Springer}
}

@article{gao2019randomized,
  title={Randomized primal--dual proximal block coordinate updates},
  author={Gao, Xiang and Xu, Yang-Yang and Zhang, Shu-Zhong},
  journal={Journal of the Operations Research Society of China},
  volume={7},
  number={2},
  pages={205--250},
  year={2019},
  publisher={Springer}
}

@article{ouyang2015accelerated,
  title={An accelerated linearized alternating direction method of multipliers},
  author={Ouyang, Yuyuan and Chen, Yunmei and Lan, Guanghui and Pasiliao Jr, Eduardo},
  journal={SIAM Journal on Imaging Sciences},
  volume={8},
  number={1},
  pages={644--681},
  year={2015},
  publisher={SIAM}
}

@article{goldstein2014fast,
  title={Fast alternating direction optimization methods},
  author={Goldstein, Tom and O'Donoghue, Brendan and Setzer, Simon and Baraniuk, Richard},
  journal={SIAM Journal on Imaging Sciences},
  volume={7},
  number={3},
  pages={1588--1623},
  year={2014},
  publisher={SIAM}
}

@article{gao2018information,
  title={On the information-adaptive variants of the {ADMM}: an iteration complexity perspective},
  author={Gao, Xiang and Jiang, Bo and Zhang, Shuzhong},
  journal={Journal of Scientific Computing},
  volume={76},
  pages={327--363},
  year={2018},
  publisher={Springer}
}

@inproceedings{ouyang2013stochastic,
  title={Stochastic alternating direction method of multipliers},
  author={Ouyang, Hua and He, Niao and Tran, Long and Gray, Alexander},
  booktitle={International conference on machine learning},
  pages={80--88},
  year={2013},
  organization={PMLR}
}

@article{barzilai1988two,
  title={Two-point step size gradient methods},
  author={Barzilai, Jonathan and Borwein, Jonathan M},
  journal={IMA journal of numerical analysis},
  volume={8},
  number={1},
  pages={141--148},
  year={1988},
  publisher={Oxford University Press}
}

@inproceedings{raghunathan2014alternating,
  title={Alternating direction method of multipliers for strictly convex quadratic programs: Optimal parameter selection},
  author={Raghunathan, Arvind U and Di Cairano, Stefano},
  booktitle={2014 American Control Conference},
  pages={4324--4329},
  year={2014},
  organization={IEEE}
}

@article{ghadimi2014optimal,
  title={Optimal parameter selection for the alternating direction method of multipliers ({ADMM}): Quadratic problems},
  author={Ghadimi, Euhanna and Teixeira, Andr{\'e} and Shames, Iman and Johansson, Mikael},
  journal={IEEE Transactions on Automatic Control},
  volume={60},
  number={3},
  pages={644--658},
  year={2014},
  publisher={IEEE}
}

@article{wohlberg2017admm,
  title={{ADMM} penalty parameter selection by residual balancing},
  author={Wohlberg, Brendt},
  journal={arXiv preprint arXiv:1704.06209},
  year={2017}
}

@inproceedings{nishihara2015general,
  title={A general analysis of the convergence of ADMM},
  author={Nishihara, Robert and Lessard, Laurent and Recht, Ben and Packard, Andrew and Jordan, Michael},
  booktitle={International conference on machine learning},
  pages={343--352},
  year={2015},
  organization={PMLR}
}

@inproceedings{xu2017adaptive3,
  title={Adaptive relaxed {ADMM}: Convergence theory and practical implementation},
  author={Xu, Zheng and Figueiredo, Mario AT and Yuan, Xiaoming and Studer, Christoph and Goldstein, Tom},
  booktitle={Proceedings of the IEEE conference on computer vision and pattern recognition},
  pages={7389--7398},
  year={2017}
}

@article{lu2021linearized,
  title={Linearized {ADMM} converges to second-order stationary points for non-convex problems},
  author={Lu, Songtao and Lee, Jason D and Razaviyayn, Meisam and Hong, Mingyi},
  journal={IEEE Transactions on Signal Processing},
  volume={69},
  pages={4859--4874},
  year={2021},
  publisher={IEEE}
}

@article{zeng2024accelerated,
  title={An accelerated stochastic {ADMM} for nonconvex and nonsmooth finite-sum optimization},
  author={Zeng, Yuxuan and Wang, Zhiguo and Bai, Jianchao and Shen, Xiaojing},
  journal={Automatica},
  volume={163},
  pages={111554},
  year={2024},
  publisher={Elsevier}
}

@article{liu2025accelerated,
  title={An accelerated semi-proximal {ADMM} with applications to multi-block sparse optimization problems},
  author={Liu, Peng and Chen, Liang and Bai, Minru},
  journal={Journal of Scientific Computing},
  volume={104},
  number={1},
  pages={1--30},
  year={2025},
  publisher={Springer}
}

@article{li2019accelerated,
  title={Accelerated alternating direction method of multipliers: An optimal {$O(1/K)$} nonergodic analysis},
  author={Li, Huan and Lin, Zhouchen},
  journal={Journal of Scientific Computing},
  volume={79},
  pages={671--699},
  year={2019},
  publisher={Springer}
}

@article{he2023accelerated,
  title={Accelerated linearized alternating direction method of multipliers with {N}esterov extrapolation},
  author={He, Xin and Huang, Nan-Jing and Fang, Ya-Ping},
  journal={arXiv preprint arXiv:2310.16404},
  year={2023}
}

@article{xu2017accelerated,
  title={Accelerated first-order primal-dual proximal methods for linearly constrained composite convex programming},
  author={Xu, Yangyang},
  journal={SIAM Journal on Optimization},
  volume={27},
  number={3},
  pages={1459--1484},
  year={2017},
  publisher={SIAM}
}

@article{melo2017iteration,
  title={Iteration-complexity of a linearized proximal multiblock {ADMM} class for linearly constrained nonconvex optimization problems},
  author={Melo, Jefferson G and Monteiro, RD},
  journal={Available on: \url{http://www.optimization-online.org}},
  year={2017}
}

@article{yashtini2022convergence,
  title={Convergence and rate analysis of a proximal linearized {ADMM} for nonconvex nonsmooth optimization},
  author={Yashtini, Maryam},
  journal={Journal of Global Optimization},
  volume={84},
  number={4},
  pages={913--939},
  year={2022},
  publisher={Springer}
}

@inproceedings{Silberman:ECCV12,
  author    = {Nathan Silberman, Derek Hoiem, Pushmeet Kohli and Rob Fergus},
  title     = {Indoor Segmentation and Support Inference from RGBD Images},
  booktitle = {ECCV},
  year      = {2012}
}

@article{sun2025accelerating,
  title={Accelerating preconditioned {ADMM} via degenerate proximal point mappings},
  author={Sun, Defeng and Yuan, Yancheng and Zhang, Guojun and Zhao, Xinyuan},
  journal={SIAM Journal on Optimization},
  volume={35},
  number={2},
  pages={1165--1193},
  year={2025},
  publisher={SIAM}
}

@article{bauschke2023real,
  title={Real roots of real cubics and optimization},
  author={Bauschke, Heinz H and Lal, Manish Krishan and Wang, Xianfu},
  journal={Journal of Convex Analysis},
  volume={32},
  number={1},
  pages={119--144},  
  year={2025}
}

@article{yang2022proximal,
  title={Proximal {ADMM} for nonconvex and nonsmooth optimization},
  author={Yang, Yu and Jia, Qing-Shan and Xu, Zhanbo and Guan, Xiaohong and Spanos, Costas J},
  journal={Automatica},
  volume={146},
  pages={110551},
  year={2022},
  publisher={Elsevier}
}

@article{liu2019linearized,
  title={Linearized {ADMM} for nonconvex nonsmooth optimization with convergence analysis},
  author={Liu, Qinghua and Shen, Xinyue and Gu, Yuantao},
  journal={IEEE Access},
  volume={7},
  pages={76131--76144},
  year={2019},
  publisher={IEEE}
}

@article{wang2024adaptive,
  title={An adaptive linearized alternating direction multiplier method with a relaxation step for convex programming},
  author={Wang, Boran},
  journal={arXiv preprint arXiv:2404.17109},
  year={2024}
}

@article{lu2016fast,
  title={Fast proximal linearized alternating direction method of multiplier with parallel splitting},
  author={Lu, Canyi and Li, Huan and Lin, Zhouchen and Yan, Shuicheng},
  journal={Proceedings of the AAAI Conference on Artificial Intelligence},
  volume={30},
  number={1},
  year={2016}
}

@article{maia2024adaptive,
  title={An Adaptive Proximal {ADMM} for Nonconvex Linearly-Constrained Composite Programs},
  author={Maia, Leandro Farias and Gutman, David H and Monteiro, Renato DC and Silva, Gilson N},
  journal={arXiv preprint arXiv:2407.09927},
  year={2024}
}

@article{li2023simple,
  title={A simple uniformly optimal method without line search for convex optimization},
  author={Li, Tianjiao and Lan, Guanghui},
  journal={Mathematical programming},
  year={2025},
  pages={1--38},
}

@article{he2000alternating,
  title={Alternating direction method with self-adaptive penalty parameters for monotone variational inequalities},
  author={He, Bing-Sheng and Yang, Hai and Wang, SL},
  journal={Journal of Optimization Theory and applications},
  volume={106},
  pages={337--356},
  year={2000},
  publisher={Springer}
}

@article{lin2017extragradient,
  title={An extragradient-based alternating direction method for convex minimization},
  author={Lin, Tianyi and Ma, Shiqian and Zhang, Shuzhong},
  journal={Foundations of Computational Mathematics},
  volume={17},
  number={1},
  pages={35--59},
  year={2017},
  publisher={Springer}
}

@article{he2016convergence,
  title={Convergence study on the symmetric version of {ADMM} with larger step sizes},
  author={He, Bingsheng and Ma, Feng and Yuan, Xiaoming},
  journal={SIAM journal on imaging sciences},
  volume={9},
  number={3},
  pages={1467--1501},
  year={2016},
  publisher={SIAM}
}

@article{poliquin1996prox,
  title={Prox-regular functions in variational analysis},
  author={Poliquin, Ren{\'e} and Rockafellar, R},
  journal={Transactions of the American Mathematical Society},
  volume={348},
  number={5},
  pages={1805--1838},
  year={1996}
}

@article{nesterov2013gradient,
  title={Gradient methods for minimizing composite functions},
  author={Nesterov, Yu},
  journal={Mathematical programming},
  volume={140},
  number={1},
  pages={125--161},
  year={2013},
  publisher={Springer}
}

@article{yang2018modified,
  title={A modified projected gradient method for monotone variational inequalities},
  author={Yang, Jun and Liu, Hongwei},
  journal={Journal of Optimization Theory and Applications},
  volume={179},
  pages={197--211},
  year={2018},
  publisher={Springer}
}

@article{malitsky2020adaptive,
  title={Adaptive gradient descent without descent},
  author={Malitsky, Yura and Mishchenko, Konstantin},
  journal={International Conference on Machine
Learning},
  year={2020}
}

@article{condat2013primal,
  title={A primal--dual splitting method for convex optimization involving Lipschitzian, proximable and linear composite terms},
  author={Condat, Laurent},
  journal={Journal of optimization theory and applications},
  volume={158},
  number={2},
  pages={460--479},
  year={2013},
  publisher={Springer}
}

@article{chang2011libsvm,
  title={LIBSVM: a library for support vector machines},
  author={Chang, Chih-Chung and Lin, Chih-Jen},
  journal={ACM transactions on intelligent systems and technology (TIST)},
  volume={2},
  number={3},
  pages={1--27},
  year={2011},
  publisher={Acm New York, NY, USA}
}

@inproceedings{gu2014weighted,
  title={Weighted nuclear norm minimization with application to image denoising},
  author={Gu, Shuhang and Zhang, Lei and Zuo, Wangmeng and Feng, Xiangchu},
  booktitle={Proceedings of the IEEE conference on computer vision and pattern recognition},
  pages={2862--2869},
  year={2014}
}

@article{sun2026automating,
  title={Automating Reformulation for parallel {ADMM}},
  author={Sun, Kaizhao and Wu, Baihao and Yuan, Kun and Yin, Wotao},
  journal={Available on: \url{https://github.com/alibaba-damo-academy/PDMO.jl}},
  year={2026}
}
\end{document}